\documentclass[a4paper,11pt,reqno]{amsart}

\usepackage{lmodern}

\usepackage[utf8]{inputenc} 
\usepackage[T1]{fontenc}

\usepackage[english]{babel}

\usepackage{amsfonts}
\usepackage{amssymb}
\usepackage{amscd}
\usepackage{amsthm}
\usepackage{a4wide}
\usepackage{mathrsfs}
\usepackage{amsmath}
\usepackage{amssymb}
\usepackage{amsthm}
\usepackage{textcomp}
\usepackage{graphicx}
\usepackage{enumerate}
\usepackage{mathrsfs}
\usepackage{frcursive}
\usepackage{tikz}
\usepackage[cyr]{aeguill}
\usepackage{xspace}
\usepackage{hyperref}
\usepackage{appendix}
\usepackage{esint}
\usepackage{cancel}
\usepackage{calc}
\usepackage[colorinlistoftodos,prependcaption]{todonotes}

\newtheorem{defn}{Definition}[section]
\newtheorem{lemma}[defn]{Lemma}
\newtheorem{prop}[defn]{Proposition}
\newtheorem{theo}[defn]{Theorem}
\newtheorem{coro}[defn]{Corollary}

\newtheorem{claim}{Claim}
\newtheorem{rk}[defn]{Remark}

\def\RR{\mathbb{R}}
\def\Ric{\mathop{\rm Ric}\nolimits}

\def\Rm{\mathop{\rm Rm}\nolimits}
\def\tr{\mathop{\rm tr}\nolimits}

\def\vol{\mathop{\rm vol}\nolimits}

\def\vol{\mathop{\rm Vol}\nolimits}

\def\inj{\mathop{\rm inj}\nolimits}

\def\div{\mathop{\rm div}\nolimits}

\def\Li{\mathop{\rm \mathscr{L}}\nolimits}

\def\Ric{\mathop{\rm Ric}\nolimits}

\def\Rm{\mathop{\rm Rm}\nolimits}
\def\tr{\mathop{\rm tr}\nolimits}

\def\vol{\mathop{\rm vol}\nolimits}

\def\vol{\mathop{\rm Vol}\nolimits}

\def\inj{\mathop{\rm inj}\nolimits}

\def\div{\mathop{\rm div}\nolimits}

\def\Li{\mathop{\rm \mathscr{L}}\nolimits}

\def\supp{\mathop{\rm supp}\nolimits}

\def\R{\mathop{\rm R}\nolimits}

\newsavebox\CBox
\newcommand\hcancel[2][0.5pt]{%
  \ifmmode\sbox\CBox{$#2$}\else\sbox\CBox{#2}\fi%
  \makebox[0pt][l]{\usebox\CBox}%
  \rule[0.5\ht\CBox-#1/2]{\wd\CBox}{#1}}

\numberwithin{equation}{section}

\begin{document}
\title{Dynamical (In)Stability of Ricci-flat ALE metrics along the Ricci flow}
\date{today}
\author{Alix Deruelle}
\address{Institut de math\'ematiques de Jussieu, 4, place Jussieu, Boite Courrier 247 - 75252 Paris}
\email{alix.deruelle@imj-prg.fr}
\author{Tristan Ozuch}
\address{MIT, Dept. of Math., 77 Massachusetts Avenue, Cambridge, MA 02139-4307.}
\email{ozuch@mit.edu}
\maketitle

	\begin{abstract}
	    We study the stability and instability of ALE Ricci-flat metrics around which a \L{}ojasiewicz inequality is satisfied for a variation of Perelman's $\lambda$ functional adapted to the ALE situation and denoted $\lambda_{\operatorname{ALE}}$. This functional was introduced by the authors in a recent work and it has been proven that it satisfies a good enough \L{}ojasiewicz inequality at least in neighborhoods of integrable ALE Ricci-flat metrics in dimension larger than or equal to $5$. 
	\end{abstract}
	
	\markboth{Alix Deruelle and Tristan Ozuch}{Stability of Ricci-flat ALE metrics along the Ricci flow}
	\tableofcontents
	
	\section*{Introduction}
	
	The understanding of Ricci-flat ALE metrics is a central issue in Riemannian geometry. These spaces model the formation of singularities of spaces with Ricci curvature bounds \cite{and,Ban-Kas-Nak} as well as the singularities of $4$-dimensional Ricci flow with bounded scalar curvature \cite{bz}, \cite{Sim-Mil-Ext-4d}. They moreover appear as finite-time blow-up limits of some Ricci flows \cite{app-EH}. Their stability therefore becomes a crucial question for the Ricci flow. 
	
	Our goal here is to study the \emph{dynamical} stability and instability of these spaces along the Ricci flow thanks to a functional on suitable neighborhoods of any ALE Ricci-flat metrics which detects Ricci-flat metrics as its critical points.

	\subsection*{An adaptation of Perelman's $\lambda$-functional to the ALE setting}
	
	In \cite{Der-Ozu-Lam}, the authors introduced an adaptation to the ALE setting of the $\lambda$-functional from \cite{Per-Ent}, this functional is denoted $\lambda_{\operatorname{ALE}}$. More precisely, for any ALE metric $g$ of order $\tau >\frac{n-2}{2}$ (see Definition \ref{defn-ALE}) whose scalar curvature is integrable, we define 
	$$\lambda_{\operatorname{ALE}}(g):= \lambda^0_{\operatorname{ALE}}(g) - m_{\operatorname{ADM}}(g),$$
	where $ \lambda^0_{\operatorname{ALE}}(g) := \inf_w \int_N 4|\nabla^g w|_g^2 + \R_g w^2 $ where the infimum is taken among the smooth functions $w$ with $w-1$ compactly supported, and $m_{\operatorname{ADM}}$ is (up to a constant) the ADM mass of $(N,g)$. The point of substracting $m_{\operatorname{ADM}}(g)$ is that the functional now extends to an analytic function on the classical space of metrics satisfying, for $\tau>\frac{n-2}{2}$, $k\in\{0,1,2,3\}$ and a Ricci-flat metric $g_b$,  $(1+r)^{\tau+k}|\nabla^{g_b,k}(g-g_b)| \ll 1$ where $r$ is the $g_b$-distance to a given point. Neither $\lambda^0_{\operatorname{ALE}}$ or $m_{\operatorname{ADM}}$ can be defined on such a neighborhood without further constraints.
	
	In \cite{Der-Ozu-Lam}, we moreover showed that $\lambda_{\operatorname{ALE}}$ detects Ricci-flat ALE metrics as its only critical points, and the Ricci flow is moreover its gradient flow. 
	
	The second variation of $\lambda_{\operatorname{ALE}}$ at an ALE Ricci-flat metric $(N^n,g_b)$ along divergence-free variations is half the Lichnerowicz operator $L_{g_b}:=\Delta_{g_b}+2\Rm(g_b)\ast$. This leads us to define the linear stability of an ALE Ricci-flat metric $(N^n,g_b)$ as the non-positivity of the associated Lichnerowicz operator $L_{g_b}$ restricted to divergence-free variations. In the integrable case, i.e. in the case where the space of ALE Ricci-flat metrics in the neighborhood of a fixed ALE Ricci-flat metric is a smooth finite-dimensional manifold, we have a nice consequence of the linear stability: any linearly stable and integrable ALE Ricci-flat metric is a local maximum for the functional $\lambda_{\operatorname{ALE}}$. Properties of the ADM mass were deduced thanks to this fact in \cite{Der-Ozu-Lam}.

		A more delicate notion of stability is that of \emph{dynamical stability} of ALE Ricci-flat metrics along the Ricci flow. Its study on \emph{compact} Ricci flat manifolds has been investigated in \cite{Gue-Ise-Kno}, \cite{Ses-Lin-Dyn-Sta}, \cite{Has-Sta} and \cite{Has-Mul}. In the non-compact situation, there are several additional difficulties. A major difference is that $0$ is not isolated in the spectrum of the linearized operator. This prevents an exponential convergence rate as in the case of \cite{Has-Sta} of an integrable Ricci-flat metric on a closed manifold, one only gets a polynomial-in-time convergence.
	
	\subsection*{A weighted \L{}ojasiewicz inequality for $\lambda_{\operatorname{ALE}}$.}One tool that has been quite popular to study the dynamical stability of fixed points of  geometric evolution equations is the notion of \L{}ojasiewicz-Simon inequalities. Its name comes from both the classical work of \L{}ojasiewicz \cite{loj} on finite dimensional dynamical systems of gradient type and that of L. Simon \cite{sim} who extended systematically these inequalities to functionals defined on infinite dimensional spaces. The main geometric applications obtained in \cite{sim} concern the uniqueness of tangent cones of isolated singularities of minimal surfaces in Euclidean space together with the uniqueness of tangent maps of minimizing harmonic maps with values into an analytic closed Riemannian manifold. These geometric equations have the advantage to be strongly elliptic. Notice that all these results do not hold true if one drops the assumption on the analyticity of the data under consideration.
	
	In the compact setting, \L{}ojasiewicz inequalities have been proved for Perelman's $\lambda$-functional in the neighborhood of \emph{compact} Ricci-flat metrics and were the main tool to study the stability of Ricci-flat metrics in \cite{Has-Sta} in the integrable case, and in \cite{Has-Mul} in the general case.
	
The main result in \cite{Der-Ozu-Lam} is that the functional $\lambda_{\operatorname{ALE}}$ satisfies a weighted \L ojasiewicz inequality in a neighborhood of any ALE Ricci-flat metric with respect to the topology of weighted H\"older spaces $C^{2,\alpha}_{\tau}$, $\alpha\in(0,1)$, with polynomial decay of rate $\tau\in (\frac{n-2}{2},n-2)$ (see Definition \ref{def-weighted-norms}). It is \emph{weighted} since it uses the $L^2_{\frac{n}{2}+1}(g_b)$-norm which is essentially $L^2((1+r)dv_{g_b})$ instead of $L^2(dv_{g_b})$. Roughly speaking, the article \cite{Der-Ozu-Lam} proves that there is some $\theta\in(0,1]$ such that for any metric $g$ in a $C^{2,\alpha}_{\tau}$-neighborhood of a given ALE Ricci flat metric, the following \L ojasiewciz inequality holds: 
\begin{equation}
		|\lambda_{\operatorname{ALE}}(g)|^{2-\theta}\leq C\|\nabla \lambda_{\operatorname{ALE}}(g)\|_{L^2_{\frac{n}{2}+1}(g_b)}^{2}.\label{loj-ineq-lambda-ALE-intro}
		\end{equation}
We refer the reader to Theorem \ref{dream-thm-loja-intro} for a precise statement.

	The fact that our spaces are non-compact induces quite a lot of new difficulties. In particular, the spectrum of the Lichnerowicz operator is not discrete anymore and $0$ belongs to the essential spectrum. This explains the need of considering weighted Sobolev spaces different from $L^2$ for which the differential of the gradient $\nabla\lambda_{\operatorname{ALE}}$ at a Ricci-flat ALE metric is Fredholm. Theorem \ref{dream-thm-loja-intro} gives an optimal $L^2_{\frac{n}{2}+1}$-\L ojasiewicz inequality with exponent $\theta=1$ in the integrable situation. 
	Nonetheless, it seems that inequality appears not to be so useful regarding the study of dynamical stability of the Ricci flow near an ALE Ricci flat metric. This is essentially due to the $L^2$-variational structure of $\lambda_{\operatorname{ALE}}$. 
	For this reason, in the setting of $L^2$-perturbations, by interpolation, we obtain the following \L ojasiewicz inequality near any ALE Ricci flat linearly stable and integrable metric in dimension greater than $4$.
			In particular, if $n\geq 5$, one has the following $L^2$-\L ojasiewicz inequality for integrable Ricci-flat ALE metrics: for $\tau\in(\frac{n}{2},n-2)$ and $0<\delta<\frac{2\tau-(n-2)}{2\tau-(n-4)}$, there exists $C>0$ such that for all $g\in B_{C^{2,\alpha}_\tau}(g_b,\varepsilon)$, 
		\begin{equation}
		    |\lambda_{\operatorname{ALE}}(g)|^{2-\theta_{L^2}}\leq C \|\nabla \lambda_{\operatorname{ALE}}(g)\|_{L^2(g_b)}^{2}, \quad\theta_{L^2}:=2-\frac{1}{\delta} < \frac{2\tau-n}{2\tau-(n-2)}.\label{loja L2 intro}
		\end{equation}

	 Notice that we cannot reach the usual optimal $L^2$-\L{}ojasiewicz exponent $\theta_{L^2} = 1$. This is consistent with the known fact that the DeTurck-Ricci flow only converges polynomially fast for perturbations of the Euclidean space: see for instance \cite{Sch-Sch-Sim} and \cite{app-scal}. We also refer to the recent work \cite{Kro-Pet-Lp-stab} on the dynamical stability of integrable ALE Ricci flat metrics carrying a parallel spinor. Indeed, an exponent $\theta_{L^2} =1$ would imply that the convergence is exponential.
	
	
	\subsection*{Stability and instability of Ricci-flat ALE metrics.}
	
	In the present article, we study the \emph{dynamical} stability or instability of Ricci-flat ALE metrics along the Ricci flow assuming an $L^2$-\L{}ojasiewicz inequality such as \eqref{loja L2 intro} holds true. This should be a quite general scheme of proof and apply to other stability questions on \emph{non-compact spaces} along the Ricci flow and other parabolic geometric flows.
	
	\begin{theo}[Stability of Ricci-flat ALE metrics]\label{theo stability introp}
	    Let $(N^n,g_b)$ be a Ricci-flat ALE manifold of dimension $n\geqslant 4$ and $\frac{n-2}{2}<\tau<n-2$. There exists $\alpha\in(0,1)$ such that if we assume:
	    \begin{enumerate}
	        \item that $g_b$ is a local maximum of $\lambda_{\operatorname{ALE}}$ in the $C^{2,\alpha}_\tau(g_b)$ topology, 
	        \item in a $C^{2,\alpha}_\tau(g_b)$-neighborhood $B_{C^{2,\alpha}_{\tau}}(g_b,\varepsilon_{\L})$ of $g_b$, an $L^2$-\L{}ojasiewicz inequality is satisfied: for any metric $g$ in $B_{C^{2,\alpha}_{\tau}}(g_b,\varepsilon_{\L})$, we have
	        $$|\lambda_{\operatorname{ALE}}(g)|^{2-\theta}\leq C \|\nabla \lambda_{\operatorname{ALE}}(g)\|_{L^2(g_b)}^{2}$$
	        for some $\theta\in(0,1)$,
	    \end{enumerate}
	    then, for any $0<\tau'<\tau$ and $0<\alpha'<\alpha$, for any metric $g$ sufficiently $C^{2,\alpha}_\tau(g_b)$-close to $g_b$, the Ricci flow starting at $g$, $C^{2,\alpha'}_{\tau'}(g_b)$-converges to a Ricci-flat metric $g'_b$ (which is $C^{2,\alpha}_\tau(g_b)$-close to $g_b$) at a polynomial speed determined by the exponent $\theta$.
	\end{theo}
	
	Theorem \ref{theo stability introp} provides Type IIb solutions of the Ricci flow unless the background Ricci flat metric is the Euclidean metric, i.e. it ensures the existence of immortal solutions $(g(t))_{t\geq 0}$ satisfying $\limsup_{t\rightarrow+\infty}t\sup_N|\Rm(g(t))|=+\infty$. Notice also that Theorem \ref{theo stability introp} ensures the convergence to hold in weighted H\"older spaces which we believe are well-suited for gluing methods. We refer the reader to the article \cite{Bre-Kap} for such an illustration. Finally, we underline the need in Theorem \ref{theo stability introp} of restricting the convergence rate in space below the threshold value $n-2$: we refer the reader to our discussion on previous results of Dai and Ma \cite{Dai-Ma-Mass} right after the proof of Proposition \ref{prop-id-scal-int-mass-init} which links the mass and the mean value of the scalar curvature along the solutions provided by Theorem \ref{theo stability introp}.

	A direct consequence of Theorem \ref{dream-thm-loja-intro} and Theorem \ref{theo stability introp} is the following stability result for integrable ALE Ricci-flat spaces of dimension at least $5$.
	
	\begin{coro}\label{coro-int-sta-dyn-sta}
	    Let $(N,g_b)$ be a Ricci-flat ALE manifold of dimension $n\geqslant 5$ with integrable Ricci-flat deformations. Assume that $g_b$ is \emph{linearly stable}, that is: its Lichnerowicz operator is nonpositive on divergence-free deformations decaying at infinity. Let $\frac{n-2}{2}<\tau<n-2$ and $0<\alpha<1$ sufficiently small.
	    
	    Then, for any $0<\tau'<\tau$ and $0<\alpha'<\alpha$, for any metric $g$ sufficiently $C^{2,\alpha}_\tau(g_b)$-close to $g_b$, the Ricci flow starting at $g$ $C^{2,\alpha'}_{\tau'}(g_b)$-converges to a Ricci-flat metric $g'_b$ (which is $C^{2,\alpha}_\tau(g_b)$-close to $g_b$) at a polynomial speed determined by the dimension $n$.
	\end{coro}
	
	Notice that Corollary \ref{coro-int-sta-dyn-sta} applies to all known Ricci-flat ALE metrics in dimension greater than or equal to $5$. For instance, it applies to Calabi's ALE Ricci-flat K\"ahler metrics \cite{Cal-Fib-Hol} on the total space of the line bundle $L^{-n}\rightarrow\mathbb{CP}^{n-1}$, $n\geq 3$, and more generally to Joyce's ALE Ricci-flat K\"ahler metrics \cite[Chapter $8$]{Joy-Book}. In particular, this makes precise the results of \cite{Chau-Tam-Stab-I} in the ALE case, where the convergence is established on compact subsets only.

	 Corollary \ref{coro-int-sta-dyn-sta} shares some similarities with the work of Kr\"oncke and Petersen \cite{Kro-Pet-Lp-stab} that we now explain. In \cite{Kro-Pet-Lp-stab}, the authors investigates the dynamical stability of integrable ALE Ricci-flat metrics which carry a parallel spinor in dimension greater than or equal to $4$. As they notice, this applies to all know ALE Ricci-flat metrics. Their method is based on a delicate analysis of the heat kernel of the Lichnerowicz operator in $L^p$ spaces and they apply it to the DeTurck-Ricci flow with time-dependent background metrics. We emphasize that the convergence result they get takes place in the $C^k\cap L^p$ topology and the corresponding polynomial convergence rate is sharp whereas Theorem \ref{theo stability introp} (and Corollary \ref{coro-int-sta-dyn-sta}) proves the convergence of the Ricci flow directly and in weighted spaces at the cost of getting an a priori non sharp convergence rate. Moreover, they assume the initial condition to be a perturbation of the ALE Ricci flat background to lie in $L^p\cap L^{\infty}$ for any $p<n$. In particular, they allow initial conditions to decay like $r^{-1-\epsilon}$ at infinity while we require a decay rate to be at least $r^{-(n-2)/2}$. Another interesting technical fact is that \cite{Kro-Pet-Lp-stab} considers nearby Ricci-flat metrics of a given ALE Ricci flat metric with respect to the Bianchi gauge whereas we study such metrics which are divergence free with respect to a background ALE Ricci-flat metric. In the setting of Theorem \ref{theo stability introp}, we prove that the Bianchi gauge converges to $0$ faster then expected: see Section \ref{sec-Bia-for-evo} for a precise statement. Therefore, in view of the previous remarks, we ask whether our methods can be carried to dimension $4$.

	The next result echoes the work \cite{Has-Mul} on the existence of ancient solutions coming out of an unstable Ricci-flat metric:
	
	\begin{theo}[Instability of Ricci-flat ALE metrics]\label{theo instability introp}
	    Let $(N^n,g_b)$ be a Ricci-flat ALE manifold of dimension $n\geqslant 4$ and assume for some $0<\alpha<1$ and $\tau>\frac{n-2}{2}$:
	    \begin{enumerate}
	        \item that $g_b$ is a \emph{not} a local maximum of $\lambda_{\operatorname{ALE}}$ in the $C^{2,\alpha}_\tau(g_b)$ topology,
	        \item in a $C^{2,\alpha}_\tau(g_b)$-neighborhood $\mathcal{U}_{g_b}$ of $g_b$, an $L^2$-\L{}ojasiewicz inequality is satisfied: for any metric $g$ in $\mathcal{U}_{g_b}$, we have
	        $$|\lambda_{\operatorname{ALE}}(g)|^{2-\theta}\leq C \|\nabla \lambda_{\operatorname{ALE}}(g)\|_{L^2(g_b)}^{2}$$
	        for some $\theta\in(0,1)$.
	    \end{enumerate}
	    Then, there exists a non Ricci-flat ancient solution to the Ricci flow $(g(t))_{t\in(-\infty,0]}$
	    which is uniformly ALE of order $\tau$ and $C^{2,\alpha}_\tau$-converges to $g_b$ at a polynomial speed determined by the exponent $\theta$.
	\end{theo}
	
	It is worth mentioning that it is still an open problem whether there are unstable ALE Ricci flat metrics in dimension greater than or equal to $4$. Also, Theorem \ref{theo instability introp} bears some resemblance with the work \cite{Takahashi} where an ancient solution coming out of the Euclidean Schwarzschild metric is constructed by hand. It would be an interesting problem to recover Takahashi's result via a suitable ALF version of our functional $\lambda_{\operatorname{ALE}}$.
	
	Both Theorem \ref{theo stability introp} and Theorem \ref{theo instability introp} rely heavily on Gaussian estimates and estimates in weighted Hölder spaces for the heat kernel which we believe are of independent interest and that we describe now.
\subsection*{Heat kernel estimates}
	
	We prove that the heat kernel associated to the forward heat equation acting on functions along a Ricci flow in a suitable neighborhood of a Ricci-flat ALE metric satisfies uniform-in-time Gaussian bounds in Theorem \ref{Gaussian-est-scal-heat-ker}. The proof follows the method of Grigor'yan \cite{Gri-Upp-Ene-Dav} and that of Zhang \cite{Zha-Qi-Gra-Est} in a Ricci flow setting.

	These controls are then used on the parabolic equation satisfied by the Ricci curvature and integrated in time to yield controls on the metric in $C^{2,\alpha}_\tau$ for $\frac{n-2}{2}<\tau<n-2$ as long as the Ricci flow stays in a given $C^{2,\alpha}_\tau$-neighborhood of a Ricci-flat metric. Indeed, lemma \ref{lemma-short-time-est-curv} starts with establishing short-time estimates for weighted H\"older norm of the curvature operator and the Ricci tensor. Its proof already reveals a constraint on the H\"older exponent $\alpha\in(0,1)$ in terms of the weight $\tau$. Lemma \ref{lemma-a-priori-wei-metr-wei} takes care of estimating the distance of a Ricci flow satisfying mild assumptions from an ALE Ricci flat metric in the weighted $C_{\tau}^0$ norm a priori. Then Lemma \ref{a-priori-wei-lemma-C0-Ric} proves an a priori bound for the weighted $C^{0,\alpha}_{\tau+2}$ of the Ricci tensor. Finally, Lemma \ref{a-priori-wei-lemma-Holder-met} establishes an a priori bound for the distance of a Ricci flow from an ALE Ricci flat metric in the full weighted H\"older $C_{\tau+2}^{2,\alpha}$ norm in terms of the weighted norm of the corresponding Ricci tensor. As expected, it is more tractable to bound the Ricci tensor than the metric itself along such a Ricci flow. 

	\subsection*{Outline of paper}
	
	In Section \ref{section properties}, we give the main basic definitions of the article and review the properties of the functional $\lambda_{\operatorname{ALE}}$ proven in \cite{Der-Ozu-Lam}.
	
	In Section \ref{section heat kernel}, we prove Gaussian bounds for the heat kernel along a Ricci flow in a $C^{2,\alpha}_\tau$-neighborhood of a Ricci-flat ALE metric and deduce controls on the flow in suitable weighted Hölder spaces.
	
	In Section \ref{section stability}, using the previous controls on the flow thanks to the heat kernel, we prove our stability result for metrics which are local maxima of $\lambda_{\operatorname{ALE}}$ and around which a suitable $L^2$-\L{}ojasiewicz inequality holds.
	
	Finally, in Section \ref{section instability}, we study the case when a metric is \emph{not} a local maximum of the functional $\lambda_{\operatorname{ALE}}$ and prove that if a suitable $L^2$-\L{}ojasiewicz inequality holds around it, then, there exists an ancient nontrivial Ricci flow coming out of it.
	
	In the Appendix, Section \ref{section appendix}, we recall some formulas for the first and second variations of the geometric quantities of interest here.
	
	\subsection*{Acknowledgements}
	We thank Klaus Kr\"oncke for his suggestions and comments on a preliminary version of this paper.
	
The first author is supported by grant ANR-17-CE40-0034 of the French National Research Agency ANR (Project CCEM).	

\section{The functional $\lambda_{\operatorname{ALE}}$ and its properties}\label{section properties}

In this section, we recall some of the properties of the functional $\lambda_{\operatorname{ALE}}$ introduced and studied in \cite{Der-Ozu-Lam}. This functional and more precisely its sign in neighborhoods of Ricci-flat ALE metrics (see Definition \ref{defn-ALE}) will determine the dynamical stability or instability of these Ricci-flat ALE metrics along the Ricci flow.

\subsection{First definitions and function spaces}~~\\

We start by defining the manifolds as well as the function spaces we will be interested in.

\begin{defn}[Asymptotically locally Euclidean (ALE) manifolds]\label{defn-ALE}
	We will call a Riemannian manifold $(N,g)$ \emph{asymptotically locally Euclidean} (ALE) of order $\tau>0$ if the following holds: there exists a compact set $K\subset N$, a radius $R>0$, $\Gamma$ a subgroup of $SO(n)$ acting freely on $\mathbb{S}^{n-1}$ and a diffeomorphism $\Phi : (\RR^n\slash\Gamma )\backslash B_e(0,R)\mapsto N\backslash K$ such that, if we denote $g_e$ the Euclidean metric on $\mathbb{R}^n\slash\Gamma$, we have, for all $k\in \mathbb{N}$,
	$$ \rho^k\big|\nabla^{g_e,k}(\Phi^*g-g_e)\big|_e = O(\rho^{-\tau}),$$
	on $\big(\RR^n\slash\Gamma\big) \backslash B_e(0,R)$, where $\rho = d_e(.,0)$.
\end{defn}

We will study ALE metrics in a neighborhood of a Ricci flat ALE metric. Let us start by defining this neighborhood thanks to weighted norms :

\begin{defn}[Weighted Hölder norms for ALE spaces]\label{def-weighted-norms}
	Let $(N,g,p)$ be an ALE manifold of dimension $n$, $\beta>0$, and $\rho_g(x):= \max\big(d_g(x,p),1\big)$. For any tensor $s$, we define the following weighted $C^{k,\alpha}_\beta$-norm :
	$$\| s \|^g_{C^{k,\alpha}_\beta(N)} := \sup_{N}\rho_g^\beta\Big( \sum_{i=0}^k\rho_g^{i}|\nabla^{g,i} s|_{g} + \rho_g^{k+\alpha}[\nabla^{g,k}s]_{C^\alpha}\Big).$$
\end{defn}

We will make a constant use of the following compact embedding whose proof can be found for instance in \cite[Lemme 3]{Chal-Cho-Bru}:
\begin{lemma}\label{lemme-Chal-Cho-Bru}
Let $(N^n,g)$ be an ALE manifold of dimension $n$. If $k\geq 0$, $\tau>0$ and $\alpha\in(0,1)$ then the continuous embeddings $C^{k,\alpha}_{\tau}\hookrightarrow C^{k,\alpha'}_{\tau'}$ are compact for $\alpha'\in(0,\alpha)$ and $\tau'\in(0,\tau)$.
\end{lemma}

The next definition concerns weighted Sobolev norms for ALE spaces:

\begin{defn}[Weighted Sobolev norms for ALE spaces]\label{def-weighted-sobolev-norms}
	Let $\beta>0$, and $(N,g,p)$ an ALE manifold of dimension $n$, and $\rho_g(x):= \max\big(d_g(x,p),1\big)$. For any tensor $s$, we define the following weighted $L_\beta^2$-norm :
	$$\| s \|_{L^{2}_{\beta}} ^g:= \Big(\int_N |s|^2 \rho_{g}^{2\beta-n}dv_{g}\Big)^\frac{1}{2}.$$
	We moreover define the $H^k_{\beta}$-norm of $s$ as 
	$$\|s\|_{H^k_\beta}^g:= \sum_{i= 0}^k {\|\nabla^i s \|^g_{L^{2}_{\beta+i}}}.$$	
\end{defn}

\begin{rk}
	Since the above definitions do not formally depend on the type of tensor, and since the metrics we will consider will be equivalent, we will often abusively simply denote these spaces $C^{k,\alpha}_\tau$ or $H^k_\beta$ for instance.
\end{rk}
\subsection{The functionals $\lambda_{\operatorname{ALE}}^0$, $m_{\operatorname{ADM}}$ and $\lambda_{\operatorname{ALE}}$}~~\\

Next, as in \cite{Dai-Ma-Mass,Lee-Parker,Bart-Mass} let us consider for $\tau>\frac{n-2}{2}$ and $\alpha\in(0,1)$, the following classical space of metrics,
\begin{equation}
\mathcal{M}_\tau:= \left\{\text{$g$ is a metric on $N$}\,|\, g-g_b\in C^{1,\alpha}_{\tau}(S^2T^*N)\,|\,\R_g \in L^1\right\},
\end{equation}

\subsubsection{Mass of ALE metrics}
On $\mathcal{M}_\tau$, the mass of an ALE metric is well-defined and only depends on the metric:
\begin{equation}
m_{\operatorname{ADM}}(g):=\lim_{R\rightarrow+\infty}\int_{\{\rho_{g_b}=R\}}\left<\div_{g_b}(g-g_b)-\nabla^{g_b}\tr_{g_b}(g-g_b),\textbf{n}\right>_{g_b}\,d\sigma_{g_b},\label{def-mass}
\end{equation}
where $\textbf{n}$ denotes the outward unit normal of the closed smooth hypersurfaces $\{\rho_{g_b}=R\}$ for $R$ large.

\begin{rk}
	Given $g_b$ a Ricci-flat metric, it is clear from the definition that the map $h\mapsto m_{\operatorname{ADM}}(g_b+h)$ is linear.
\end{rk}

\subsubsection{The functional $\lambda_{\operatorname{ALE}}^0$}\label{subsec-lambda}

\begin{defn}[$\lambda_{\operatorname{ALE}}^0$, a first renormalized Perelman's functional]
	Let $(N^n,g_b)$ be an ALE Ricci flat metric and let $g\in \mathcal{M}^{2,\alpha}_{\tau}(g_b,\varepsilon)$. Define the $\mathcal{F}_{\operatorname{ALE}}$-energy by:
	\begin{eqnarray}
	\mathcal{F}_{\operatorname{ALE}}(w,g):=\int_N\big(4|\nabla^g w|_g^2 +\R_g w^2 \big)\,d\mu_g, 
	\end{eqnarray}
	where $w-1\in C^{\infty}_c(N)$, where $C^{\infty}_c(N)$ is the space of compactly supported smooth functions.
	The $\lambda_{\operatorname{ALE}}^0$-functional associated to the $\mathcal{F}_{\operatorname{ALE}}$-energy is:
	$$\lambda_{\operatorname{ALE}}^0(g) := \inf_w \mathcal{F}_{\operatorname{ALE}}(w,g),$$
	where the infimum is taken over functions $w:N\rightarrow \RR$ such that $w-1\in C_c^\infty(N)$.
\end{defn}

We prove in \cite{Der-Ozu-Lam} that the above infimum is attained by the (unique) solution $w_g\in 1+C^{2,\alpha}_\tau$ to the equation 
\begin{equation}
-4\Delta_g w_g + \R_g w_g =0.\label{eqdefwg}
\end{equation}
It turns out that $w_g$ is a positive function which lets us to consider the potential function associated to a metric $g$ defined as
$$f_g:=-2\ln w_g,$$
which is the unique solution $f_g\in C^{2,\alpha}_\tau$ to
\begin{equation}
2\Delta_gf_g-|\nabla^gf_g|^2_g+\R_g =0.\label{eqdeffg}
\end{equation}

\subsubsection{Definition of the functional $\lambda_{\operatorname{ALE}}$}\label{sec-def-lambda}~~\\

The above functionals $m_{\operatorname{ADM}}$ and $\lambda_{\operatorname{ALE}}^0$ are only well-defined when the scalar curvature is integrable which is not a convenient assumption as for instance $\mathcal{M}_\tau$ is not a closed subset of $C^{2,\alpha}_\tau$ for $ \frac{n-2}{2}<\tau<n-2$. Moreover, $m_{\operatorname{ADM}}$ and $\lambda_{\operatorname{ALE}}^0$ are not continuous with respect to the $C^{2,\alpha}_\tau$-topology. 
\begin{rk}
	Another finer topology for $\mathcal{M}_\tau$ is obtained by adding the $L^1$-norm of the scalar curvature. We will see in the rest of the present article that there are Ricci flows of ALE metric with nonvanishing (or even infinite) mass $C^{2,\alpha}_\tau$-converging to a Ricci-flat ALE metric. This implies that the scalar curvature does not converge in an $L^1$-sense by \cite[Corollary 3]{Dai-Ma-Mass}.
\end{rk}

In order to solve these problems, we define for $g$ in a small $C^{2,\alpha}_\tau$-neighborhood of a Ricci-flat ALE metric $g_b$ and in $\mathcal{M}_\tau$, the functional
$$\lambda_{\operatorname{ALE}}(g) := \lambda_{\operatorname{ALE}}^0(g)-m_{\operatorname{ADM}}(g).$$
The advantage is that $\lambda_{\operatorname{ALE}}$ extends as an analytic function on a $C^{2,\alpha}_\tau$-neighborhood of any Ricci-flat ALE metric, see \cite[Proposition 3.4]{Der-Ozu-Lam}. Moreover, by denoting $w_g$ the solution to \eqref{eqdefwg}, which is well-defined without assuming that the scalar curvature is integrable, we have
\begin{equation}
\begin{split}
\lambda_{\operatorname{ALE}}(g)=\lim_{R\rightarrow+\infty}\Bigg(\int_{\{\rho_{g_b}\leq R\}}&\left(|\nabla^{g}f_{g}|^2_{g}+\R_{g}\right)\,e^{-f_{g}}d\mu_{g}\\
&-\int_{\{\rho_{g_b}=R\}}\left<\div_{g_b}(g)-\nabla^{g_b}\tr_{g_b}(g),\textbf{n}_{g_b}\right>_{g_b}\,d\sigma_{g_b}\Bigg).\label{true-def-lambda}
\end{split}
\end{equation}
It is worth noting from \cite[Example 3.1]{Der-Ozu-Lam} that for most perturbations, none of the above integrals converge as $R\to \infty$, but their difference always does if $g$ is sufficiently $C^{2,\alpha}_{\tau}$-close to $g_b$ for $\tau>\frac{n}{2}-1$.

\subsection{Main properties of $\lambda_{\operatorname{ALE}}$}~~\\

We now list some of the properties of $\lambda_{\operatorname{ALE}}$ proven in \cite{Der-Ozu-Lam}.

\subsubsection{Variations of $\lambda_{\operatorname{ALE}}$}~~\\

We have the following first and second variations for $\lambda_{\operatorname{ALE}}$.

\begin{prop}[First variation of $\lambda_{\operatorname{ALE}}$]\label{first-var-prop} Let $(N^n,g_b)$ be an ALE Ricci flat metric asymptotic to $\RR^n\slash\Gamma$, for some finite subgroup $\Gamma$ of $SO(n)$ acting freely on $\mathbb{S}^{n-1}$, and let $\tau\in(\frac{n-2}{2},n-2)$. 
	
	The first variation of $\lambda_{\operatorname{ALE}}$ on a neighborhood of $B_{C^{2,\alpha}_{\tau}}(g_b,\varepsilon)$ at $g$ in the direction $h$ is:
	\begin{equation}
	\begin{split}
	\delta_g \lambda_{\operatorname{ALE}}(h)=&-\int_N\langle h,\Ric(g)+\nabla^{g,2}f_g\rangle_g \,e^{-f_g}d\mu_g.\label{first-var-lambda}
	\end{split}
	\end{equation}
	and the tensor $\Ric(g)+\nabla^{g,2}f_g$ is weighted divergence free, i.e. 
	\begin{equation}
	\div_{f_g}\left(\Ric(g)+\nabla^{g,2}f_g\right)=0,\label{wei-div-free-obs-ten}
	\end{equation}
	where $\div_{f_g} T = \div_gT-T(\nabla^g f_g)$ for a symmetric $2$-tensor $T$.
	
	Moreover, if $(g(t))_{t\in[0,T)}$ is a solution to the Ricci flow on $N$ lying in $B_{C^{2,\alpha}_{\tau}}(g_b,\varepsilon)$, then we have the following monotonicity formula,
\begin{equation}
\frac{d}{dt}\lambda_{\operatorname{ALE}}(g(t))=2\|\Ric(g(t))+\nabla^{g(t),2}f_{g(t)}\|_{L^2(e^{-f_{g(t)}}d\mu_{g(t)})}^2\geq 0.\label{mono-lambda}
\end{equation}

\end{prop}
We next consider the second variations of $\lambda_{\operatorname{ALE}}$.
\begin{defn}\label{defn-Lic-Op}
	Let $(N,g)$ be a Riemmanian metric. Then the \emph{Lichnerowicz operator} associated to $g$ acting on symmetric $2$-tensors, denoted by $L_{g}$, is defined by:
	\begin{equation}
	L_{g}h:=\Delta_{g}h + 2\Rm(g)(h)-\Ric(g)\circ h-h\circ\Ric(g),\quad h\in C_{loc}^2(S^2T^*N),\label{defn-Lic-op}
	\end{equation}
	where $\Delta_{g}=-\nabla_g^*\nabla_g$ and where $\Rm(g)(h)(X,Y) := h(\Rm(g)(e_i,X)Y,e_i)$ for an orthonormal basis $(e_i)_{i=1}^n$ with respect to $g$. In particular, if $(N^n,g_b)$ is an ALE Ricci flat metric, then,
	\begin{equation}
	L_{g}h:=\Delta_{g}h + 2\Rm(g)(h),\quad h\in C_{loc}^2(S^2T^*N).\label{defn-Lic-op}
	\end{equation}
\end{defn}

\begin{prop}[Second variation of $\lambda_{\operatorname{ALE}}$ at a Ricci flat metric]\label{second-var-prop}
	Let $(N^n,g_b)$ be an ALE Ricci flat metric asymptotic to $\RR^n\slash\Gamma$, for some finite subgroup $\Gamma$ of $SO(n)$ acting freely on $\mathbb{S}^{n-1}$, and let $\tau\in(\frac{n-2}{2},n-2)$. Then the second variation of $\lambda_{\operatorname{ALE}}$ at $g_b$ along a divergence free variation $h\in S^2T^*N$ such that $h\in B_{C^{2,\alpha}_{\tau}}(g_b,\varepsilon)$ is:
	\begin{equation}
	\delta^2_{g_b}\lambda_{\operatorname{ALE}}(h,h)=\frac{1}{2}\langle L_{g_b}h,h\rangle_{L^2}.
	\end{equation}
\end{prop}

\subsubsection{A \L{}ojasiewicz inequality}~~\\

As explained in the Introduction of this paper, in \cite{Der-Ozu-Lam}, we moreover proved that in a $C^{2,\alpha}_\tau$-neighborhood of any Ricci-flat ALE space, a \L{}ojasiewicz inequality holds for $\lambda_{\operatorname{ALE}}$.

\begin{theo}[{\cite{Der-Ozu-Lam}}]\label{dream-thm-loja-intro}
		Let $(N^n,g_b)$ be an ALE Ricci-flat manifold of dimension $n\geq 4$. Let $\alpha\in(0,1)$ and $\tau\in(\frac{n-2}{2},n-2)$. Then there exist a neighborhood $B_{C^{2,\alpha}_{\tau}}(g_b,\varepsilon)$ of $g_b$, a constant $C>0$ and $\theta\in (0,1]$ such that for any metric $g\in B_{C^{2,\alpha}_{\tau}}(g_b,\varepsilon)$, we have the following $L^2_{\frac{n}{2}+1}$-\L ojasiewicz inequality,
		\begin{equation}
		|\lambda_{\operatorname{ALE}}(g)|^{2-\theta}\leq C\|\nabla \lambda_{\operatorname{ALE}}(g)\|_{L^2_{\frac{n}{2}+1}(g_b)}^{2}.\label{loj-ineq-lambda-ALE texte}
		\end{equation}
		Moreover, if $(N^n,g_b)$ has integrable infinitesimal Ricci-flat deformations, then $\theta=1$.
		
		In particular, if $n\geq 5$, one has the following $L^2$-\L ojasiewicz inequality for integrable Ricci-flat ALE metrics: if $\tau\in(\frac{n}{2},n-2)$ then for any $0<\delta<\frac{2\tau-(n-2)}{2\tau-(n-4)}$, there exists $C>0$ such that for all $g\in B_{C^{2,\alpha}_\tau}(g_b,\varepsilon)$, 
		\begin{equation}
		    |\lambda_{\operatorname{ALE}}(g)|^{2-\theta_{L^2}}\leq C \|\nabla \lambda_{\operatorname{ALE}}(g)\|_{L^2(g_b)}^{2}, \quad\theta_{L^2}:=2-\frac{1}{\delta} = \frac{2\tau-n}{2\tau-(n-2)}.\label{loja L2 texte}
		\end{equation}
	\end{theo}
	
Here $\nabla \lambda_{\operatorname{ALE}}$ denotes the gradient of $\lambda_{\operatorname{ALE}}$ in the $L^2(g)$ sense.

\section{Preliminary estimates for short and large time}\label{section heat kernel}

\subsection{Heat kernel Gaussian bounds}\label{sec-heat-ker}~~\\

In this section, we fix an ALE Ricci flat metric $(N^n,g_b)$ once and for all. Let $(g(t))_{t\in[0,T)}$ be a solution to the Ricci flow with $g(0)\in B_{C^{2,\alpha}_{\tau}}(g_b,\varepsilon)$ with $\tau\in (\frac{n-2}{2},n-2)$.
The main aim of this section is to establish Gaussian bounds for the heat kernel acting on functions along the Ricci flow.

As explained in the Introduction, we follow the same strategy adopted by Grigor'yan \cite{Gri-Upp-Ene-Dav} and Zhang \cite{Zha-Qi-Gra-Est}.\\

 We denote the heat kernel associated to the forward heat equation by $K(x,t,y,s)$, for $0\leq s<t<T$, and $x,y\in N$: 
\begin{equation}
\label{scal-heat-equ-ker}
\left\{
\begin{array}{rl}
&\partial_tK(\cdot,\cdot,y,s)=\Delta_{g(t)}K(\cdot,\cdot,y,s),\\
&\\
&\partial_tg=-2\Ric(g(t)),\\
&\\
&\lim_{t\rightarrow s^{+}}K(\cdot,t,y,s)=\delta_y,
\end{array}
\right.
\end{equation}
on $N\times(0,T)$.
This heat kernel always exists and is positive: see \cite{Gue-Fun-Sol}.

We also consider its conjugate heat equation: if $(x,t)\in N\times(0,T)$ is fixed, then $K(x,t,\cdot,\cdot)$ is the heat kernel associated to the conjugate backward heat equation: 
\begin{equation}
\label{scal-heat-equ-conj-ker}
\left\{
\begin{array}{rl}
&\partial_sK(x,t,\cdot,\cdot)=-\Delta_{g(s)}K(x,t,\cdot,\cdot)+\R_{g(s)}K(x,t,\cdot,\cdot),\\
&\\
&\partial_sg=-2\Ric(g(s)),\\
&\\
&\lim_{s\rightarrow t^{-}}K(x,t,\cdot,s)=\delta_x.
\end{array}
\right.
\end{equation}
on $N\times (0,t)$.

We start with the following proposition that estimates the $L^1$ norms of both forward and backward heat kernels:

\begin{prop}[$L^1$-bound]\label{L^1-bound-heat-kernel-fct}
With the setting and notations introduced above, if $0\leq s\leq t<T$,
\begin{eqnarray}
e^{-\int_s^t\|\R_{g(t')}\|_{C^0}\,dt'}&\leq&\int_NK(x,t,y,s)\,d\mu_{g(t)}(x)\leq e^{\int_s^t\|\R_{g(t')}\|_{C^0}\,dt'},\quad  y\in N,\label{L^1-est-fct}\\
 1&=&\int_NK(x,t,y,s)\,d\mu_{g(s)}(y),\quad \quad x\in N.\label{L^infty-est-fct}
\end{eqnarray}

\end{prop}

\begin{proof}
Let $(\Omega_j)_{j\geq 0}$ be an increasing sequence of domains of $N$ with smooth boundary exhausting the manifold $N$. Let $K_j(\cdot,\cdot,y,s)$ be the heat kernel associated to (\ref{scal-heat-equ-ker}) on $\Omega_j$ with Dirichlet boundary condition. Then, one can prove that $(K_j(\cdot,\cdot,y,s))_{j\geq 0}$ is an increasing sequence converging locally uniformly to $K(\cdot,\cdot,y,s)$. By integrating by parts, one gets, for some fixed $(y,s)\in N\times [0,T)$,
\begin{eqnarray*}
\partial_t\int_{\Omega_j}K_j(x,t,y,s)d\mu_{g(t)}(x)&=&\int_{\Omega_j}\Delta_{g(t)}K_j(x,t,y,s)-\R_{g(t)}(x)K_j(x,t,y,s)\,d\mu_{g(t)}(x)\\
&=&\int_{\partial\Omega_j}\left<\nabla^{g(t)}_{x}K_j(x,t,y,s),\mathbf{n}\right>\,d\sigma_{j,g(t)}(x)\\
&&-\int_{\Omega_j}\R_{g(t)}(x)K_j(x,t,y,s)\,d\mu_{g(t)}(x)\\
&\leq&-\int_{\Omega_j}\R_{g(t)}(x)K_j(x,t,y,s)\,d\mu_{g(t)}(x)\\
&\leq&\|\R_{g(t)}\|_{C^0}\int_{\Omega_j}K_j(x,t,y,s)\,d\mu_{g(t)}(x),\\
\end{eqnarray*}
where $d\sigma_{j,g(t)}$ is the induced measure on $\partial \Omega_j$ by $d\mu_{g(t)}$. Therefore, by Gr\"onwall's inequality:
\begin{equation*}
\begin{split}
\int_{\Omega_j}K_j(x,t,y,s)\,d\mu_{g(t)}(x)&\leq e^{\int_{t'}^t\|\R_{g(s)}\|_{C^0}ds}\int_{\Omega_j}K_j(x,t',y,s)\,d\mu_{g(t')}(x),\\
\end{split}
\end{equation*}
for any $t\geq t'>s$. Let $j$ large enough so that $y\in \Omega_j$ and let $t'$ go to $s$ to get:
\begin{equation*}
\begin{split}
\int_{\Omega_j}K_j(x,t,y,s)\,d\mu_{g(t)}(x)&\leq  e^{\int_{s}^t\|\R_{g(t')}\|_{C^0}\,dt'},
\end{split}
\end{equation*}
for any $t>s$.
 By letting $j$ tending to $+\infty$, we get a half of the first estimate (\ref{L^1-est-fct}), i.e.
  \begin{equation}
\begin{split}\label{half-est-L1}
\int_{N}K(x,t,y,s)\,d\mu_{g(t)}(x)&\leq  e^{\int_{s}^t\|\R_{g(t')}\|_{C^0}\,dt'}.
\end{split}
\end{equation}

On the other hand, let $(\phi_k)_k$ be any sequence of smooth cut-off functions approximating the constant function $1$ with values into $[0,1]$ such that $\lim_{k\rightarrow+\infty}\|\nabla^{g(t'),k}\phi_k\|_{C^0}=0$ for $k=1,2$ and uniformly in $t'\in[s,t]$. Then
\begin{equation*}
\begin{split}
\partial_t\int_N\phi_k(x)K(x,t,y,s)\,d\mu_{g(t)}(x)&=-\int_N\Delta_{g(t)}\phi_k(x)K(x,t,y,s)\,d\mu_{g(t)}(x)\\
&-\int_N\phi_k(x)\R_{g(t)}(x)K(x,t,y,s)\,d\mu_{g(t)}(x)\\
&\geq -e^{\int_{s}^t\|\R_{g(t')}\|_{C^0}\,dt'}\sup_{\supp(\phi_k)}\arrowvert\Delta_{g(t)}\phi_k\arrowvert\\
&-\|R_{g(t)}\|_{C^0}\int_N\phi_k(x)K(x,t,y,s)\,d\mu_{g(t)}(x).
\end{split}
\end{equation*}
Here, we have used (\ref{half-est-L1}) in the second inequality. Again, by Gr\"onwall's inequality:
\begin{equation*}
\begin{split}
\int_N\phi_k(x)K(x,t,y,s)\,d\mu_{g(t)}(x)\geq e^{-\int_{s}^t\|\R_{g(t')}\|_{C^0}\,dt'}\left(1-C\int_s^t\sup_{\supp(\phi_k)}\arrowvert\Delta_{g(t')}\phi_k\arrowvert\,dt'\right),
\end{split}
\end{equation*}
which implies the expected estimate by letting $k$ go to $+\infty$.

The second estimate (\ref{L^infty-est-fct}) is proved similarly. Remark that $K_j(x,t,\cdot,\cdot)$ satisfies (\ref{scal-heat-equ-conj-ker}) with Dirichlet boundary condition. Therefore, 
\begin{eqnarray*}
\int_{\Omega_j}K_j(x,t,y,s)\,d\mu_{g(s)}(y)\leq 1,\quad s<t,
\end{eqnarray*}
for any index $j$ large enough. This implies 
\begin{eqnarray*}
\int_N K(x,t,y,s)\,d\mu_{g(s)}(y)\leq 1,\quad s<t.
\end{eqnarray*}
Similarly, one can prove the expected lower bound for $\int_NK(x,t,y,s)\,d\mu_{g(s)}(y)$.

\end{proof}

\begin{rk}\label{rk-L^1-L^infty}
In particular, Proposition \ref{L^1-bound-heat-kernel-fct} tells us that the heat semigroup associated to $\Delta_{g(t)}$ is a bounded operator when interpreted as an operator on $L^1$ and $L^{\infty}$. By interpolation, one gets for any $p\in[1,+\infty]$,
\begin{eqnarray*}
&&\|u_t\|_{L^p(d\mu_{g(t)})}\leq e^{\frac{1}{p}\int_{s}^t\|\R_{g(t')}\|_{C^0}\,dt'}\| u_s\|_{L^p(d\mu_{g(s)})},\quad t>s,\\
&&u_t(x):=K(x,t,\cdot,s)\ast u_s.
\end{eqnarray*}

\end{rk}
The next proposition concerns an on diagonal upper bound for the forward heat kernel  along the Ricci flow.
\begin{prop}[On diagonal upper bound: $L^2\rightarrow L^{\infty}$ bound]\label{diag-upper-heat-ker}
Let $(N^n,g_b)$ be an ALE Ricci flat metric. Let $(g(t))_{t\in[0,T)}$ be a solution to the Ricci flow such that $g(t)\in B_{C^0}(g_b,\varepsilon)$ for all $t\in[0,T)$ and such that $$\int_{s}^t\|\R_{g(t')}\|_{C^0}\,dt'\leq \frac{1}{2}.$$ Then,
\begin{equation*}
0<K(x,t,y,s)\leq\frac{C}{(t-s)^{\frac{n}{2}}},\quad 0\leq s<t,\quad x,y\in N,
\end{equation*}
where $C=C(n,g_b,\varepsilon)$ is a time-independent positive constant.
\end{prop}
\begin{rk}
A more accurate but also more ad-hoc assumption on the smallness of the time integral of the $C^0$ norm of the scalar curvature would be $$\int_{\frac{s+t}{2}}^t\|\R_{g(t')}\|_{C^0}\,dt'\leq \frac{1}{2}.$$
\end{rk}
We start with the following lemma that is crucial to prove Proposition \ref{diag-upper-heat-ker}:

\begin{lemma}[$L^1$ mean value inequality]\label{L^1-mean-inequ}
Let $(N^n,g_b)$ be an ALE Ricci flat metric. Let $(g(t))_{t\in[0,T)}$ be a solution to the Ricci flow such that $g(t)\in B_{C^0}(g_b,\varepsilon)$ for all $t\in[0,T)$. Then any nonnegative subsolution $u$ of the heat equation along such a Ricci flow, i.e.
\begin{equation*}
\partial_tu\leq \Delta_{g(t)}u,\quad \mbox{on $N\times (0,T)$,}
\end{equation*}
satisfies, for $\eta\in(0,1)$ and $r^2< t$ such that $\int_{t-r^2}^t\|\R_{g(s)}\|_{C^0}\,ds\leq \frac{1}{2}$,
\begin{eqnarray*}
\sup_{P\left(x,t,\eta r\right)}u\leq C\fint_{P(x,t,r)}u\,d\mu_{g(s)}ds,
\end{eqnarray*}
for some positive constant $C=C(n,g_b,\varepsilon,\eta)$ and where $$P(x,t,r):=\left\{(y,s)\in N\times[0,T)\quad|\quad s\in (t-r^2,t],\quad y\in B_{g_b}(x,r)\right\}.$$
\end{lemma}
The proof of Lemma \ref{L^1-mean-inequ} is inspired by that of Zhang \cite{Zha-Qi-Gra-Est} in a Ricci flow setting. This was already noticed in \cite{Der-Lam-Wea-Sta} in the context of expanding gradient Ricci solitons. Therefore, we only sketch a proof of Lemma \ref{L^1-mean-inequ} for the convenience of the reader. 
\begin{proof}[Proof of Lemma \ref{L^1-mean-inequ}]
Let $p\in [1,+\infty)$. Then the function $u^p$ is a sub-solution to the heat equation, i.e.
\begin{equation}\label{sub-sol-power}
\partial_tu^p\leq \Delta_{g(t)}u^p,
\end{equation}
 on $N\times (0,T)$. Consider any smooth space-time cutoff function $\psi$ and multiply (\ref{sub-sol-power}) by $\psi^2u^p$ and integrate by parts as follows: 
 \begin{equation*}
 \begin{split}
&\int_{t- r^2}^{t'}\int_N\arrowvert\nabla^{g(s)}(\psi u^p)\arrowvert^2_{g(s)}-\arrowvert\nabla^{g(s)}\psi\arrowvert^2_{g(s)}u^{2p} \,d\mu_{g(s)}ds\leq\\
& \int_{t-r^2}^{t'}\int_N\frac{\partial_s(\psi^2\ln d\mu_{g(s)})}{2}u^{2p}\,d\mu_{g(s)}ds-\frac{1}{2}\int_N\psi^2u^{2p}\,d\mu_{g(t')},
\end{split}
\end{equation*}
for any $t'\in(t-r^2,t].$

Hence,
\begin{equation}
\begin{split}\label{inequ-IBP-nash-moser-pre}
&\int_{t-r^2}^{t'}\int_N\arrowvert\nabla^{g(s)}(\psi u^p)\arrowvert^2_{g(s)}\,d\mu_{g(s)}ds+\frac{1}{2}\int_N\psi^2u^{2p}\,d\mu_{g(t')}\leq\\
&\int_{t-r^2}^{t'}\int_N\left(\frac{\partial_s\psi^2}{2}+\arrowvert\nabla^{g(s)}\psi\arrowvert_{g(s)}^2+\frac{1}{2}\|\R_{g(s)}\|_{C^0}\psi^2\right)u^{2p}\,d\mu_{g(s)}ds\leq\\
&\int_{t-r^2}^{t'}\int_N\left(\frac{\partial_s\psi^2}{2}+\arrowvert\nabla^{g(s)}\psi\arrowvert_{g(s)}^2\right)u^{2p}\,d\mu_{g(s)}ds\\
&+\frac{1}{2}\int_{t-r^2}^{t'}\|\R_{g(s)}\|_{C^0}\,ds\sup_{s\in(t-r^2,t']}\int_N\psi^2u^{2p}\,d\mu_{g(s)}\leq\\
&\int_{t-r^2}^{t'}\int_N\left(\frac{\partial_s\psi^2}{2}+\arrowvert\nabla^{g(s)}\psi\arrowvert_{g(s)}^2\right)u^{2p}\,d\mu_{g(s)}ds+\frac{1}{2}\int_{t-r^2}^{t}\|\R_{g(s)}\|_{C^0}\,ds\sup_{s\in(t-r^2,t']}\int_N\psi^2u^{2p}\,d\mu_{g(s)}.
\end{split}
\end{equation}
Notice in particular that (\ref{inequ-IBP-nash-moser-pre}) implies, if $\int_{t-r^2}^{t}\|\R_{g(s)}\|_{C^0}\,ds\leq \frac{1}{2}$,
\begin{equation}
\sup_{s\in(t-r^2,t']}\int_N\psi^2u^{2p}\,d\mu_{g(s)}\leq 4\int_{t-r^2}^{t'}\int_N\left(\frac{\partial_s\psi^2}{2}+\arrowvert\nabla^{g(s)}\psi\arrowvert_{g(s)}^2\right)u^{2p}\,d\mu_{g(s)}ds.\label{bd-sup-slice-int-power-u}
\end{equation}

Let $\tau, \sigma\in (0,+\infty)$ such that $\tau\geq \sigma$ and $\tau+\sigma<r$. In particular, $P(x,t, \tau)\subset P(x,t,\tau+\sigma)\subset P(x,t,r)$.
Now, choose two smooth functions $\phi:\mathbb{R}_+\rightarrow[0,1]$ and $\eta:\mathbb{R}_+\rightarrow[0,1]$ such that
\begin{eqnarray*}
&& \supp(\phi)\subset [0,\tau+\sigma],\quad\phi\equiv 1\quad\mbox{in $[0,\tau]$},\quad \phi\equiv 0\quad\mbox{in $[\tau+\sigma,+\infty)$},\quad -c/\sigma\leq \phi'\leq 0,\\
&& \supp(\eta)\subset [t-(\tau+\sigma)^2,+\infty),\quad\eta\equiv 1\quad\mbox{in $[t-\tau^2,+\infty)$},\\
&& \eta\equiv 0\quad\mbox{in $(t-r^2,t-(\tau+\sigma)^2]$},\quad 0\leq \eta'\leq c/\sigma^2.
\end{eqnarray*}
Define $\psi(y,s):=\phi(d_{g_b}(x,y))\eta(s)$, for $(y,s)\in N\times(0,T)$. Then,
\begin{eqnarray*}
&&\arrowvert\nabla^{g(s)}\psi\arrowvert_{g(s)}\leq\frac{C}{\sigma},\quad \arrowvert\partial_s\psi\arrowvert\leq \frac{C}{\sigma^2},
\end{eqnarray*}
for some time-independent positive constant $C$. Here we have used the fact that 
$g(t)$ is $\varepsilon$ $C^0$-close to $g_b$. 
Now, $(N^n,g(t))_{t\in(0,T)}$ satisfies the following Euclidean Sobolev inequality since $(N^n,g_b)$ does by [Chap.$3$, \cite{Sal-Cos-Sob-Boo}] and the fact that $g(t)$ is $\varepsilon$ $C^0$-close to $g_b$: there exists $C=C(g_b)>0$ such that
\begin{equation}\label{unif-sob-eucl-ineq}
\left(\int_N(\psi u^p)^{\frac{2n}{n-2}}\,d\mu_{g(t)}\right)^{\frac{n-2}{n}}\leq C\int_N\arrowvert\nabla^{g(t)}(\psi u^p)\arrowvert_{g(t)}^2\,d\mu_{g(t)},
\end{equation}
for any $t\in[0,T)$.
We are now in a position to follow \cite{Der-Lam-Wea-Sta} very closely:
if $\alpha_n:=1+2/n$, one gets by H\"older's inequality together with (\ref{unif-sob-eucl-ineq}),
\begin{equation*}
\begin{split}
\int_{P(x,t,\tau)}\left(u^{2p}\right)^{\alpha_n}\,d\mu_{g(s)}ds\leq&\int_{P(x,t,r)}\left(\psi u^p\right)^{2\alpha_n}\,d\mu_{g(s)}ds\\
\leq&\int_{t-r^2}^t\left(\int_N (\psi u^p)^{\frac{2n}{n-2}}\,d\mu_{g(s)}\right)^{\frac{n-2}{n}}\left(\int_N (\psi u^p)^2\,d\mu_{g(s)}\right)^{\frac{2}{n}}ds\\
\leq&C\sup_{s\in(t-r^2,t]}\left(\int_N (\psi u^p)^2\,d\mu_{g(s)}\right)^{\frac{2}{n}}\int_{P(x,t,r)}\arrowvert\nabla^{g(s)}(\psi u^p)\arrowvert_{g(s)}^2\,d\mu_{g(s)}\\
\leq&C\frac{1}{\sigma^{2\alpha_n}}\left(\int_{P(x,t,\tau+\sigma)}u^{2p}\,d\mu_{g(s)}\right)^{\alpha_n},
\end{split}
\end{equation*}
for some time-independent positive constant $C$.
Here, we have used (\ref{inequ-IBP-nash-moser-pre}) and (\ref{bd-sup-slice-int-power-u}) in the last line together with the fact that $\psi$ is compactly supported in $P(x,t,\tau+\sigma)$. It is now sufficient to iterate the previous inequality for suitable sequences $(p_i)_i$, $(r_i)_i$ and $(\tau_i)_i$, $(\sigma_i)_i$ to reach the desired conclusion. 

\end{proof}
We are in a position to give a proof of Proposition \ref{diag-upper-heat-ker}.

\begin{proof}[Proof of Proposition \ref{diag-upper-heat-ker}]
It suffices to apply Lemma \ref{L^1-mean-inequ} together with Proposition \ref{L^1-bound-heat-kernel-fct} to the nonnegative (sub)solution 
\begin{equation*}
u(x,t):=K(x,t,y,s),
\end{equation*}
 for some fixed $(y,s)\in N\times[0,T)$ with $r^2=(t-s)/2$: 
\begin{equation*}
\begin{split}
K(x,t,y,s)\leq&\sup_{P\left(x,t,\frac{1}{2}\sqrt{\frac{t-s}{2}}\right)}K(\cdot,\cdot,y,s)\\
\leq&\frac{C}{(t-s)^{\frac{n+2}{2}}}\int_{P(x,t,\sqrt{\frac{t-s}{2}})}K(x',t',y,s)\,d\mu_{g(t')}(x')dt'\\
\leq&\frac{C}{(t-s)^{\frac{n}{2}}}e^{\int_s^t\|\R_{g(t')}\|_{C^0}\,dt'},
\end{split}
\end{equation*}
for some positive constant $C=C(n,g_b,\varepsilon).$
This ends the proof of the expected estimate since $\int_s^t\|\R_{g(t')}\|_{C^0}\,dt'\leq \frac{1}{2}$.

\end{proof}

We now state the main result of this section:
\begin{theo}[Gaussian estimate] \label{Gaussian-est-scal-heat-ker}
Let $(N^n,g_b)$ be an ALE Ricci flat metric. Let $(g(t))_{t\in[0,T)}$ be a solution to the Ricci flow such that $g(t)\in B_{C^0}(g_b,\varepsilon)$ for all $0\leqslant s \leqslant t<T$ and such that 
\begin{equation}
\int_{s}^t\|\Ric(g(t'))\|_{C^0}\,dt'\leq \frac{1}{2n}.\label{hyp-ad-hoc-ric}
\end{equation}

 Then the heat kernel associated to (\ref{scal-heat-equ-ker}) satisfies the following Gaussian estimate: 
\begin{equation}
\begin{split}
&|K(x,t,y,s)|\leq \frac{C}{(t-s)^{\frac{n}{2}}}\exp\left\{-\frac{d_{g(s)}^2(x,y)}{D(t-s)}\right\},\quad 0\leq s<t,\quad x,y\in N,\quad D>D_0,\label{gauss-heat-ker}
\end{split}
\end{equation}
 where $C=C(n,g_b,\varepsilon,D)$ and $D_0= D_0(n,g_b,\varepsilon)$ are time-independent positive constants.
\end{theo}
\begin{rk}
The Gaussian weight in (\ref{gauss-heat-ker}) could be stated in terms of the distance $d_{g(t)}(x,y)$ instead thanks to Proposition \ref{prop-A-priori-dist-distance-RF}.
\end{rk}
\begin{proof}
 Again, we follow closely the presentation of \cite{Der-Lam-Wea-Sta}. Define the following integral quantities for a positive constant $D$ to be chosen later:
 \begin{equation*}
 \begin{split}
&E^1_D(s,t,z):=\int_N|K(x,t,z,s)|^2e^{\frac{2d_{g(t)}^2(x,z)}{D(t-s)}}\,d\mu_{g(t)}(x),\quad s<t,\quad z\in N, \\
&E^2_D(s,t,z):=\int_N| K(z,t,y,s)|^2e^{\frac{2d_{g(s)}^2(z,y)}{D(t-s)}}\,d\mu_{g(s)}(y),\quad s<t,\quad z\in N.\\
\end{split}
\end{equation*}

The motivation for introducing such quantities follows from the next crucial observation: 
 by the semi-group property and the triangular inequality,
\begin{equation*}
\begin{split}
K(x,t,y,s)&=\int_NK(x,t,z,\tau)\cdot K(z,\tau,y,s)\,d\mu_{g(\tau)}(z)\\
&\leq\int_NK(x,t,z,\tau)e^{\frac{d_{g(\tau)}^2(x,z)}{D(t-\tau)}}K(z,\tau,y,s)e^{\frac{d_{g(\tau)}^2(z,y)}{D(\tau-s)}}\,d\mu_{g(\tau)}(z)e^{-\frac{d_{g(\tau)}^2(x,y)}{D(t-s)}}\\
\end{split}
\end{equation*}
where $s<\tau:=\frac{t+s}{2}<t.$
Therefore, by Cauchy-Schwarz inequality, we get the following "universal" inequality:
\begin{equation}\label{univ-ineq-heat-ker}
|K(x,t,y,s)|\leq C \sqrt{E^1_{D}(s,\tau,y)E^2_{ D}(\tau,t,x)}\exp\left\{-\frac{d_{g(\tau)}^2(x,y)}{D(t-s)}\right\}.
\end{equation}
We claim the following:
\begin{claim}\label{claim-est-univ-fct-E}
\begin{equation}
E^1_{D}(s,t,\cdot)+E^2_D(s,t,\cdot)\leq \frac{C}{(t-s)^{n/2}},\quad s<t,\quad D>D_0,
\end{equation}
for some time-independent positive constants $C=C(n,g_b,\varepsilon,D)$ and $D_0= D_0(n,g_b,\varepsilon)$.

\end{claim}
The proof of Claim \ref{claim-est-univ-fct-E} is essentially based on Proposition \ref{diag-upper-heat-ker} implied by (\ref{hyp-ad-hoc-ric}) and is virtually identical to the proof given in \cite{Der-Lam-Wea-Sta}: it will therefore be omitted.

Now, thanks to (\ref{univ-ineq-heat-ker}) together with Claim \ref{claim-est-univ-fct-E}, one obtains the expected Gaussian estimate: 
\begin{eqnarray*}
|K(x,t,y,s)|&\leq&\frac{C}{(t-s)^{\frac{n}{2}}}\exp\left\{-\frac{d_{g(\tau)}^2(x,y)}{D(t-s)}\right\}\\
&\leq&\frac{C}{(t-s)^{\frac{n}{2}}}\exp\left\{-\frac{d_{g(s)}^2(x,y)}{e^{\frac{1}{n}}\cdot D(t-s)}\right\},\quad 0\leq s<t,\quad x,y\in N,
\end{eqnarray*}
for some positive constant $C=C(n,g_b,\varepsilon,D)$ and any $D\geq D_0(n,g_b,\varepsilon)$. Here we have used Proposition \ref{prop-A-priori-dist-distance-RF} to estimate from below the distance $d_{g(\tau)}(x,y)$ in terms of the distance $d_{g(s)}(x,y)$.

\end{proof}

\subsection{$L^2-C^0$ estimate}~~\\


We start by checking that the $L^2$-norms of the Ricci tensor and the Hessian of the potential function satisfying (\ref{eqdeffg}) are controlled from above by that of the associated Bakry-\'Emery tensor.
 \begin{prop}[A priori $L^2$ estimate for the Ricci flow]\label{stabilityALE-RF-L2}
Let $(N^n,g_b)$ be an ALE Ricci flat metric. Let $g$ be a metric in $B_{C^{2,\alpha}_{\tau}}(g_b,\varepsilon)$. Then,
\begin{equation}
\begin{split}\label{a-priori-l2-bd-ric-hess-bakry}
\|\Ric(g)\|_{L^2}+\|\nabla^{g,2}f_{g}\|_{L^2}\leq C\left(\|\Ric(g)+\nabla^{g,2}f_{g}\|_{L^2}\right),
\end{split}
\end{equation}
for some positive constant $C=C(n,g_b,\varepsilon)$. 
\end{prop}

\begin{proof}
Let $g$ be a metric in $B_{C^{2,\alpha}_{\tau}}(g_b,\varepsilon)$. Then observe that from the Euler-Lagrange equation (\ref{eqdeffg}) satisfied by the potential function $f_{g}$, one gets:
\begin{equation}
\Delta_{f_{g}}f_{g}=-\Delta_{g}f_{g}-\R_{g}=-\tr_{g}\left(\Ric(g)+\nabla^{g,2}f_{g}\right).\label{euler-lag-disguise}
\end{equation}
 Let us apply the Bochner formula to the smooth metric measure space $\left(N^n,g,\nabla^{g}f_{g}\right)$:
\begin{equation*}
\begin{split}
\Delta_{f_{g}}|\nabla^{g}f_{g}|^2_{g}=&2|\nabla^{g,2}f_{g}|^2_{g}+2\left(\Ric(g)+\nabla^{g,2}f_{g}\right)(\nabla^{g}f_{g},\nabla^{g}f_{g})\\
&+2\langle\nabla^{g}f_{g},\nabla^{g}\Delta_{f_{g}}f_{g}\rangle_{g}\\
=&2|\nabla^{g,2}f_{g}|^2_{g}+2\left(\Ric(g)+\nabla^{g,2}f_{g}\right)(\nabla^{g}f_{g},\nabla^{g}f_{g})\\
&-2\left\langle\nabla^{g}f_{g},\nabla^{g}\tr_{g}\left(\Ric(g)+\nabla^{g,2}f_{g}\right)\right\rangle_{g}.
\end{split}
\end{equation*}
Here, we have used (\ref{euler-lag-disguise}) in the second line. By integrating by parts (with respect to the weighted measure $e^{-f_{g}}d\mu_{g}$) the previous identity, one ends up with:
\begin{equation}
\begin{split}\label{inequ-boc-hess-est-T}
2\|\nabla^{g,2}f_{g}\|^2_{L^2}=&-2\left\langle \Delta_{f_{g}}f_{g},\tr_{g}\left(\Ric(g)+\nabla^{g,2}f_{g}\right)\right\rangle_{L^2}\\
&-2\int_N\left(\Ric(g)+\nabla^{g,2}f_{g}\right)(\nabla^{g}f_{g},\nabla^{g}f_{g})\,e^{-f_{g}}d\mu_{g}\\
=&2\left\|\tr_{g}\left(\Ric(g)+\nabla^{g,2}f_{g}\right)\right\|^2_{L^2}\\
&-2\int_N\left(\Ric(g)+\nabla^{g,2}f_{g}\right)(\nabla^{g}f_{g},\nabla^{g}f_{g})\,e^{-f_{g}}d\mu_{g}\\
\leq&2\left\|\tr_{g}\left(\Ric(g)+\nabla^{g,2}f_{g}\right)\right\|^2_{L^2}\\
&+\|\rho_{g_b}^2\left(\Ric(g)+\nabla^{g,2}f_{g}\right)\|_{C^0}\|\rho_{g_b}^{-1}\nabla^{g}f_{g}\|_{L^2}^2\\
\leq&2\left\|\tr_{g}\left(\Ric(g)+\nabla^{g,2}f_{g}\right)\right\|^2_{L^2}\\
&+C(n,g_b)\|\rho_{g_b}^2\left(\Ric(g)+\nabla^{g,2}f_{g}\right)\|_{C^0}\|\nabla^{g,2}f_{g}\|^2_{L^2}.
\end{split}
\end{equation}
Here we have used (\ref{euler-lag-disguise}) in the second line, and Hardy's inequality from \cite{Min-Har-Ine} is invoked in the last inequality. In particular, if $g$ lies in a sufficiently $C^{2,\alpha}_{\tau}$ small neighborhood of $g_b$ such that $C(n,g_b)\|\rho_{g_b}^2(\Ric(g)+\nabla^{g,2}f_{g})\|\leq 1$ then the last term on the righthand side of (\ref{inequ-boc-hess-est-T}) can be absorbed by the lefthand side so that,
\begin{equation}
\|\nabla^{g,2}f_{g}\|^2_{L^2}\leq 2\left\|\tr_{g}\left(\Ric(g)+\nabla^{g,2}f_{g}\right)\right\|^2_{L^2}.\label{inequ-hess-tr-T}
\end{equation}
This ends the proof of the estimate on the Hessian term in (\ref{a-priori-l2-bd-ric-hess-bakry}). The bound on the $L^2$-norm of the Ricci curvature follows by the triangular inequality.
\end{proof}

The next proposition lets us get an a priori $C^0$ estimate on the distance from a Ricci flow to the origin given by an ALE Ricci flat metric:
\begin{prop}[A priori $L^2-C^{0}$ estimate for the Ricci flow]\label{coro-A-priori-L^2-C^0-est-RF}
Let $(N^n,g_b)$ be an ALE Ricci flat metric. Let $(g(t))_{t\in[0,T)}$ be a solution to the Ricci flow such that $g(t)\in B_{C^0}(g_b,\varepsilon)$ and with bounded curvature for all $t\in[0,T)$.

For $t\in(0,T)$ and $0<r^2<t$, if $\int_{t-r^2}^t\|\R_{g(s)}\|_{C^0}\,ds\leq \frac{1}{2}$ and $\Ric(g(s))\in L^2$, $s\in[t-r^2,t]$, then,
\begin{eqnarray}
\|\Ric(g(t))\|_{C^0}&\leq& Cr^{-\frac{n}{2}}\exp\left(c(n)\int_{t-r^2}^t\|\Rm(g(s))\|_{C^0}\,ds\right)\|\Ric(g(t-r^2)\|_{L^2},\label{int-ric-c-0}
\end{eqnarray}
for some positive constant $C=C(n,g_b,\varepsilon)$.
\end{prop}

\begin{proof}
Thanks to the evolution equation [(\ref{evo-eqn-Ric-for}), Lemma \ref{lemma-evo-eqn-Ric-Rm}] satisfied by the Ricci tensor, one gets that the function $u(t):=\exp\left(-c(n)\int_{0}^t\|\Rm(g(s))\|_{C^0}\,ds\right)|\Ric(g(t)|_{g(t)}$ is a subsolution to the heat equation, i.e. $\partial_tu\leq \Delta_{g(t)}u$ in the weak sense.

Thanks to Lemma \ref{L^1-mean-inequ}, if $r^2< t\leq T$ and $x\in N$, one is led to:
\begin{equation}
\begin{split}
&\sup_{B_{g_b}(x,\frac{r}{2})}|\Ric(g(t))|^2_{g(t)}\leq \frac{C}{r^{n+2}}e^{c(n)\int_{t-r^2}^t\|\Rm(g(s))\|_{C^0}\,ds}\int_{t-r^2}^t\int_{B_{g_b}(x,r)}|\Ric(g(s))|^2\,d\mu_{g(s)}ds,\label{C0-L2-ver-0-Ric}
\end{split}
\end{equation}
where $C=C(n,g_b,\varepsilon)$ is a positive constant which is independent of space and time variables.

This being said, estimate (\ref{C0-L2-ver-0-Ric}) implies in particular that 
\begin{equation}
\|\Ric(g(t))\|_{C^0}\leq \frac{C(n,g_b,\varepsilon)}{r^{\frac{n}{2}}}e^{c(n)\int_{t-r^2}^t\|\Rm(g(s))\|_{C^0}\,ds}\sup_{t-r^2\leq s\leq t}\|\Ric(g(s))\|_{L^2}.\label{corona-inequ-v0}
\end{equation}
Now, Proposition \ref{prop-a-priori-L2-est-rm-ric} gives:
 \begin{equation}
 \sup_{t-r^2\leq s\leq t}\|\Ric(g(s))\|_{L^2}\leq e^{c(n)\int_{t-r^2}^t\|\Rm(g(s))\|_{C^0}\,ds}\|\Ric(g(t-r^2))\|_{L^2}.
 \end{equation}
 Combining this fact with (\ref{corona-inequ-v0}) finally give 
\begin{equation}
\|\Ric(g(t))\|_{C^0}\leq \frac{C(n,g_b,\varepsilon)}{r^{\frac{n}{2}}}e^{c(n)\int_{t-r^2}^t\|\Rm(g(s))\|_{C^0}\,ds}\|\Ric(g(t-r^2))\|_{L^2}.\label{corona-inequ-v1}
\end{equation}
Then (\ref{corona-inequ-v1}) establishes the desired estimate (\ref{int-ric-c-0}).
\end{proof}

We end this section by an estimate on the decay of the $C^0$-norm of the Ricci tensor along a Ricci flow satisfying some mild assumptions:
\begin{prop}\label{prop-ric-dec-large-time}
Let $(N^n,g_b)$ be an ALE Ricci flat metric. Let $(g(t))_{t\in[0,T)}$ be a solution to the Ricci flow such that $g(t)\in B_{C^0}(g_b,\varepsilon)$ and with bounded curvature for all $t\in[0,T)$ and such that 
\begin{equation}
\int_{0}^t\|\Ric(g(t'))\|_{C^0}\,dt'\leq \frac{1}{2n}.
\end{equation}
Then, if $p\geq 1$, $0<s<t<T$ and $\Ric(g(s))\in L^p$,
\begin{equation}
\|\Ric(g(t))\|_{C^0}\leq  \frac{C}{(t-s)^{\frac{n}{2p}}}\|\Ric(g(s))\|_{L^p}+C\sup_{t'\in[0,T)}\|\Rm(g(t'))\|_{C^0}\int_s^t\|\Ric(g(t'))\|_{C^0}\,dt',
\end{equation}
for some time-independent positive constant $C=C(n,g_b,\varepsilon,p).$
\end{prop}
\begin{rk}
Proposition \ref{prop-ric-dec-large-time} will be most useful in case $p>\frac{n}{\tau+2}$ since we will ultimately consider a solution $(g(t))_{t\in[0,T)}$ to the Ricci flow such that $\Ric(g(t))=O(\rho_{g_b}^{-\tau-2})$ for each time $t\in[0,T)$.
\end{rk}
\begin{proof}
The assumptions legitimate the use of Theorem \ref{Gaussian-est-scal-heat-ker} through Duhamel's formula as we now explain.

According to the evolution equation satisfied by the Ricci curvature along the Ricci flow given by Lemma (\ref{lemma-evo-eqn-Ric-Rm}) together with Duhamel's formula, one can write:
\begin{equation}
\begin{split}\label{duha-ric-evo-eqn-c0-dec-time}
\Ric(g(t))(x)=&\int_N\left<\mathcal{K}(x,t,\cdot,s),\Ric(g(s))\right>\,d\mu_{g(s)}\\
&+\int_s^t\int_N\left<\mathcal{K}(x,t,\cdot,t'),\Rm(g(t'))\ast\Ric(g(t'))\right>\,d\mu_{g(t')}\,dt',
\end{split}
\end{equation}
where $\mathcal{K}(x,t,y,s)$ denotes the heat kernel acting on symmetric $2$-tensors associated to the one-parameter family of metrics $(g(t))_{t\in[0,T)}$.

Now, by Kato's inequality, the operator norm of $\mathcal{K}(x,t,y,s)$ is bounded by the heat kernel acting on functions, i.e.
\begin{equation*}
\|\mathcal{K}(x,t,y,s)\|_{Hom(S^2T_y^*N,S^2T_x^*N)}\leq K(x,t,y,s),\quad (x,y)\in N\times N, \,\,t>s\geq 0,
\end{equation*}
where $K(x,t,y,s)$ denotes the heat kernel acting on functions with respect to the metric $g(t)$ introduced and studied in Section \ref{sec-heat-ker}. In particular, [(\ref{gauss-heat-ker}), Theorem \ref{Gaussian-est-scal-heat-ker}] ensures (once we fix $D=D_0$) that $$\|K(x,t,\cdot,s)\|_{L^q}\leq \frac{C(n,g_b)}{(t-s)^{\frac{n}{2}\left(1-\frac{1}{q}\right)}} ,$$ for $0<s<t<T.$ 
Therefore, one gets by H\"older's inequality and \eqref{duha-ric-evo-eqn-c0-dec-time}:
\begin{equation*}
\begin{split}
\|\Ric(g(t))\|_{C^0}\leq\,& \sup_{x\in N}\|K(x,t,\cdot,s)\|_{L^{q}}\|\Ric(g(s))\|_{L^p}\\
&+\sup_{t'\in[0,T)}\|\Rm(g(t'))\|_{C^0}\int_s^t\|K(x,t,\cdot,s)\|_{L^{1}}\|\Ric(g(t'))\|_{C^0}\,dt'\\
\leq\,&\frac{C(n,g_b)}{(t-s)^{\frac{n}{2p}}}\|\Ric(g(s))\|_{L^p}+C\sup_{t'\in[0,T)}\|\Rm(g(t'))\|_{C^0}\int_s^t\|\Ric(g(t'))\|_{C^0}\,dt',
\end{split}
\end{equation*}
where $p,q\geq 1$ are such that $p^{-1}+q^{-1}=1$ and where $C$ is a positive constant independent of time.

\end{proof}
\subsection{Weighted estimates}\label{sec-wei-est}~~\\

In this section, we consider an ALE Ricci flat metric $(N^n,g_b)$ and prove an a priori $C^0$ estimate on the distance of a Ricci flow $(g(t))_{t\in[0,T)}$ lying in a small neighborhood $B_{C^{0}}(g_b,\varepsilon)$ starting from a metric $g(0)\in B_{C^{2,\alpha}_{\tau}}(g_b,\varepsilon)$.

To do so, we essentially use the Gaussian estimates proved in Theorem \ref{Gaussian-est-scal-heat-ker} via the Duhamel's formula. We will distinguish short-time from large-time estimates.

\begin{lemma}[Short-time estimates]\label{lemma-short-time-est-curv}
Let $(N^n,g_b)$ be an ALE Ricci flat metric. Let $\tau\in(\frac{n-2}{2},n-2)$, $\alpha\in\left(0,\min\{1,n-2-\tau\}\right)$ and let $(g(t))_{t\in[0,T)}$ be a solution to the Ricci flow in a neighborhood $B_{C^{0}}(g_b,\varepsilon)$ such that $g(0)\in B_{C^{2,\alpha}_{\tau}}(g_b,\varepsilon)$ and
\begin{equation}
\int_{0}^t\|\Ric(g(t'))\|_{C^0}\,dt'\leq \frac{1}{2n}.\label{hyp-ad-hoc-ric-1}
\end{equation}
Then for $t\in[0,T)\,\cap\,[0,T_{\operatorname{Shi}}]$,
\begin{eqnarray}
&&\|\Rm(g(t))\|_{C^{0,\alpha}_{\tau+2}}\leq C(n,g_b,\varepsilon)e^{C(n,g_b,\varepsilon)\cdot t}\|\Rm(g(0))\|_{C^{0,\alpha}_{\tau+2}},\label{short-rm-wei-c0}\\
&&\|\Ric(g(t))\|_{C^{0,\alpha}_{\tau+2}}\leq C(n,g_b,\varepsilon)e^{C(n,g_b,\varepsilon)\cdot t}\|\Ric(g(0))\|_{C^{0,\alpha}_{\tau+2}}.\label{short-ric-wei-c0}
\end{eqnarray}
\end{lemma}
\begin{rk}
The proof of Lemma \ref{lemma-short-time-est-curv} uses Theorem \ref{Gaussian-est-scal-heat-ker} which explains the assumption (\ref{hyp-ad-hoc-ric-1}). Using the maximum principle and suitable barriers in the same spirit as it is done in \cite{Li-Yu-AF} would have led to the same result without assuming (\ref{hyp-ad-hoc-ric-1}). 
\end{rk}
\begin{proof}
According to (\ref{shi-est-curv-tensor}) from Proposition \ref{prop-a-priori-C0-est-rm-ric}, the curvature operator is uniformly bounded on $[0,T)\,\cap\,[0,T_{\operatorname{Shi}}]$. Thanks to the evolution equation satisfied by the curvature operator [(\ref{evo-eqn-Rm-for}), Lemma \ref{lemma-evo-eqn-Ric-Rm}], this implies that the function $e^{-c(n)\|\Rm(g(0))\|_{C^0}\cdot t}|\Rm(g(t))|_{g(t)}$ is a subsolution to the heat equation along the Ricci flow. In particular, by Duhamel's formula together with the maximum principle,
\begin{equation*}
\begin{split}
|\Rm(g(t))|_{g(t)}(x)\leq\,& e^{c(n)\|\Rm(g(0))\|_{C^0}\cdot t}\int_NK(x,t,y,0)|\Rm(g(0))|_{g(0)}(y)\,d\mu_{g(0)}(y)\\
\leq\,&e^{C\cdot t}\int_NK(x,t,y,0)|\Rm(g(0))|_{g(0)}(y)\,d\mu_{g(0)}(y),
\end{split}
\end{equation*}
for some positive constant $C=C(n,g_b,\varepsilon)$ by assumption on $g(0)$.
By assumption, [(\ref{gauss-heat-ker}), Theorem \ref{Gaussian-est-scal-heat-ker}] holds true and we get (once we fix $D=D_0$),
\begin{equation}
\begin{split}
|\Rm(g(t))|_{g(t)}(x)&\leq\frac{C}{t^{\frac{n}{2}}}\int_N\exp\left\{-\frac{d_{g(0)}^2(x,y)}{D_0t}\right\}|\Rm(g(0))|_{g(0)}(y)\,d\mu_{g(0)}(y)\\
&\leq C\|\Rm(g(0))\|_{C^0_{\tau+2}}\int_Nt^{-\frac{n}{2}}\exp\left\{-\frac{d_{g_b}^2(x,y)}{2D_0t}\right\}\rho_{g_b}^{-\tau-2}(y)\,d\mu_{g_b}(y),\label{est-rough-dec-rm}
\end{split}
\end{equation}
where we have used $d_{g(0)}^2(x,y)\geq \frac{1}{2}d_{g_b}^2(x,y)$ for all $x,y\in N$ in the last line. By splitting the integral on the righthand side of (\ref{est-rough-dec-rm}) in two parts, whether $y\in B_{g_b}(p,d_{g_b}(p,x)/2)$ or not for some fixed point $p\in N$, one gets:
\begin{equation}
\begin{split}\label{est-heat-ker-weight-pol-rm}
\int_Nt^{-\frac{n}{2}}&\exp\left\{-\frac{d_{g_b}^2(x,y)}{2D_0t}\right\}\rho_{g_b}^{-\tau-2}(y)\,d\mu_{g_b}(y)\leq \\
&\quad t^{-\frac{n}{2}}\exp\left\{-\frac{d_{g_b}^2(p,x)}{8D_0t}\right\}\int_{B_{g_b}(p,d_{g_b}(p,x)/2)}\rho_{g_b}^{-\tau-2}(y)\,d\mu_{g_b}(y)\\
&\quad +\frac{C}{(d_{g_b}(p,x)+1)^{\tau+2}}\int_{B^{c}_{g_b}(p,d_{g_b}(p,x)/2)}t^{-\frac{n}{2}}\exp\left\{-\frac{d_{g_b}^2(x,y)}{2D_0t}\right\}\,d\mu_{g_b}(y)\\
&\leq C\left(\frac{d_{g_b}^2(p,x)}{t}\right)^{\frac{n}{2}}\exp\left\{-\frac{d_{g_b}^2(p,x)}{8D_0t}\right\}(d_{g_b}(p,x)+1)^{-\tau-2}+\frac{C}{(d_{g_b}(p,x)+1)^{\tau+2}}\\
&\leq \frac{C}{(d_{g_b}(p,x)+1)^{\tau+2}},
\end{split}
\end{equation}
for some positive constant $C$ independent of time $t>0$ that may vary from line to line. Here, we have used the fact that $\tau+2<n$ in the second line. Estimates (\ref{est-rough-dec-rm}) and (\ref{est-heat-ker-weight-pol-rm}) lead to the expected result (\ref{short-rm-wei-c0}) at the level of the $C^0$-norm:
\begin{equation}
\|\Rm(g(t))\|_{C^{0}_{\tau+2}}\leq Ce^{C\cdot t}\|\Rm(g(0))\|_{C^{0}_{\tau+2}},\label{short-rm-wei-c0-proof}
\end{equation}
for some positive constant $C=C(n,g_b,\varepsilon).$

 Regarding the $C^{\alpha}$-seminorm, by a similar reasoning considering difference quotients together with Duhamel's formula	, one gets schematically if $\alpha\in\left(0,\min\{1,n-2-\tau\}\right)$:
 \begin{equation}
 \begin{split}\label{short-rm-wei-calpha-proof}
\rho_{g_b}^{\tau+2+\alpha}[\Rm(g(t))]_{C^{\alpha}}\leq\,& Ce^{C\cdot t}\| \rho_{g_b}^{\tau+2+\alpha}[\Rm(g(0))]_{C^{\alpha}}\|_{C^0}\\
&+C \int_0^t \|\Rm(g(t'))\ast[\Rm(g(t'))]_{C^{\alpha}}\|_{C^{0}_{\tau+2+\alpha}}\,dt'\\
\leq\,& Ce^{C\cdot t}\|\Rm(g(0))\|_{C^{0,\alpha}_{\tau+2}}\\
&+Ce^{C\cdot t} \|\Rm(g(0))\|_{C^0}\int_0^t \|\Rm(g(t'))\|_{C^{0,\alpha}_{\tau+2}}\,dt'.
\end{split}
\end{equation}
Concatenating (\ref{short-rm-wei-calpha-proof}) together with (\ref{short-rm-wei-c0-proof}) leads to (\ref{short-rm-wei-c0}) by invoking Gr\"onwall's inequality.

The proof of (\ref{short-ric-wei-c0}) goes along the same lines by using (\ref{evo-eqn-Ric-for}) together with (\ref{short-rm-wei-c0}).
\end{proof}
The next lemma takes care of a priori integral-in-time estimates for the weighted $C^0_{\tau}$-norm of the Ricci curvature.

\begin{lemma}[A priori large-time $C^0$ estimate]\label{lemma-a-priori-wei-metr-wei}
Let $(N^n,g_b)$ be an ALE Ricci flat metric. Let $\tau\in(\frac{n-2}{2},n-2)$ and let $(g(t))_{t\in[0,T)}$ be a solution to the Ricci flow in a neighborhood $B_{C^{0}}(g_b,\varepsilon)$ such that $\sup_{t\in [0,T)}\|\Rm(g(t))\|_{C^0_{\tau+2}}\leq C(n,g_b)$ and 
\begin{equation}
\int_{0}^t\|\Ric(g(t'))\|_{C^0}\,dt'\leq \frac{1}{2n}.\label{hyp-ad-hoc-ric-2}
\end{equation}
 Then we have
\begin{equation}
\begin{split}\label{dream-C-0-wei-int-ric-stat}
\int_s^t\|\Ric(g(t'))\|_{C^0_{\tau}}\,dt'&\leq C\left(\|\Ric(g(s))\|_{C^0_{\tau+2}}+\int_s^t\|\Ric(g(s'))\|_{C^0}\,ds'\right),\quad  0\leq s< t<T,
\end{split}
\end{equation}
for some time-independent positive constant $C=C(n,g_b,\varepsilon)$.
In particular, one gets the following a priori $C^0$ estimate:
\begin{equation}\label{dream-C-0-wei-dist-met-stat}
\|g(t)-g(s)\|_{C^0_{\tau}}\leq C\left(\|\Ric(g(s))\|_{C^0_{\tau+2}}+\int_s^t\|\Ric(g(s'))\|_{C^0}\,ds'\right),\quad  0\leq s<t<T.
\end{equation}

\end{lemma}

\begin{proof}
By arguing as in the proof of Proposition \ref{prop-ric-dec-large-time}, one ensures with the help of Duhamel's formula that for $x\in N$ and $t\geq s\geq 0$,
\begin{equation}
\begin{split}\label{est-duha-dec-ric}
|\Ric(g(t))|_{g(t)}(x)\leq&\int_NK(x,t,\cdot,s)|\Ric(g(s))|_{g(s)}\,d\mu_{g(s)}\\
&+\int_s^t\int_NK(x,t,\cdot,t')|\Rm(g(t'))\ast\Ric(g(t'))|_{g(t')}\,d\mu_{g(t')}\,dt'.
\end{split}
\end{equation}

By invoking [(\ref{gauss-heat-ker}), Theorem \ref{Gaussian-est-scal-heat-ker}] (once we fix $D=D_0$), the first term on the righthand side of (\ref{est-duha-dec-ric}) can be estimated as follows:
\begin{equation}
\begin{split}
K(x,t,\cdot,s)\ast | \Ric(g(s))|_{g(s)}&:=\int_NK(x,t,\cdot,s)|\Ric(g(s))|_{g(s)}\,d\mu_{g(s)}\\
&\leq\frac{C}{(t-s)^{\frac{n}{2}}}\int_N\exp\left\{-\frac{d_{g(s)}^2(x,y)}{D_0(t-s)}\right\}|\Ric(g(s))|_{g(s)}(y)\,d\mu_{g(s)}(y)\\
&\leq C\|\Ric(g(s))\|_{C^0_{\tau+2}}\int_N(t-s)^{-\frac{n}{2}}\exp\left\{-\frac{d_{g_b}^2(x,y)}{2D_0(t-s)}\right\}\rho_{g_b}^{-\tau-2}(y)\,d\mu_{g_b}(y),
\end{split}
\end{equation}
where we have used $d_{g(s)}^2(x,y)\geq \frac{1}{2}d_{g_b}^2(x,y)$ for all $x,y\in N$ in the last line together with the fact that $g(t)\in B_{C^0}(g_b,\varepsilon)$ for all $t\in[0,T)$.

Observe that for $x\neq y$, and $n\geq 3$,
\begin{equation*}
\begin{split}
\int_s^{+\infty} (t-s)^{-\frac{n}{2}}\exp\left\{-\frac{d_{g_b}^2(x,y)}{2D_0(t-s)}\right\}\,dt\leq Cd_{g_b}^{-n+2}(x,y).
\end{split}
\end{equation*}

In particular, by Fubini's theorem,
\begin{equation*}
\begin{split}
\int_s^tK(x,t',\cdot,s)\ast|\Ric(g(s))|_{g(s)}\,dt'\leq C\|\Ric(g(s))\|_{C^0_{\tau+2}}\int_Nd_{g_b}^{-n+2}(x,y)\rho_{g_b}^{-\tau-2}(y)\,d\mu_{g_b}(y).
\end{split}
\end{equation*}
By considering regions of the type $B_{g_b}(p,d_{g_b}(p,x)/2)$, $B_{g_b}(p,2d_{g_b}(p,x))\setminus B_{g_b}(p,d_{g_b}(p,x)/2)$ and $N\setminus B_{g_b}(p,2d_{g_b}(p,x))$ for a point $x\in N$ and a fixed point $p\in N$, one gets an a priori estimate on the following $C^0_{\tau}$-norm:
\begin{equation}\label{first-est-C^0-int-ric}
\int_s^t\|K(x,t',\cdot,s)\ast|\Ric(g(s))|_{g(s)}\|_{C^0_{\tau}}\,dt'\leq C\|\Ric(g(s))\|_{C^0_{\tau+2}}.
\end{equation}

Notice that we heavily use the restriction $\tau<n-2$ here.
We proceed similarly to handle the second term on the righthand side of (\ref{est-duha-dec-ric}):
\begin{equation}
\begin{split}\label{sec-est-C^0-int-ric}
\int_s^t&\left|\int_s^{t'}\int_N\left<\mathcal{K}(x,t',\cdot,s'),\Rm(g(s'))\ast\Ric(g(s'))\right>\,d\mu_{g(s')}\,ds'\right|dt'\leq\\
&C\int_s^t\int_N\left(\int_{s'}^tK(x,t',y,s')\,dt'\right)|\Rm(g(s'))|(y)\,d\mu_{g(s')}(y)\|\Ric(g(s'))\|_{C^0}\,ds'\\
&\leq C\left(\int_s^t\|\Ric(g(s'))\|_{C^0}\,ds'\right)\int_Nd_{g_b}(x,y)^{-n+2}\rho_{g_b}(y)^{-\tau-2}\,d\mu_{g_b}(y)\\
&\leq C\left(\int_s^t\|\Ric(g(s'))\|_{C^0}\,ds'\right)\rho_{g_b}(x)^{-\tau}.
\end{split}
\end{equation}
Here, we have used Fubini's theorem in the second line together with the assumption on the curvature tensor $\Rm(g(s'))$, $s'\in[0,T)$ in the penultimate line which implies that the curvature tensor $\Rm(g(s'))$ decays as fast as $\rho_{g_b}^{-\tau-2}$ uniformly in time. 


Estimates (\ref{first-est-C^0-int-ric}) and (\ref{sec-est-C^0-int-ric}) lead directly to the desired estimate (\ref{dream-C-0-wei-int-ric-stat}).

Finally, (\ref{dream-C-0-wei-dist-met-stat}) is proved with the help of (\ref{dream-C-0-wei-int-ric-stat}) by noticing that:
\begin{equation*}
\begin{split}
\|g(t)-g(s)\|_{C^0_{\tau}}&\leq\int_s^t\|\Ric(g(t'))\|_{C^0_{\tau}}\,dt'.
\end{split}
\end{equation*}

\end{proof}
The following lemma proves an a priori weighted $C^{0,\alpha}$ estimate on the Ricci tensor of a solution of the Ricci flow $C^{2,\alpha}_{\tau}$-close to a stable ALE Ricci flat metric.
\begin{lemma}[A priori weighted $C^{0,\alpha}$ estimate on Ricci curvature]\label{a-priori-wei-lemma-C0-Ric}
Let $(N^n,g_b)$, $n\geq 4$, be an ALE Ricci flat metric. Let $\tau\in(\frac{n-2}{2},n-2)$, $\alpha\in(0,1)$ and let $(g(t))_{t\in[0,T)}$ be a solution to the Ricci flow in a neighborhood $B_{C^{2,\alpha}_{\tau}}(g_b,\varepsilon)$ with uniformly bounded covariant derivatives, i.e. $\|\nabla^{g(t),k}\Rm(g(t))\|_{C^0}\leq C_k$, $k\geq 1$ and such that 
\begin{equation}
\int_{0}^t\|\Ric(g(t'))\|_{C^0}\,dt'\leq \frac{1}{2n}.\label{hyp-ad-hoc-ric-3}
\end{equation}

 Then we have the following estimate,
\begin{equation}
\begin{split}\label{a-priori-est-Ric-pointwise}
\|\Ric(g(t))\|_{C^0_{\tau+2}}&\leq C\left(\|\Ric(g(s))\|_{C^0_{\tau+2}}+\int_s^t\|\Ric(g(t'))\|_{C^0}\,dt'\right),\quad  0\leq s < t<T,
\end{split}
\end{equation}
for some time-independent positive constant $C=C(n,g_b,\varepsilon)$.

Moreover, if $\alpha\in\left(0,\min\{1,n-2-\tau\}\right)$ then for any $\eta\in(0,1)$:
\begin{equation}
\begin{split}\label{a-priori-est-Ric-pointwise-holder}
\|\Ric(g(t))\|_{C^{0,\alpha}_{\tau+2}}&\leq C_{\eta}\left(\|\Ric(g(s))\|_{C^{0,\alpha}_{\tau+2}}+\int_s^t\|\Ric(g(t'))\|^{1-\eta}_{C^0}\,dt'\right),\quad  0\leq s< t<T.
\end{split}
\end{equation}
\end{lemma}

\begin{proof}
Similarly to the beginning of the proof of Lemma \ref{lemma-a-priori-wei-metr-wei}, we use Duhamel's formula to estimate the $C^0_{\tau+2}$ norm of the Ricci curvature along such a Ricci flow.

Using the Gaussian bounds (\ref{gauss-heat-ker}) from Theorem \ref{Gaussian-est-scal-heat-ker} with $D=D_0$ fixed once and for all, leads us to:
\begin{equation}
\begin{split}\label{ric-est-c0-duha}
&\left|\int_N\left<\mathcal{K}(x,t,\cdot,s),\Ric(g(s))\right>\,d\mu_{g(s)}\right|_{g(t)}\\
&\leq\frac{C}{(t-s)^{\frac{n}{2}}}\int_N\exp\left\{-\frac{d_{g(0)}^2(x,y)}{D_0(t-s)}\right\}|\Ric(g(s))|_{g(s)}(y)\,d\mu_{g(s)}(y)\\
&\leq C\|\Ric(g(s))\|_{C^0_{\tau+2}}\int_N(t-s)^{-\frac{n}{2}}\exp\left\{-\frac{d_{g_b}^2(x,y)}{2D_0(t-s)}\right\}\rho_{g_b}^{-\tau-2}(y)\,d\mu_{g_b}(y).
\end{split}
\end{equation}
Here we have used the fact that $g_b$ and $g(0)$ are uniformly equivalent (due to their $\varepsilon$ closeness).
By splitting the integral on the righthand side of (\ref{ric-est-c0-duha}) in two parts, whether $y\in B_{g_b}(p,d_{g_b}(p,x)/2)$ or not for some fixed point $p\in N$, one gets:
\begin{equation}
\int_N(t-s)^{-\frac{n}{2}}\exp\left\{-\frac{d_{g_b}^2(x,y)}{2D_0(t-s)}\right\}\rho_{g_b}^{-\tau-2}(y)\,d\mu_{g_b}(y)\leq \frac{C}{(d_{g_b}(p,x)+1)^{\tau+2}},\label{est-heat-ker-weight-pol}
\end{equation}
for some positive constant time-independent $C$ that may vary from line to line. Here, we have used the fact that $\tau+2<n$ in the second line. Concatenating estimates (\ref{ric-est-c0-duha}) and (\ref{est-heat-ker-weight-pol}) give the following estimates in terms of the initial data:
\begin{equation}
\begin{split}\label{ric-est-c0-duha-first-term}
\left\|\int_N\left<\mathcal{K}(x,t,\cdot,s),\Ric(g(s))\right>\,d\mu_{g(s)}\right\|_{C^0_{\tau+2}}\leq C\|\Ric(g(s))\|_{C^0_{\tau+2}},\quad t>s\geq 0.
\end{split}
\end{equation}
We proceed similarly to estimate the second term on the right hand side of \eqref{est-duha-dec-ric}:
\begin{equation}
\begin{split}\label{dealing-with-non-lin-Ric-wei}
&\left\|\int_s^t\int_N\left<\mathcal{K}(x,t,\cdot,t'),\Rm(g(t'))\ast\Ric(g(t'))\right>\,d\mu_{g(t')}\,dt'\right\|_{C^0_{\tau+2}}\\
&\leq C \int_s^t \|\Rm(g(t'))\ast\Ric(g(t'))\|_{C^0_{\tau+2}}\,ds\\
&\leq C\sup_{t'\in[s,t]}\|\Rm(g(t'))\|_{C^0_2}\int_s^t\|\Ric(g(t'))\|_{C^0_{\tau}}\,dt'\\
&\leq C\left(\|\Ric(g(s))\|_{C^0_{\tau+2}}+\int_s^t\|\Ric(g(t'))\|_{C^0}\,dt'\right).
\end{split}
\end{equation}
Here, we have used the fact that $g(t)$ is close to $g_b$ in the $C^{2}_{\tau}$-topology in the penultimate line which implies in particular that the curvature tensor $\Rm(g(s))$ decays as fast as $\rho_{g_b}^{-2}$ uniformly in time. Finally, [(\ref{dream-C-0-wei-int-ric-stat}), Lemma \ref{a-priori-wei-lemma-C0-Ric}] is invoked in the last line. Notice that we have not used the whole strength of the decay of the curvature tensor.
 
 The proof of the a priori estimate on the $C^{0,\alpha}_{\tau+2}$ norm of the Ricci tensor goes along the same lines as the $C^{0}_{\tau+2}$ a priori estimate we just proved by using difference quotients under the restriction on $\alpha\in(0,1)$ so that $\tau+2+\alpha<n.$ 
 
 Indeed, under this restriction on $\alpha$, we get schematically:
 \begin{equation}
 \begin{split}
\rho_{g_b}^{\tau+2+\alpha}[\Ric(g(t))]_{C^{\alpha}}\leq\,& C\| \rho_{g_b}^{\tau+2+\alpha}[\Ric(g(s))]_{C^{\alpha}}\|_{C^0}+C \int_s^t \|\Rm(g(t'))\ast[\Ric(g(t'))]_{C^{\alpha}}\|_{C^0_{\tau+2+\alpha}}\,dt'\\
&+C \int_s^t \|[\Rm(g(t'))]_{C^{\alpha}}\ast\Ric(g(t'))\|_{C^0_{\tau+2+\alpha}}\,dt'\\
\leq \,&C\|\Ric(g(s))\|_{C^{0,\alpha}_{\tau+2}}+C\int_s^t \|\Rm(g(t'))\ast[\Ric(g(t'))]_{C^{\alpha}}\|_{C^0_{\tau+2+\alpha}}\,dt'\\
&+C\sup_{t'\in[s,t]}\|\rho_{g_b}^{\alpha}[\Rm(g(t'))]_{C^{\alpha}}\|_{C^0_{2}}\int_s^t\|\Ric(g(t'))\|_{C^0_{\tau}}\,dt'\\
\leq\,& C\|\Ric(g(s))\|_{C^{0,\alpha}_{\tau+2}}+C\int_s^t \|\Rm(g(t'))\ast[\Ric(g(t'))]_{C^{\alpha}}\|_{C^0_{\tau+2+\alpha}}\,dt'\\
&+ C\left(\|\Ric(g(s))\|_{C^0_{\tau+2}}+\int_s^t\|\Ric(g(t'))\|_{C^0}\,dt'\right)\\
\leq\,&C\|\Ric(g(s))\|_{C^{0,\alpha}_{\tau+2}}+C\int_s^t \|\Rm(g(t'))\ast[\Ric(g(t'))]_{C^{\alpha}}\|_{C^0_{\tau+2+\alpha}}\,dt'\\
&+C\int_s^t\|\Ric(g(t'))\|_{C^0}\,dt'.
\label{est-c-alph-ric-wei}
\end{split}
\end{equation}
Here we have used Lemma \ref{lemma-a-priori-wei-metr-wei} in the third inequality.

In order to handle the last integral on the righthand side of (\ref{est-c-alph-ric-wei}), we proceed as follows: by interpolation together with H\"older inequality, 
\begin{equation}
\begin{split}\label{hol-covid-1}
\int_s^t \rho_{g_b}(x)^{\alpha}&[\Ric(g(t'))]_{C^{\alpha}}(x)\,dt'\leq C_{\alpha}\int_s^t\|\Ric(g(t'))\|_{C^0}^{1-\alpha}\|\nabla^{g(t')}\Ric(g(t'))\|^{\alpha}_{C^0_{1}}\,dt'\\
\leq\,&C_{\alpha}\left(\int_s^t\|\Ric(g(t'))\|_{C^0}\,dt'\right)^{1-\alpha}\cdot\left(\int_s^t\|\nabla^{g(t')}\Ric(g(t'))\|_{C^0_{1}}\,dt'\right)^{\alpha}\\
\end{split}
\end{equation}
for some time-independent positive constant $C_{\alpha}$. 
Now, by interpolation inequalities from Lemma \ref{inter-inequ-gag-nir}, we have for any $k\geq 2$:
\begin{equation*}
\|\nabla^{g(t')}\Ric(g(t'))\|_{C^0_{1}}\leq C_k\|\Ric(g(t'))\|_{C^0_{\frac{k}{k-1}}}^{1-\frac{1}{k}},
\end{equation*}
since we assume all the covariant derivatives of the curvature tensor to be bounded uniformly in time.

Therefore, by H\"older's inequality, if $k\geq 2$ is sufficiently large so that $1-\frac{1}{\tau}>\frac{1}{k}$,
\begin{equation}
\begin{split}
\int_s^t\|\nabla^{g(t)}\Ric(g(t'))\|_{C^0_{1}}\,dt'\leq\,& C_k\int_s^t\|\Ric(g(t'))\|_{C^0_{\frac{k}{k-1}}}^{1-\frac{1}{k}}\,dt'\\
\leq\,& C_k\left(\int_s^t\|\Ric(g(t'))\|_{C^0_{\tau}}\,dt'\right)^{\frac{1}{\tau}}\left(\int_s^t\|\Ric(g(t'))\|_{C^0}^{1-\frac{\tau}{k(\tau-1)}}\,dt'\right)^{1-\frac{1}{\tau}}.\label{hol-covid-2}
\end{split}
\end{equation}
Now, by Lemma \ref{lemma-a-priori-wei-metr-wei} together with (\ref{hol-covid-1}) and (\ref{hol-covid-2}), one gets:
\begin{equation}
\begin{split}\label{hol-covid-3}
\int_s^t \rho_{g_b}(x)^{\alpha}[\Ric(g(t'))]_{C^{\alpha}}(x)\,dt'\leq \,&C_{k,\alpha} \left(\|\Ric(g(s))\|_{C^0_{\tau}}+\int_s^t\|\Ric(g(t'))\|_{C^0}\,dt'\right)\\
&+C_{k,\alpha}\int_s^t\|\Ric(g(t'))\|_{C^0}^{1-\frac{\tau}{k(\tau-1)}}\,dt',
\end{split}
\end{equation}
which implies the expected estimate (\ref{a-priori-est-Ric-pointwise-holder}) and ends the proof.
\end{proof}

\begin{rk}\label{rk-work-harder}
The proof of Lemma \ref{a-priori-wei-lemma-C0-Ric} uses the Gaussian estimates established in Section \ref{sec-heat-ker} in an essential way. Indeed, by linearizing the evolution equation satisfied by the Ricci curvature along the Ricci flow at the metric $g_b$, given by Lemma (\ref{lemma-evo-eqn-Ric-Rm}), one gets schematically:
\begin{equation}
\begin{split}
\partial_t\Ric(g(t))=\Delta_{g_b}\Ric(g(t))+Q(h(t),\Ric(g(t))),
\end{split}
\end{equation}
where $Q(h(t),\Ric(g(t)))$ is a symmetric $2$-tensor on $N$ such that pointwise,
\begin{equation*}
\begin{split}
&\left|Q(h(t),\Ric(g(t)))\right|\leq \\
&C(n,g_b,\varepsilon)\left(|\Rm(g_b)|_{g_b}|\Ric(g(t))|_{g_b}+\sum_{k=0}^2\left|\nabla^{g_b,k}(g(t)-g_b)\right|_{g_b} \left|\nabla^{g_b,2-k}\Ric(g(t))\right|_{g_b}\right).
\end{split}
\end{equation*}
Observe that $g_b$ has non-negative Ricci curvature (since it is Ricci flat) and non-collapsed at all scales, i.e. $\vol_{g_b}B_{g_b}(x,r)\geq v(g_b)r^n$ for all radii $r>0$ and points $x\in N$ so that well-known Gaussian estimates established by Li-Yau \cite{Li-Yau-Sch} can be used.

Reasoning as in (\ref{dealing-with-non-lin-Ric-wei}) would have led us to estimate a priori a term such as $$\int_0^t\|(g(s)-g_b)\ast \nabla^{g_b,2}\Ric(g(s))\|_{C^0_{\tau+2}}\,ds.$$ In particular, we are left with estimating $$\sup_{s\in[0,T)}\|(g(s)-g_b)\|_{C^0_2}\int_0^t\|\nabla^{g_b,2}\Ric(g(s))\|_{C^0_{\tau}}\,ds$$ from above and uniformly in time. This is fine by Lemma \ref{lemma-a-priori-wei-metr-wei} as long as $\tau\geq 2$. Such a condition is always satisfied if $n\geq 6$ or if $n=5$ by restricting $\tau$. However, we are stuck in dimension $4$ since $\tau\in(1,2)$. 
\end{rk}

We end this section by establishing an a priori $C^{2,\alpha}_{\tau}$-estimate on $g(t)-g_b$, $t\in[0,T)$ as long as $g(t)\in B_{C^{2,\alpha}_{\tau}}(g_b,\varepsilon)$.

\begin{lemma}[A priori weighted $C^{2,\alpha}_{\tau}$ estimate]\label{a-priori-wei-lemma-Holder-met} Let $(N^n,g_b)$, $n\geq 4$, be an ALE Ricci flat metric. Let $\tau\in(\frac{n-2}{2},n-2)$, $\alpha\in(0,1)$ and let $(g(t))_{t\in[0,T)}$ be a solution to the Ricci flow in a neighborhood $B_{C^{2,\alpha}_{\tau}}(g_b,\varepsilon)$ with uniformly bounded covariant derivatives, i.e. $\sup_{t\in[0,T)}\|\nabla^{g(t)}\Rm(g(t))\|_{C^0}\leq C_k$, $k\geq 1$, and such that 
\begin{equation}
\int_{0}^t\|\Ric(g(t'))\|_{C^0}\,dt'\leq \frac{1}{2n}.\label{hyp-ad-hoc-ric-4}
\end{equation}
 Then for $\eta\in(0,1)$ and $\alpha\in\left(0,\min\left\{1,\tau-1,n-2-\tau\right\}\right)$:
\begin{equation}
\begin{split}\label{a-priori-est-met-hold-pointwise}
\|g(t)-g(s)\|_{C^{2,\alpha}_{\tau}}\leq C_{\eta}\left(\|\Ric(g(s))\|_{C^{0,\alpha}_{\tau+2}}+\int_s^t\|\Ric(g(t'))\|^{1-\eta}_{C^0}\,dt'\right),\quad  0\leq s< t<T,
\end{split}
\end{equation}
for some time-independent positive constant $C_{\eta}=C(n,g_b,\varepsilon,\eta)$.
\end{lemma}

\begin{proof}
By elliptic regularity, it is sufficient to establish an a priori bound on the $C^{0,\alpha}_{\tau+2}$-norm of $\Delta_{g_b}(g(t)-g_b)$ or equivalently $\Delta_{g(t)}(g(t)-g_b)$ since $g(t)$ is $\varepsilon$-close to $g_b$ in the $C^{2,\alpha}_{\tau}$-topology. 
Indeed, observe that if $h(t):=g(t)-g_b$, then schematically,
\begin{equation*}
\Delta_{g(t)}h(t)-\Delta_{g_b}h(t)=\nabla^{g_b,2}h(t)\ast h(t)+\nabla^{g_b}h(t)\ast \nabla^{g_b}h(t),
\end{equation*}
which implies 
\begin{equation*}
\|\Delta_{g_b}h(t)\|_{C^{0,\alpha}_{\tau+2}}\leq \|\Delta_{g(t)}h(t)\|_{C^{0,\alpha}_{\tau+2}}+c(n)\|h(t)\|_{C^{2,\alpha}_{\tau}}\|h(t)\|_{C^{1}}.
\end{equation*}

In particular, (\ref{beautiful-int-der-lap-bis}) from the proof of Lemma \ref{lemma-beauty-int-lap-met-ric}
 leads to the following estimate on the $C^0_{\tau+2}$ estimate:
\begin{equation}
\begin{split}\label{not-so-beautiful-int-der-lap}
\|\Delta_{g(t)}(g(t)-g(s))\|_{C^0_{\tau+2}}\leq\,&2\|\Ric(g(s))\|_{C^0_{\tau+2}}+2\|\Ric(g(t))\|_{C^0_{\tau+2}}\\
&+\sup_{t'\in[s,t]}\|\Rm(g(t'))\|_{C^0_{2}}\int_s^t\|\Ric(g(t'))\|_{C^0_{\tau}}\,dt'\\
&+\sum_{k=0}^2\int_s^t\|\nabla^{g(t'),k}(g(t')-g(s))\|_{C^0_{2}}\|\nabla^{g(t'),2-k}\Ric(g(t'))\|_{C^0_{\tau}}\,dt'.
\end{split}
\end{equation}
Invoking Lemmata \ref{lemma-a-priori-wei-metr-wei} and \ref{a-priori-wei-lemma-C0-Ric}, the proof of (\ref{a-priori-est-met-hold-pointwise}) with $\alpha=0$ ends \textbf{provided} $\tau\geq 2$. Indeed, in order to bound the norms $\|\nabla^{g(t'),k}(g(t')-g(s))\|_{C^0_{2}}$, $k=0,1,2$, uniformly in time, this requires either $n\geq 6$ or $n=5$ by restricting $\tau$ accordingly: this phenomenon echoes Remark \ref{rk-work-harder}.

Arguing as in (\ref{not-so-beautiful-int-der-lap}) leads to the desired expected estimate on $\|\Delta_{g(t)}(g(t)-g_b)\|_{C^{0,\alpha}_{\tau+2}}$ by using Lemma (\ref{lemma-beauty-int-lap-met-ric}) instead.
Indeed, we arrive at the following estimate:
\begin{equation}
\begin{split}
\|\Delta_{g(t)}(g(t)-g(s))\|_{C^{0,\alpha}_{\tau+2}}\leq\,&C\left(\|\Ric(g(s)\|_{C^{0,\alpha}_{\tau+2}}+\|\Ric(g(t))\|_{C^{0,\alpha}_{\tau+2}}\right)\\
&+C\sup_{s\leq t'\leq t}\|\Rm(g(t'))\|_{C^{0,\alpha}_{\tau+2}}\int_s^t\|\Ric(g(t'))\|_{C^{0,\alpha}}\,dt'\\
&+C\sup_{s\leq t'\leq t}\|\nabla^{g(t'),2}(g(t')-g(s))\|_{C^{0,\alpha}_{\tau+2}}\int_s^t\|\Ric(g(t'))\|_{C^{0,\alpha}}\,dt'\\
&+C\sup_{s\leq t'\leq t}\|\nabla^{g(t')}(g(t')-g(s))\|_{C^{0,\alpha}_{\tau+1}}\int_s^t\|\nabla^{g(t')}\Ric(g(t'))\|_{C_1^{0,\alpha}}\,dt'\\
\leq\,&C_{\eta}\left(\|\Ric(g(s))\|_{C^{2,\alpha}_{\tau+2}}+\int_s^t\|\Ric(g(t'))\|_{C^0}^{1-\eta}\,dt'\right)\\
&+C\int_s^t\|\nabla^{g(t')}\Ric(g(t'))\|_{C_1^{0,\alpha}}\,dt',\label{Too-many-est}
\end{split}
\end{equation}
if $\alpha\in(0,\min\{1,n-2-\tau\})$ and $\eta\in(0,1)$. Here we have invoked Lemma \ref{a-priori-wei-lemma-C0-Ric} in the second inequality which explains the restriction on $\alpha$. Now, if $\alpha\in (0,\min\{1,\tau-1\})$, a similar reasoning that led to (\ref{hol-covid-3}) in the proof of Lemma \ref{a-priori-wei-lemma-C0-Ric} gives an estimate of the last integral on the righthand side of (\ref{Too-many-est}) as follows if $\gamma\in\left(0,1-\frac{1+\alpha}{\tau}\right)$ and $x\in N$:
\begin{equation*}
\begin{split}
\rho_{g_b}(x)^{1+\alpha}\int_s^t&|[\nabla^{g(s)}\Ric(g(s))]_{C^{\alpha}}|(x)\,ds\leq\, C_{\alpha,\gamma}\int_s^t\|\Ric(g(t'))\|_{C^0_{\tau}}^{\frac{1+\alpha}{\tau}}\|\Ric(g(t'))\|^{1-\gamma-\frac{1+\alpha}{\tau}}_{C^0}\,dt'\\
\leq\,&C_{\alpha,\gamma}\left(\int_s^t\|\Ric(g(t'))\|_{C^0_{\tau}}\,dt'\right)^{\frac{1+\alpha}{\tau}}\cdot\left(\int_s^t\|\Ric(g(t'))\|^{1-\gamma\frac{\tau}{\tau-1-\alpha}}_{C^0}\,dt'\right)^{1-\frac{1+\alpha}{\tau}}.
\end{split}
\end{equation*}
The case $\alpha=0$ is handled similarly. 
Therefore, if $\eta\in (0,1)$, an application of Lemma \ref{lemma-a-priori-wei-metr-wei} leads us to,
\begin{equation*}
\begin{split}
\int_s^t\|\nabla^{g(t')}\Ric(g(t'))\|_{C_1^{0,\alpha}}\,dt'\leq& \,C_{\alpha,\eta} \left(\|\Ric(g(s))\|_{C^0_{\tau}}+\int_s^t\|\Ric(g(t'))\|_{C^0}\,dt'\right)\\
&+C_{\alpha,\eta}\int_s^t\|\Ric(g(t'))\|_{C^0}^{1-\eta}\,dt'\\
\leq&\,C_{\alpha,\eta} \left(\|\Ric(g(s))\|_{C^0_{\tau}}+\int_s^t\|\Ric(g(t'))\|^{1-\eta}_{C^0}\,dt'\right)
\end{split}
\end{equation*}
in case $\alpha\in(0,\min\{1,n-2-\tau,\tau-1\})$ and this ends the proof of the desired estimate.

\end{proof}

	\section{Stability of Ricci-flat ALE metrics}\label{section stability}
	\subsection{A stability result}~~\\
	
Let us now use the \L{}ojasiewicz inequality \eqref{loja L2 intro} to study the stability of Ricci-flat ALE metrics. We will say that a Ricci-flat ALE metric $(N^n,g_b)$ is \textit{stable} if it is a local maximizer of $\lambda_{\operatorname{ALE}}$, and unstable otherwise. From now on, we make the following assumption.\newline

In a $C^{2,\alpha}_\tau(g_b)$-neighborhood $B_{C^{2,\alpha}_{\tau}}(g_b,\varepsilon_{\L})$ of $g_b$, an $L^2$-\L{}ojasiewicz inequality is satisfied: for any metric $g$ in $B_{C^{2,\alpha}_{\tau}}(g_b,\varepsilon_{\L})$, we have
	        \begin{equation}\label{basic-ass-2}
	        |\lambda_{\operatorname{ALE}}(g)|^{2-\theta}\leq C \|\nabla \lambda_{\operatorname{ALE}}(g)\|_{L^2(g_b)}^{2},
\end{equation}
   for some $\theta\in(0,1)$.
   
	The purpose of this section is to give a proof of the following stability result:
	in the stable case, just like in \cite{Has-Sta}, \cite{Has-Mul} or \cite{Kro-Pet-Lp-stab}, the Ricci flow converges to a Ricci-flat metric at a particular rate. 

\begin{theo}\label{theo-dyn-stab-loc-max-lambda}
    Let $n\geq 4$ and $\tau\in(\frac{n-2}{2},n-2)$. Let $\alpha\in\left(0,\min\left\{1,\tau-1,n-2-\tau\right\}\right)$. Let $(N^n,g_b)$ be a stable Ricci-flat ALE metric such that Inequality \ref{basic-ass-2} holds on a neighborhood $B_{C^{2,\alpha}_{\tau}}(g_b,\varepsilon_{\L})$ with exponent $\theta\in(0,1)$. 
        
    Then for every $\varepsilon\in (0,\varepsilon_{\L})$, there exists $\delta>0$ such that the Ricci flow starting at any metric in $B_{C^{2,\alpha}_{\tau}}(g_b,\delta)$ stays in $B_{C^{2,\alpha}_{\tau}}(g_b,\varepsilon)$ and converges to a Ricci-flat metric $g_{\infty}$ in $B_{C^{2,\alpha}_{\tau}}(g_b,\varepsilon)$ in the $C^{2,\alpha'}_{\tau'}$-topology for any $\tau'\in(\frac{n-2}{2},\tau)$ and $\alpha'\in(0,\alpha)$.\\
    
     Moreover there exists a positive constant $C=C(n,g_b,\varepsilon,\theta)$ such that
    \begin{equation}
    \|g(t)-g_{\infty}\|_{C^{0}}\leq Ct^{-\frac{\theta}{2(1-\theta)}},\quad t\geq 1,
    \end{equation}
    and
        \begin{equation}
\int_0^{+\infty}\|\Ric(g(s))\|_{C^0}+\|\Ric(g(s))\|_{L^2}\,ds\leq C(n,g_b,\varepsilon,\theta)\left(\delta+|\lambda_{\operatorname{ALE}}(g(0))|^{\frac{\theta}{2}}\right). \label{inequ-c-0-thm-int-ric}\\
\end{equation}

\end{theo}
\begin{rk}
The convergence in which Theorem \ref{theo-dyn-stab-loc-max-lambda} takes place is reminiscent of the convergence result obtained in \cite[Theorem $5.1$]{Li-Yu-AF}. However, unless the background Ricci flat metric $(N^n,g_b)$ is flat, the solution $(N^n,g(t))_{t\geq 0}$ we provide is Type IIb, i.e. $$\limsup_{t\rightarrow+\infty}t|\Rm(g(t))|_{g(t)}=+\infty.$$ Because of this fact, it is unclear if one can get a convergence for all rescaled covariant derivatives of the metric as in \cite[Theorem $5.1$]{Li-Yu-AF}.
\end{rk}

To end this section, let us discuss the need for Gaussian bounds on the heat kernel given by Theorem \ref{Gaussian-est-scal-heat-ker}.  In Proposition \ref{claim-T-max-closed}, we start by proving an a priori $C^0$ estimate on the distance of a Ricci flow $(g(t))_{t\in[0,T)}$ to the origin $g_b$ lying in a small neighborhood $B_{C^{2,\alpha}_{\tau}}(g_b,\varepsilon)$.

At this stage, one is tempted to use appropriate interpolation inequalities together with the a priori $L^2$ bound on $\Ric(g(t))$, $t\in[0,T)$, established in Lemma \ref{stabilityALE} to get such a time-independent a priori $C^0_{\tau}$ bound on $g(t)-g_b$, $t\in[0,T)$. Let us apply the Gagliardo-Nirenberg interpolation inequalities \cite[Theorem $3.70$, Chapter $3$]{Aub-Boo} which can be adapted to $(N^n,g_b)$, to the tensor $\partial_tg$ with $p=\infty$, $j=0$, $r=\infty$, $q=2$ and $m\geq 1$ to get:
\begin{equation*}
\begin{split}
\rho_{g_b}(x)^{\tau}\|\partial_tg\|_{C^0(B_{g_b}(x,\rho_{g_b}(x)/4))}&\leq C(n,g_b)\|\partial_{t}g\|_{L^2}^{1-a}\cdot \|\partial_{t}g\|_{C^m_{\tau+2}}^{a}\cdot\rho_{g_b}(x)^{\tau-a(m+\tau+2)},
\end{split}
\end{equation*}
where $a\in(0,1)$ is such that $(2m+n)a=n$ and if $(g(t))_{t\in[0,T)}$ is assumed to stay in a fixed neighborhood $B_{C^{m+2,\alpha}_{\tau}}(g_b,\varepsilon)$ of $g_b$.
This estimate already shows that it asks for too much regularity of the solution.  Moreover, one needs to ensure the exponent $\tau-a(m+\tau+2)$ to be non-positive which constrains $\tau$ to lie below $\frac{n}{2}\cdot\frac{m+2}{m}$. 
Worse, the application of Lemma \ref{stabilityALE} requires the exponent $\eta$ (which equals $a$ here) to be strictly less than $\frac{\theta}{2-\theta}$, this in turn restricts the range of the exponent $\theta$ in the \L ojasiewicz inequality from Theorem \ref{dream-thm-loja-intro} which is uncheckable in general.

For all these reasons, we somewhat proceed more directly by using the heat kernel estimates from Section \ref{sec-heat-ker} via Duhamel's formula.
\subsection{Proof of Theorem \ref{theo-dyn-stab-loc-max-lambda}}\label{subsec-theo-proof-stab}~~\\

We first show that given $\varepsilon$ as in the statement of Theorem \ref{theo-dyn-stab-loc-max-lambda}, there exists $\delta>0$ such that the Ricci flow starting at any metric $g(0)$ in $B_{C^{2,\alpha}_{\tau}}(g_b,\delta)$ stays in $B_{C^{2,\alpha}_{\tau}}(g_b,\varepsilon)$ and exists for all time. 


Let $g(0)\in B_{C^{2,\alpha}_{\tau}}(g_b,\delta)$ with $\delta<\varepsilon$ to be constrained later and let $(N^n,g(t))_{t\in[0,T_{\operatorname{Shi}}]}$ be Shi's solution \cite{Shi-Def} with 
\begin{equation}
T_{\operatorname{Shi}}:=\frac{ \varepsilon(n)}{\sup_N|\Rm(g(0))|_{g(0)}}\geq \frac{ \varepsilon(n)}{C(n,g_b,\varepsilon)}=:T(n,g_b)>0,
\end{equation}
which exists since $g(0)$ is $\varepsilon$-close to $g_b$ in the $C^{2,\alpha}_{\tau}$-topology.
 
 We define the maximal time of existence of this solution to the Ricci flow with respect to the $C^{2,\alpha}_{\tau}$-topology as follows:
 \begin{equation*}
T_{\operatorname{max}}:=\sup\left\{T>0\,|\,\quad g(t)\in B_{C^{2,\alpha}_{\tau}}(g_b,\varepsilon),\quad \forall t\in[0,T)\right\}.
\end{equation*}
We start with the following proposition:
\begin{prop}\label{claim-T-max-pos}
There exists $\delta(n,g_b,\varepsilon)>0$ such that for $0<\delta\leq \delta(n,g_b,\varepsilon)$, one has $T_{\operatorname{max}}\geq T(n,g_b)>0$ and the set $$\left\{T>0\,|\,\quad \forall t\in[0,T),\quad g(t)\in B_{C^{2,\alpha}_{\tau}}(g_b,\varepsilon)\right\},$$
is open.
\end{prop}
\begin{rk}\label{rk-claim-T-max-pos}
The proof of Proposition \ref{claim-T-max-pos} does not use the stability of $(N^n,g_b)$.
\end{rk}
\begin{proof}[Proof of Proposition \ref{claim-T-max-pos}]
We only prove the fact that $T_{\operatorname{max}}\geq T(n,g_b)$, since the proof of the openness of the set of solutions is very similar and will therefore be omitted. 

According to Propositions \ref{prop-a-priori-C0-est-rm-ric} and \ref{prop-a-priori-L2-est-rm-ric}, 
\begin{equation}
\begin{split}
&\|\Rm(g(t))\|_{C^0}\leq 2\|\Rm(g(0))\|_{C^0}\leq C(n,g_b),\\  
&\|\Ric(g(t))\|_{C^0}\leq e^{C(n,g_b)\cdot t}\|\Ric(g(0))\|_{C^0}\leq C(n,g_b)\delta,
\end{split}
\end{equation}
 for $t\in [0,T_{\operatorname{Shi}}]$. In particular, if $t\in[0,T(n,g_b)]$,
 \begin{equation*}
\|g(t)-g_b\|_{C^0}\leq 2\int_0^t\|\Ric(g(s))\|_{C^0}\,ds+\|g(0)-g_b\|_{C^0}\leq C(n,g_b)\delta <\varepsilon,
\end{equation*}
if $\delta$ is chosen small enough. Moreover, we choose $\delta$ small enough so that $\int_0^t\|\Ric(g(s))\|_{C^0}\,ds\leq\frac{1}{2n}$. Thanks to this choice, Lemma \ref{lemma-short-time-est-curv} is applicable and gives us $\|\Ric(g(t))\|_{C^{0,\alpha}_{\tau+2}}\leq C(n,g_b)\delta$ for $t\in [0,T(n,g_b)]$. By integrating in time this inequality and by reducing $\delta$ once more if necessary, one gets for $t\in[0,T(n,g_b)]$ that
$\|g(t)-g_b\|_{C^0_{\tau}}<\varepsilon$ (we actually get a stronger decay in space but we do not use this fact as it will not be preserved for large time).

Finally, we invoke Lemma \ref{lemma-beauty-int-lap-met-ric} to get an a priori bound on the full $C^{2,\alpha}_{\tau}$-norm. Indeed, with $0=:s<t\leq T(n,g_b)$, Lemma \ref{lemma-beauty-int-lap-met-ric} leads to:
\begin{equation*}
\begin{split}
\|\Delta_{g(t)}(g(t)-g(0))\|_{C^{0,\alpha}_{\tau+2}}\leq\,& \|\Ric(g(t))\|_{C^{0,\alpha}_{\tau+2}}+\|\Ric(g(0))\|_{C^{0,\alpha}_{\tau+2}}\\
&+\int_0^t\|\Rm(g(t'))\|_{C^{0,\alpha}_{\tau+2}}\|\Ric(g(t'))\|_{C^{0,\alpha}}\,dt'\\
&+\int_0^t\|\nabla^{g(t')}\Ric(g(t'))\|_{C^{0,\alpha}_1}\|g(t')-g(0)\|_{C^{2,\alpha}_{\tau}}\,dt'\\
&+\int_0^t\|\Ric(g(t'))\|_{C^{0,\alpha}}\|g(t')-g(0)\|_{C^{2,\alpha}_{\tau}}\,dt'\\
\leq \,&C(n,g_b)\|\Ric(g(0))\|_{C^{0,\alpha}_{\tau+2}}+C(n,g_b)t\sup_{t'\in[0,t]}\|\Ric(g(t'))\|_{C^{0,\alpha}}\\
&+\int_0^t\|\nabla^{g(t')}\Ric(g(t'))\|_{C^{0,\alpha}_1}\|g(t')-g(0)\|_{C^{2,\alpha}_{\tau}}\,dt'\\
\leq\,& C(n,g_b)\|\Ric(g(0))\|_{C^{0,\alpha}_{\tau+2}}\\
&+\sup_{t'\in[0,t]}\|g(t')-g(0)\|_{C^{2,\alpha}_{\tau}}\int_0^t\|\nabla^{g(t')}\Ric(g(t'))\|_{C^{0,\alpha}_1}\,dt'.
\end{split}
\end{equation*}
Here we have used Lemma \ref{lemma-short-time-est-curv} in the second inequality. Observe that thanks to [(\ref{shi-est-ric-cov-tensor}), Proposition \ref{prop-a-priori-C0-est-rm-ric}], 
\begin{equation}
\begin{split}
\int_0^t\|\nabla^{g(t')}\Ric(g(t'))\|_{C^{0}_1}\,dt'\leq\,& C\int_0^t\frac{\|\Ric(g(0))\|_{C^0_{1}}}{\sqrt{t'}}\,dt'\\
\leq\,&C\|\Ric(g(0))\|_{C^0_{\tau+2}}\\
\leq\,& C\|g(0)-g_b\|_{C^{2,\alpha}_{\tau}},\label{covid-sras-ric-short-time-1-alpha}
\end{split}
\end{equation}
for some positive constant $C=C(n,g_b,\varepsilon)$ that may vary from line to line. 
 
 Similarly,
 \begin{equation}
\begin{split}
\int_0^t\rho_{g_b}^{1+\alpha}&[\nabla^{g(t')}\Ric(g(t'))]_{C^{\alpha}}\,dt'\\
\leq\,& C\int_0^t\rho_{g_b}^{1+\alpha-(\tau+2)}\|\nabla^{g(t')}\Ric(g(t'))\|_{C^0_{\tau+2}}^{1-\alpha}\|\nabla^{g(t'),2}\Ric(g(t'))\|_{C^0_{\tau+2}}^{\alpha}\,dt'\\
\leq\,&C\|\Ric(g(0))\|_{C^0_{\tau+2}}\int_0^t\frac{1}{t'^{\frac{1-\alpha}{2}+\alpha}}\,dt'\\
\leq\,&C\|\Ric(g(0))\|_{C^0_{\tau+2}}\\
\leq\,& C\|g(0)-g_b\|_{C^{2,\alpha}_{\tau}},\label{covid-sras-ric-short-time-2-alpha}
\end{split}
\end{equation}
 for some positive constant $C=C(n,g_b,\varepsilon)$ that may vary from line to line. Here we have used that $1+\alpha-(\tau+2)\leq 0$ since $0<\alpha<1$ and $\tau$ is positive.


Therefore, as an intermediate conclusion, we obtain:
\begin{equation}
\begin{split}
\|\Delta_{g(t)}(g(t)-g(0))\|_{C^{0,\alpha}_{\tau+2}}\leq\,&C(n,g_b,\varepsilon)\delta\left(1+\sup_{t'\in[0,t]}\|g(t')-g(0)\|_{C^{2,\alpha}_{\tau}}\right).\label{int-concl-sch-est-short-time}
\end{split}
\end{equation}

Now, observe that if $h(t):=g(t)-g(0)$, then schematically,
\begin{equation*}
\Delta_{g(t)}h(t)-\Delta_{g(0)}h(t)=\nabla^{g(0),2}h(t)\ast h(t)+\nabla^{g(0)}h(t)\ast \nabla^{g(0)}h(t),
\end{equation*}
which implies 
\begin{equation}
\|\Delta_{g(0)}h(t)\|_{C^{0,\alpha}_{\tau+2}}\leq \|\Delta_{g(t)}h(t)\|_{C^{0,\alpha}_{\tau+2}}+C(n,g_b)\|h(t)\|_{C^{2,\alpha}_{\tau}}\|h(t)\|_{C^{1}}.\label{equiv-appli-sch-est-ell}
\end{equation}

Plugging (\ref{int-concl-sch-est-short-time}) in the previous estimate (\ref{equiv-appli-sch-est-ell}) leads to:
\begin{equation*}
\|\Delta_{g(0)}h(t)\|_{C^{0,\alpha}_{\tau+2}}\leq C(n,g_b,\varepsilon)\delta\left(1+\sup_{t'\in[0,t]}\|h(t')\|_{C^{2,\alpha}_{\tau}}\right)+C(n,g_b)\|h(t)\|_{C^{2,\alpha}_{\tau}}\|h(t)\|_{C^{1}}.
\end{equation*}

According to elliptic Schauder estimate applied to the background initial metric $g(0)$, one gets $\|h(t)\|_{C^{2,\alpha}_{\tau}}\leq C(n,g_b,\varepsilon)\left(\|\Delta_{g(0)} h(t)\|_{C^{0,\alpha}_{\tau+2}}+\|h(t)\|_{C^0_{\tau}}\right)$ which leads to:
\begin{equation}
\|h(t)\|_{C^{2,\alpha}_{\tau}}\leq C(n,g_b,\varepsilon)\delta\left(1+\sup_{t'\in[0,t]}\|h(t')\|_{C^{2,\alpha}_{\tau}}\right)+C(n,g_b)\|h(t)\|_{C^{2,\alpha}_{\tau}}\|h(t)\|_{C^{1}}.\label{last-least-est-h-2-alpha}
\end{equation}

Notice that \eqref{covid-sras-ric-short-time-1-alpha} and \eqref{covid-sras-ric-short-time-2-alpha} imply a $C^{1,\alpha}$-estimate on $g(t)-g(0)$ by integrating over $[0,T(n,g_b)]$: $\sup_{t'\in[0,T(n,g_b)]}\|h(t')\|_{C^{1,\alpha}}\leq C(n,g_b,\varepsilon)\delta.$

We are then in a position to conclude since this previous fact combined with (\ref{last-least-est-h-2-alpha}) imply if $\delta\leq \delta(n,g_b,\varepsilon)$:
\begin{equation}
\|h(t)\|_{C^{2,\alpha}_{\tau}}\leq C(n,g_b,\varepsilon)\delta\left(1+\sup_{t'\in[0,t]}\|h(t')\|_{C^{2,\alpha}_{\tau}}\right),\quad t\in[0,T(n,g_b)].
\end{equation}
Considering the function in time $\sup_{t'\in[0,t]}\|h(t')\|_{C^{2,\alpha}_{\tau}}$, one gets, by choosing $\delta\leq \delta(n,g_b,\varepsilon)$ sufficiently small that $\|h(t)\|_{C^{2,\alpha}_{\tau}}\leq C(n,g_b,\varepsilon)\delta$ for all $t\in[0,T(n,g_b)]$ which by the triangular inequality leads to the expected result, i.e. $g(t)-g_b\in B_{C^{2,\alpha}_{\tau}}(g_b,\varepsilon)$ for all $t\in[0,T(n,g_b)]$.
\end{proof}
The next result consists in showing a priori estimates on the $C^{2,\alpha}_{\tau}$-norm of our solution which are time-independent. This is where the \L ojasiewciz inequality comes into play.
\begin{prop}\label{claim-T-max-closed}
The set $$\left\{T>0\,|\,\quad \forall t\in[0,T),\quad g(t)\in B_{C^{2,\alpha}_{\tau}}(g_b,\varepsilon)\right\},$$
is closed. More precisely, there exist time-independent positive constants $C=C(n,g_b,\varepsilon,\theta)$ and $C=C(n,g_b,\varepsilon,\theta,\eta)$, $\eta\in\left(0,\frac{\theta}{2-\theta}\right)$, such that for $t\in[0,T)$:
\begin{eqnarray}
\int_0^t\|\Ric(g(t'))\|_{C^0}\,dt'&\leq& C(n,g_b,\varepsilon,\theta)\left(\delta+|\lambda_{\operatorname{ALE}}(g(0))|^{\frac{\theta}{2}}\right), \label{inequ-c-0-prop-closed-int-ric}\\
\|g(t)-g(0)\|_{C^0_{\tau}}&\leq&C(n,g_b,\varepsilon,\theta)\left(\delta+|\lambda_{\operatorname{ALE}}(g(0))|^{\frac{\theta}{2}}\right),\label{inequ-c-0-tau-prop-closed}\\
\|g(t)-g(0)\|_{C^{2,\alpha}_{\tau}}&\leq& C(n,g_b,\varepsilon,\theta,\eta)\left(\delta^{1-\eta}+|\lambda_{\operatorname{ALE}}(g(0))|^{\frac{\theta}{2}(1+\eta)-\eta}\right).
\end{eqnarray}

\end{prop}

Before we prove Proposition \ref{claim-T-max-closed}, we establish one more crucial lemma which gives an a priori $L^2$ control on the distance of a Ricci flow from the origin given by a stable Ricci flat ALE metric:

\begin{lemma}[A priori $L^2$ estimate for the Ricci flow]\label{stabilityALE}
Let $(N^n,g_b)$ be a stable ALE Ricci flat metric such that Inequality \ref{basic-ass-2} holds on a neighborhood $B_{C^{2,\alpha}_{\tau}}(g_b,\varepsilon_{\L})$ with exponent $\theta\in(0,1)$. Let $(g(t))_{t\in[0,T)}$ be a solution to the Ricci flow in $B_{C^{2,\alpha}_{\tau}}(g_b,\varepsilon)$, $\varepsilon<\varepsilon_{\L}$. 
 Then, one has the following decay in time:
\begin{equation}\label{crucial-est-int-l2-norm-time-der}
\begin{split}
C^{-1}\int_s^{t}\|\Ric(g(t'))+&\nabla^{g(t'),2}f_{g(t')}\|_{L^2}^2\,dt'\leq|\lambda_{\operatorname{ALE}}(g(s))|\leq \frac{|\lambda_{\operatorname{ALE}}(g(0))|}{\left(1+C|\lambda_{\operatorname{ALE}}(g(0))|^{1-\theta}\cdot s\right)^{1+\beta_{\theta}}},\\
&\quad 0\leq s<t<T,\quad \beta_{\theta}:=\frac{\theta}{1-\theta},
\end{split}
\end{equation}
for some positive constant $C=C(n,g_b,\varepsilon,\theta)$ independent of time and where the $L^2$ norm is understood with respect to the weighted measure $e^{-f_{g(t')}}d\mu_{g(t')}$.

In particular, one has the following uniform energy bound if $0\leq s<t<T$: 
\begin{equation}\label{ric-hess-L2-control}
\int_s^{t}\|\Ric(g(t'))\|_{L^2}^2+\|\nabla^{g(t'),2}f_{g(t')}\|^2_{L^2}\,dt'\leq \frac{C|\lambda_{\operatorname{ALE}}(g(0))|}{\left(1+C|\lambda_{\operatorname{ALE}}(g(0))|^{1-\theta}\cdot s\right)^{1+\beta_{\theta}}}.
\end{equation}

Finally, if $\eta\in\left[0,\frac{\theta}{2-\theta}\right)$, one has
\begin{equation}
\begin{split}\label{est-l2-ric-hess-eta-a-priori}
\int_s^{t}\|\Ric(g(t'))\|_{L^2}^{1-\eta}+\|\nabla^{g(t'),2}f_{g(t')}\|_{L^2}^{1-\eta}\,dt'\leq C|\lambda_{\operatorname{ALE}}(g(s))|^{\frac{\theta}{2}(1+\eta)-\eta},\quad 0\leq s<t<T,
\end{split}
\end{equation}
for some positive constant $C=C(n,g_b,\varepsilon,\theta,\eta)$. 

\end{lemma}
\begin{rk}
	The right-hand side of \eqref{ric-hess-L2-control} also tends to $0$ as $s\to \infty$ for some negative values of $\theta$. Recall that up to a second order error, we have $ \lambda_{\operatorname{ALE}}(g) \sim \|g-g_b\|^2_{H^1_{n/2-1}} $ if $g-g_b\perp \ker_{L^2}L_{g_b}$. A natural question would be: can the Ricci flows considered in the proof of Theorem \ref{theo-dyn-stab-loc-max-lambda} still converge in $H^1_{n/2-1}$ if a sufficiently good $L^2$-\L{}ojasiewicz inequality does not hold?
\end{rk}
\begin{proof}
Observe first that by definition of the Ricci flow together with the first variation of $\lambda_{\operatorname{ALE}}$ computed in [(\ref{mono-lambda}), Proposition \ref{first-var-prop}], one has for $t>s\geq0$,
\begin{equation}
\begin{split}\label{int-L-2-inequ-lambda}
2\int_s^{t}\|\Ric(g(t'))+\nabla^{g(t'),2}f_{g(t')}\|_{L^2(e^{-f_{g(t)}}d\mu_{g(t)})}^2\,dt'&=\lambda_{\operatorname{ALE}}(g(t))-\lambda_{\operatorname{ALE}}(g(s))\\
&\leq -\lambda_{\operatorname{ALE}}(g(s))\leq |\lambda_{\operatorname{ALE}}(g(0))|.
\end{split}
\end{equation}
Here, we have used the fact that $0=\lambda_{\operatorname{ALE}}(g_b)\geq \lambda_{\operatorname{ALE}}(g)$ for all $g\in B_{C^{2,\alpha}_{\tau}}(g_b,\varepsilon)$ since $g_b$ is a local maximizer of $\lambda_{\operatorname{ALE}}$ by assumption. This proves the first part of (\ref{crucial-est-int-l2-norm-time-der}) since the measures $e^{-f_{g(t)}}d\mu_{g(t)}$ and $d\mu_{g(t)}$ are uniformly equivalent in time and space.

By using the \L{}ojasiewicz inequality for ALE metrics given by Inequality \ref{basic-ass-2} with exponent $\theta$, one gets:
\begin{equation}
\begin{split}\label{int-L-2-inequ-lambda}
\frac{d}{dt}(-\lambda_{\operatorname{ALE}}(g(t)))\leq -C(-\lambda_{\operatorname{ALE}}(g(t)))^{2-\theta},\quad t\in(0,T),
\end{split}
\end{equation}
for some positive constant $C$ independent of time.

Integrating this differential inequality leads to:

\begin{equation}
(-\lambda_{\operatorname{ALE}}(g(t)))\leq \frac{|\lambda_{\operatorname{ALE}}(g(0))|}{\left(1+C|\lambda_{\operatorname{ALE}}(g(0))|^{1-\theta}\cdot t\right)^{\frac{1}{1-\theta}}},\quad 0\leq t<T,\label{int-ineq-diff}
\end{equation}
for some positive constant $C$ independent of time. This gives us the full estimate (\ref{crucial-est-int-l2-norm-time-der}). The a priori $L^2$-estimate (\ref{ric-hess-L2-control}) follows by invoking the previous bound (\ref{crucial-est-int-l2-norm-time-der}) together with Proposition \ref{stabilityALE-RF-L2}.\\

Now, on the one hand, consider the function $t\rightarrow(-\lambda_{\operatorname{ALE}}(g(t)))^{\gamma}$ for some positive constant $\gamma$ to be constrained later. Observe as in \cite{Has-Mul} that if $\eta\in[0,\theta/(2-\theta))$ for some $\theta\in[0,1)$:
\begin{equation}
\begin{split}\label{diff-evo-equ-lambda-exp}
\frac{d}{dt}(-\lambda_{\operatorname{ALE}}(g(t)))^{\gamma}&=-\gamma (-\lambda_{\operatorname{ALE}}(g(t)))^{\gamma-1}\frac{d}{dt}\lambda_{\operatorname{ALE}}(g(t))\\
&=-2\gamma (-\lambda_{\operatorname{ALE}}(g(t)))^{\gamma-1}\|\Ric(g(t))+\nabla^{g(t),2}f_{g(t)}\|^2_{L^2}\\
&=-2\gamma (-\lambda_{\operatorname{ALE}}(g(t)))^{\gamma-1}\|\Ric(g(t))+\nabla^{g(t),2}f_{g(t)}\|_{L^2}^{(1+\eta)+(1-\eta)}\\
&\leq-C\gamma (-\lambda_{\operatorname{ALE}}(g(t)))^{\gamma-1+(1+\eta)\left(1-\frac{\theta}{2}\right)}\|\Ric(g(t))+\nabla^{g(t),2}f_{g(t)}\|_{L^2}^{1-\eta}\\
&=-C\gamma \|\Ric(g(t))+\nabla^{g(t),2}f_{g(t)}\|_{L^2}^{1-\eta},
\end{split}
\end{equation}
if $$\gamma:=\frac{\theta}{2}(1+\eta)-\eta>0.$$ Here we have used the \L{}ojasiewicz inequality for ALE metrics given by Inequality \ref{basic-ass-2} with exponent $\theta$ in the fourth line.

Integrating (\ref{diff-evo-equ-lambda-exp}) in time, one gets the expected result, i.e.
 \begin{equation}
\begin{split}
\int_s^{t}\|\Ric(g(t'))+\nabla^{g(t'),2}f_{g(t')}\|_{L^2}^{1-\eta}\,dt'
&\leq \frac{1}{C\gamma}|\lambda_{\operatorname{ALE}}(g(s))|^{\gamma},\quad t>s\geq0.
\end{split}
\end{equation}
This holds true whenever this function is differentiable, i.e. as long as $\lambda_{\operatorname{ALE}}(g(t))$ is not zero. Otherwise, one can adapt this reasoning by considering $t\rightarrow (-\lambda_{\operatorname{ALE}}(g(t))+\tilde{\varepsilon})^{\gamma}$ with $\tilde{\varepsilon}$ small and then we let $\tilde{\varepsilon}$ go to $0$.
This concludes the proof of (\ref{est-l2-ric-hess-eta-a-priori}) by invoking Proposition \ref{stabilityALE-RF-L2} once more.


\end{proof}

We are now in a good position to prove Proposition \ref{claim-T-max-closed}.

\begin{proof}[Proof of Proposition \ref{claim-T-max-closed}]
We proceed by establishing successive claims which lead to the expected result. It is sufficient to prove the corresponding estimates for large time, i.e. if $t\geq \frac{T(n,g_b)}{2}$.
\begin{claim}\label{claim-c0-a-priori}
For $t\in[0,T)$, (\ref{inequ-c-0-prop-closed-int-ric}) and the a priori $C^0_{\tau}$ estimate \eqref{inequ-c-0-tau-prop-closed} hold true.
\end{claim}
\begin{proof}[Proof of Claim \ref{claim-c0-a-priori}]
Let us check that the assumptions of Lemma \ref{lemma-a-priori-wei-metr-wei} are satisfied. It is sufficient to check condition (\ref{hyp-ad-hoc-ric-2}) since we assume $g(t)\in B_{C^{2,\alpha}_{\tau}}(g_b,\varepsilon)$. Since $g(t)\in B_{C^{2,\alpha}_{\tau}}(g_b,\varepsilon)$ for all $t\in[0,T)$ then if $r^2:=\frac{T(n,g_b)}{2}$, it is straightforward to check that $\int_{t-r^2}^t\|\R_{g(s)}\|_{C^0}\,ds\leq \frac{1}{2}$ by considering $\varepsilon$ small enough. This fact lets us to use Proposition \ref{coro-A-priori-L^2-C^0-est-RF} to get for $0<r^2:=\frac{T(n,g_b)}{2}<t$:
\begin{equation}
\begin{split}
\|\Ric(g(t))\|_{C^0}\leq\,& C(n,g_b,\varepsilon)\exp\left(c(n)\int_{t-r^2}^t\|\Rm(g(s))\|_{C^0}\,ds\right)\|\Ric(g(t-r^2)\|_{L^2}\\
\leq\,&C(n,g_b,\varepsilon)\exp\left(C(n,g_b,\varepsilon)r^2\right)\|\Ric(g(t-r^2)\|_{L^2}\\
\leq\,&C(n,g_b,\varepsilon)\|\Ric(g(t-r^2)\|_{L^2}.\label{intermed-c0-L2-claim}
\end{split}
\end{equation}
Now, \eqref{intermed-c0-L2-claim} and [\eqref{est-l2-ric-hess-eta-a-priori}, Lemma \ref{stabilityALE}] with $s:=0<\frac{T(n,g_b)}{2}\leq t<T$ and $\eta=0$ imply:
\begin{equation}
\begin{split}\label{intermed-c0-int-t-claim}
\int_0^t\|\Ric(g(t'))\|_{C^0}\,dt'\leq\,&\int_0^{\frac{T(n,g_b)}{2}}\|\Ric(g(t'))\|_{C^0}\,dt'+C(n,g_b,\varepsilon,\theta)|\lambda_{\operatorname{ALE}}(g(0))|^{\frac{\theta}{2}}\\
\leq\,&C(n,g_b)\|\Ric(g(0))\|_{C^0}+C(n,g_b,\varepsilon,\theta)|\lambda_{\operatorname{ALE}}(g(0))|^{\frac{\theta}{2}}\\
\leq\,&C(n,g_b,\varepsilon,\theta)\left(\delta+|\lambda_{\operatorname{ALE}}(g(0))|^{\frac{\theta}{2}}\right),
\end{split}
\end{equation}
where we have used Proposition \ref{prop-a-priori-L2-est-rm-ric} in the second inequality. In particular, if $\delta$ (and therefore $|\lambda_{\operatorname{ALE}}(g(0))|$) is chosen small enough, then \eqref{hyp-ad-hoc-ric-2} is satisfied.

Applying [\eqref{dream-C-0-wei-dist-met-stat}, Lemma \ref{lemma-a-priori-wei-metr-wei}] to $s:=0< t<T$ gives:
\begin{equation*}
\begin{split}
\|g(t)-g(0)\|_{C^0_{\tau}}\leq\,& C(n,g_b)\left(\|\Ric(g(0))\|_{C^0_{\tau+2}}+\int_0^t\|\Ric(g(s'))\|_{C^0}\,ds'\right)\\
\leq\,&C(n,g_b,\varepsilon,\theta)\left(\delta+|\lambda_{\operatorname{ALE}}(g(0))|^{\frac{\theta}{2}}\right),
\end{split}
\end{equation*}
where we have used (\ref{intermed-c0-int-t-claim}) in the second line. This concludes the proof of Claim \ref{claim-c0-a-priori}.

\end{proof}
\begin{claim}\label{claim-Ric-C-2-alpha}
For $t\in[0,T)$ and $\eta\in\left(0,\frac{\theta}{2-\theta}\right)$,
\begin{equation}\label{est-claim-Ric-C-2-alpha}
\|g(t)-g(0)\|_{C^{2,\alpha}_{\tau}}\leq C(n,g_b,\varepsilon,\theta,\eta)\left(\delta^{1-\eta}+|\lambda_{\operatorname{ALE}}(g(0))|^{\frac{\theta}{2}(1+\eta)-\eta}\right).
\end{equation}
\end{claim}
\begin{proof}[Proof of Claim \ref{claim-Ric-C-2-alpha}]
It is sufficient to prove this claim if $t\geq \frac{T(n,g_b)}{2}$.

In order to use Lemma \ref{a-priori-wei-lemma-Holder-met}, we first check that for each $k\geq 1$, there exists some time-independent positive constant $C_k=C(n,k,g_b,\varepsilon)$ such that $\|\nabla^{g(t),k}\Rm(g(t))\|_{C^0}\leq C_k$ for $t\geq\frac{T(n,g_b)}{2}$. Since $g(t)\in B_{C^{2,\alpha}_{\tau}}(g_b,\varepsilon)$ for all $t\in[0,T)$ by assumption, $\sup_{t\in[0,T)}\|\Rm(g(t))\|_{C^0}\leq C(n,g_b,\varepsilon)$. In particular, Proposition \ref{prop-a-priori-C0-est-rm-ric} implies that $\|\nabla^{g(t),k}\Rm(g(t))\|_{C^0}\leq C_k$ for $t=\frac{T(n,g_b)}{2}$. Therefore, Lemma \ref{lemma-shi-global} applies and guarantees that such uniform-in-time bounds on the covariant derivatives of the curvature tensor hold true for $t\geq \frac{T(n,g_b)}{2}$.\\

Lemma \ref{a-priori-wei-lemma-Holder-met} is therefore applicable: if $\eta\in(0,1)$,
\begin{equation}
\begin{split}
\|g(t)-g(0)\|_{C^{2,\alpha}_{\tau}}\leq\,& \left\|g(t)-g\left(\frac{T(n,g_b)}{2}\right)\right\|_{C^{2,\alpha}_{\tau}}+\left\|g\left(\frac{T(n,g_b)}{2}\right)-g(0)\right\|_{C^{2,\alpha}_{\tau}}\\
\leq\,&C\|g(0)-g_b\|_{C^{2,\alpha}_{\tau}}\\
&+C_{\eta}\left(\left\|\Ric\left(g\left(\frac{T(n,g_b)}{2}\right)\right)\right\|_{C^{0,\alpha}_{\tau+2}}+\int_{\frac{T(n,g_b)}{2}}^t\|\Ric(g(t'))\|^{1-\eta}_{C^0}\,dt'\right).
\end{split}
\end{equation}
Now, a similar argument based on [(\ref{est-l2-ric-hess-eta-a-priori}), Lemma \ref{stabilityALE}] as in the proof of Claim \ref{claim-c0-a-priori} leads to:
\begin{equation*}
\|g(t)-g(0)\|_{C^{2,\alpha}_{\tau}}\leq C\delta+ C_{\eta}\left(\delta^{1-\eta}+|\lambda_{\operatorname{ALE}}(g(0))|^{\frac{\theta}{2}(1+\eta)-\eta}\right),
\end{equation*}
under the restriction that $\eta\in\left(0,\frac{\theta}{2-\theta}\right)$. This ends the proof of Claim \ref{claim-Ric-C-2-alpha}.
\end{proof}
Claim \ref{claim-Ric-C-2-alpha} ends the proof of this proposition by choosing $\delta$ (and therefore $\lambda_{\operatorname{ALE}}(g(0))$) so small so that the righthand side of (\ref{est-claim-Ric-C-2-alpha}) is less than $\frac{\varepsilon}{2}$, say.
\end{proof}
Notice that Propositions \ref{claim-T-max-pos} and \ref{claim-T-max-closed} prove that $T_{\operatorname{max}}=+\infty$, i.e. the solution to the Ricci flow considered at the beginning of this Section \ref{subsec-theo-proof-stab} is immortal.
 
The last step to prove Theorem \ref{theo-dyn-stab-loc-max-lambda} is summarized in the following proposition.
\begin{prop}\label{prop-exi-lim-conv-rate}
Let $(N^n,g_b)$ be a stable Ricci-flat ALE metric such that Inequality \ref{basic-ass-2} holds on a neighborhood $B_{C^{2,\alpha}_{\tau}}(g_b,\varepsilon_{\L})$ with exponent $\theta\in(0,1)$. Let $(g(t))_{t\in[0,+\infty)}$ be a solution to the Ricci flow in $B_{C^{2,\alpha}_{\tau}}(g_b,\varepsilon)$, $\varepsilon<\varepsilon_{\L}$. Then there exists an ALE Ricci flat metric $g_{\infty}\in B_{C^{2,\alpha}_{\tau}}(g_b,\varepsilon)$ such that:
\begin{equation}
\|g(t)-g_{\infty}\|_{C^{0}}\leq Ct^{-\frac{\theta}{2(1-\theta)}},\quad t\geq 1,\label{est-dec-time-c0-prop-sharp}
\end{equation}
    for some positive constant $C=C(n,g_b,\varepsilon,\theta)$. 
\end{prop}

\begin{rk}
The decay in time for the $C^{2,\alpha'}_{\tau'}$-norms, $\tau'\in[0,\tau)$, $\alpha'\in[0,\alpha)$, can be obtained by interpolation between the $C^0$ and $C^{2,\alpha}_{\tau}$ norms together with \eqref{est-dec-time-c0-prop-sharp}. The proof of such interpolation inequalities for H\"older norms can be found for instance in \cite[Chapter $3$, Theorem $3.2.1$]{Kry-Ell-Par-Sch-Boo}.
\end{rk}
\begin{proof}[Proof of Proposition \ref{prop-exi-lim-conv-rate}]
Let us show first that there exists a unique limit metric $g_{\infty}\in B_{C^{2,\alpha}_{\tau}}(g_b,\varepsilon)$ in the $C^{2,\alpha'}_{\tau'}$-topology for any $\tau'\in[0,\tau)$ and $\alpha'\in[0,\alpha)$.

First of all, since we assume $(g(t))_{t\geq 0}$ to lie in $B_{C^{2,\alpha}_{\tau}}(g_b,\varepsilon)$, compactness of the embedding $C^{2,\alpha}_{\tau}\hookrightarrow C^{2,\alpha'}_{\tau'}$, for $\alpha'\in[0,\alpha)$, $\tau'\in[0,\tau)$ ensured by Lemma \ref{lemme-Chal-Cho-Bru}, there is a sequence $(g(t_i))_{i}$ converging to a $2$-tensor $g_{\infty}$ in $B_{C^{2,\alpha}_{\tau}}(g_b,\varepsilon)$ in the $C^{2,\alpha'}_{\tau'}$-topology for any $\tau'\in[0,\tau)$ and $\alpha'\in[0,\alpha)$. The limit $2$-tensor is a Riemannian metric since it is $\varepsilon$-close to $g_b$.\\

Now, if $t\geq t_i\geq 1$, \eqref{intermed-c0-L2-claim} in the proof of Claim \ref{claim-c0-a-priori} together with [\eqref{est-l2-ric-hess-eta-a-priori}, Lemma \ref{stabilityALE}] applied to $\eta=0$ gives the expected decay in time for the $C^0$ norm:
\begin{equation}
\|g(t)-g(t_i)\|_{C^0}\leq 2\int_{t_i}^t\|\Ric(g(s))\|_{C^0}\,ds\leq Ct_i^{-\frac{\theta}{2(1-\theta)}}.\label{real-limit-c0-RF}
\end{equation}
Inequality \eqref{real-limit-c0-RF} shows that the limit metric $g_{\infty}$ is unique, i.e. it does not depend on the sequence $(t_i)_i$. This also proves \eqref{est-dec-time-c0-prop-sharp}.\\

 In order to estimate the decay in time for the $C^0$-norm of the Ricci tensor, we make use of Proposition \ref{prop-ric-dec-large-time} with $s:=\frac{t}{2}$, $t:=t$ and $p>\frac{n}{\tau+2}$ to get:
 \begin{equation}
 \begin{split}\label{ric-est-dec-in-time-ti}
\|\Ric(g(t))\|_{C^0}\leq\,&  Ct^{-\frac{n}{2p}}\left\|\Ric\left(g\left(\frac{t}{2}\right)\right)\right\|_{L^p}+C\sup_{t'\in[0,T)}\|\Rm(g(t'))\|_{C^0}\int_{\frac{t}{2}}^{t}\|\Ric(g(t'))\|_{C^0}\,dt'\\
\leq\,&Ct^{-\frac{n}{2p}}\left\|\Ric\left(g\left(\frac{t}{2}\right)\right)\right\|_{L^p}+C\int_{\frac{t}{2}}^{t}\|\Ric(g(t'))\|_{C^0}\,dt'\\
\leq\,&Ct^{-\frac{n}{2p}}\left\|\Ric\left(g\left(\frac{t}{2}\right)\right)\right\|_{L^p}+Ct^{-\frac{\theta}{2(1-\theta)}},
\end{split}
\end{equation}
where $C=C(n,g_b,\varepsilon,\theta,p)$ is a time-independent positive constant that may vary from line to line. Now, if $p:=2>\frac{n}{\tau+2}$ by the choice of $\tau$ in the previous estimate, Proposition \ref{prop-a-priori-L2-est-rm-ric} together with [\eqref{est-l2-ric-hess-eta-a-priori}, Lemma \ref{stabilityALE}] lead to:
\begin{equation}
\begin{split}\label{est-l2-norm-dec-seq}
\left\|\Ric\left(g\left(\frac{t}{2}\right)\right)\right\|_{L^2}\leq\,& C\int_{\frac{t}{2}-1}^{\frac{t}{2}}\|\Ric(g(t'))\|_{L^2}\,dt'\\
\leq\,& Ct^{-\frac{\theta}{2(1-\theta)}},
\end{split}
\end{equation}
if $t$ is large enough for some positive constant $C=C(n,g_b,\varepsilon,\theta)$.

Therefore, (\ref{ric-est-dec-in-time-ti}) and (\ref{est-l2-norm-dec-seq}) imply the following decay for the $C^0$ norm of the Ricci tensor for $t_i$ large enough:
\begin{equation}\label{sweet-dec-ric-sec}
\|\Ric(g(t))\|_{C^0}\leq Ct^{-\frac{\theta}{2(1-\theta)}}.
\end{equation}
Moreover, (\ref{sweet-dec-ric-sec}) shows that $\Ric(g_{\infty}) =0$, i.e. $g_{\infty}$ is a Ricci flat metric in $B_{C^{2,\alpha}_{\tau}}(g_b,\varepsilon)$. In particular, by \cite{Ban-Kas-Nak}, $g_{\infty}$ is an ALE Ricci flat metric.

Now, [(\ref{inequ-c-0-tau-prop-closed}), Proposition \ref{claim-T-max-closed}] ensures that Lemma \ref{lemma-a-priori-wei-metr-wei} is applicable and for $t\geq t_i\geq 1$,
\begin{equation}\label{success-est-gt-ti}
\begin{split}
\|g(t)-g(t_i)\|_{C^0_{\tau}}\leq \,&C(n,g_b,\varepsilon)\left(\|\Ric(g(t_i))\|_{C^0_{\tau+2}}+\int_{t_i}^t\|\Ric(g(t'))\|_{C^0}\,dt'\right)\\
\leq\,&C(n,g_b,\varepsilon)\left(\|\Ric(g(t_i))\|_{C^0_{\tau+2}}+|\lambda_{\operatorname{ALE}}(g(t_i))|^{\frac{\theta}{2}}\right),
\end{split}
\end{equation}
where $C=C(n,g_b,\varepsilon,\theta)$ is a time-independent positive constant. Here we have used [\eqref{crucial-est-int-l2-norm-time-der}, Lemma \ref{stabilityALE}] in the fourth line.


\end{proof}

	\subsection{Evolution of the Bianchi form, the scalar curvature and the mass} \label{sec-Bia-for-evo}~~\\
	
In this section, we investigate the evolution of the Bianchi one-form together with that of the scalar curvature along a Ricci flow $(N^n,g(t))_{t\geq 0}$ satisfying the assumptions of Theorem \ref{theo-dyn-stab-loc-max-lambda}. Let us denote $g_{\infty}$ the limit ALE Ricci-flat metric of the solution $g(t)$ as $t$ tends to $+\infty$.

Recall that the Bianchi one-form $B(t)$ of $g(t)$ and $g_{\infty}$ is defined by

\begin{equation}
B(t)_i:=B(g(t),\nabla^{g_{\infty}},g(t))_i=\frac{1}{2}g(t)^{kl}\left(\nabla^{g_{\infty}}_kg(t)_{il}+\nabla^{g_{\infty}}_lg(t)_{ik}-\nabla_i^{g_{\infty}}g(t)_{kl}\right),\quad i=1,...,n.\label{lovely-bianchi-defn}
\end{equation}

  \begin{rk}
  Notice that the difference between the Bianchi one-form of $g(t)$ and $g_{\infty}$ introduced by Kotschwar following Hamilton and deTurck, and the Bianchi operator defined in (\ref{defn-bianchi-op}) is neglectible in the sense that if $h(t):=g(t)-g_{\infty}$,
  \begin{equation*}
  \begin{split}
B(t)_i-B_{g(t)}(h(t))_i=\,&B(g(t),\nabla^{g_{\infty}},g(t))_i-B_{g(t)}(g(t)-g_{\infty})_i\\
=\,&\frac{1}{2}g(t)^{kl}\left(\nabla^{g_{\infty}}_kg(t)_{il}+\nabla^{g_{\infty}}_lg(t)_{ik}-\nabla^{g_{\infty}}g(t)_{kl}\right)\\
&-\div_{g_{\infty}}(h(t))_i+\frac{1}{2}\nabla^{g_{\infty}}\tr_{g_{\infty}}h(t)_i\\
=\,&\frac{1}{2}g(t)^{kl}\left(\nabla^{g_{\infty}}_kh(t)_{il}+\nabla^{g_{\infty}}_lh(t)_{ik}-\nabla^{g_{\infty}}h(t)_{kl}\right)\\
&-\div_{g_{\infty}}(h(t))_i+\frac{1}{2}\nabla^{g_{\infty}}\tr_{g_{\infty}}h(t)_i\\
=\,&\frac{1}{2}\left(g(t)^{kl}-g_{\infty}^{kl}\right)\left(\nabla^{g_{\infty}}_kh(t)_{il}+\nabla^{g_{\infty}}_lh(t)_{ik}-\nabla^{g_{\infty}}h(t)_{kl}\right).
\end{split}
\end{equation*}
Therefore, $B(t)-B_{g(t)}(h(t))=g(t)^{-1}\ast h(t)\ast g_{\infty}^{-1}\ast \nabla^{g_{\infty}}h(t)$ is quadratic in $h$ and its derivatives.
  \end{rk}
	
		Next, we recall the evolution equation satisfied by the Bianchi gauge along the Ricci flow with a Ricci flat background metric $g_{\infty}$ given in \cite{Kot-Evo-Eqn-Bia}:
	\begin{lemma}[Kotschwar]
	Let $(N^n,g_{\infty})$ be a Ricci flat metric and let $(g(t))_{t\in[0,T)}$ be a Ricci flow on $N$. Then, the evolution equation satisfied by $B$ is schematically:
	\begin{equation}\label{evo-eqn-bianchi}
\partial_tB=g(t)^{-1}\ast B(t)\ast \Ric(g(t))+g(t)^{-1}\ast g(t)^{-1}\ast\nabla^{g_{\infty}}(g(t)-g_{\infty})\ast \Ric(g(t)),
\end{equation}
where, if $S$ and $T$ are two tensors on $N$, $S\ast T$ means some weighted sum of contractions of the tensor product of $S$ and $T$ with respect to the background metric $g_b$ with coefficients bounded by universal constants.

	\end{lemma}
	
	Our main result is the following proposition on the decay of the Bianchi one-form in space-time coordinates:
	\begin{prop}\label{Bianchi-C-0-est-prop}
	Let $n\geq 4$ and $\tau\in(\frac{n-2}{2},n-2)$. Let $\alpha\in\left(0,\min\left\{1,\tau-1,n-2-\tau\right\}\right)$. Let $(N^n,g_b)$ be a stable Ricci-flat ALE metric.  Let $(g(t))_{t\in[0,+\infty)}$ be a solution to the Ricci flow in $B_{C^{2,\alpha}_{\tau}}(g_b,\varepsilon)$, for some $\varepsilon>0$ which satisfies the conclusion of Theorem \ref {theo-dyn-stab-loc-max-lambda} and let $g_{\infty}\in B_{C^{2,\alpha}_{\tau}}(g_b,\varepsilon)$ be its limit Ricci-flat metric $g_{\infty}$.
  Then for any positive $\eta$ small enough and any $k\geq 0$, we have
  \begin{equation}
\|\nabla^{g_{\infty},k}B(t)\|_{C^0}\leq C_{\eta,k}\,t^{-\frac{\theta}{1-\theta}+\eta},\quad t\geq 1,\label{Bianchi-C-0-est}
\end{equation}
for some positive constant $C=C(n,g_b,\varepsilon,\theta,\eta,k)$. 

  	\end{prop}
		\begin{rk}
	In [Proposition 6.9, \cite{Kro-Pet-Lp-stab}], the decay of the Bianchi one-form obtained in Proposition \ref{Bianchi-C-0-est-prop} is improved if the perturbation lies in $L^p\cap L^{\infty}$ for $p<n$ small enough. 
	\end{rk}
	\begin{proof}
	
	 From now on, the symbol $\lesssim$ (respectively $\gtrsim$) means " less than or equal" (respectively " larger than or equal") up to a positive multiplicative constant that depends on $n$, $g_b$, $k$, $\alpha$ and $\varepsilon$ only. All the norms are understood with respect the metric $g_{\infty}$. For the sake of clarity, we omit the dependence of the Levi-Civita connection on $g_{\infty}$. By (\ref{evo-eqn-bianchi}), the norm of the Bianchi gauge $B(t)$ satisfies in the weak sense:
	\begin{equation*}
\partial_t|B|\gtrsim -|\Ric(g(t))||B(t)|-|\nabla h(t)||\Ric(g(t))|,\quad t>0.
\end{equation*}
	By Gr\"onwall's inequality, one gets pointwise in space:
	\begin{equation*}
	\begin{split}
|B(t)|\gtrsim\, &|B(s)|\exp\left\{-\int_s^t|\Ric(g(t'))|\,dt'\right\}\\
&-\int_s^t\exp\left\{-\int_{t'}^{t}|\Ric(g(s'))|\,ds'\right\}|\nabla h(t')||\Ric(g(t'))|\,dt',\,\,t>s>0.
\end{split}
\end{equation*}
This implies, once we let $t$ tend to $+\infty$:

\begin{equation}
	\begin{split}
|B(s)|\lesssim \,&\int_s^{+\infty}\exp\left\{\int_{s}^{t'}|\Ric(g(s'))|\,ds'\right\}|\nabla h(t')||\Ric(g(t'))|\,dt'\\
\lesssim\,&\exp\left\{\int_{s}^{+\infty}|\Ric(g(s'))|\,ds'\right\}\int_s^{+\infty}|\nabla h(t')||\Ric(g(t'))|\,dt'\\
\lesssim\,& \int_s^{+\infty}|\nabla h(t')||\Ric(g(t'))|\,dt'.\label{step-1-bianchi}
\end{split}
\end{equation}
where $C=C(n,g_b,\varepsilon,\theta)$ is a time-independent positive constant. Here, we have used Proposition \ref{claim-T-max-closed} in the last inequality.

Now, according to the conclusion of Theorem \ref {theo-dyn-stab-loc-max-lambda}, $\|g(t)-g_{\infty}\|_{C^0}\leq Ct^{-\frac{\theta}{2(1-\theta)}}$ if $t\geq 1$. Since for any $k\geq 0$ and $t\geq 1$, $|\nabla^{g(t),k}\Rm(g(t))|_{g(t)}\leq C_k$, for some time-independent positive constant $C_k$ by Lemma \ref{lemma-shi-global}, standard interpolation inequalities 
show that the same decay in time holds true for higher covariant derivatives of $g(t)-g_{\infty}$ at the expense of an arbitrary small error: for any positive $\eta$ small enough and any $k\geq 0$, one gets $\|\nabla^{k}(g(t)-g_{\infty})\|_{C^0}\leq Ct^{-\frac{\theta}{2(1-\theta)}+\eta}$. 

Therefore, (\ref{step-1-bianchi}) leads to the expected estimate (\ref{Bianchi-C-0-est}) for $s\geq 2$:
\begin{equation*}
	\begin{split}
\|B(s)\|_{C^0}\lesssim\,& s^{-\frac{\theta}{2(1-\theta)}+\eta}\int_s^{+\infty}\|\Ric(g(t'))\|_{C^0}\,dt'\\
\lesssim\,&C_{\eta}s^{-\frac{\theta}{2(1-\theta)}+\eta}|\lambda_{\operatorname{ALE}}(g(s-1))|^{\frac{\theta}{2}}\\
\lesssim\,&C_{\eta}s^{-\frac{\theta}{(1-\theta)}+\eta},
\end{split}
\end{equation*}
for any $\eta>0$ small enough. Here we have used Proposition \ref{coro-A-priori-L^2-C^0-est-RF} with $0<r^2:=1<t$ together with [\eqref{est-l2-ric-hess-eta-a-priori}, Lemma \ref{stabilityALE}]. The corresponding decay in time on the covariant derivatives of $B(t)$ are obtained by interpolation.


\end{proof}

We continue by estimating the decay in time of the $C^0$-norm of the scalar curvature. In case the scalar curvature at time $t=0$ is integrable, we also show this property is preserved: this echoes the result of \cite[Theorem $1$, part B]{Dai-Ma-Mass}. Our proof is based on heat kernel estimates established in Section \ref{sec-heat-ker}.
\begin{prop}
Let $n\geq 4$ and $\tau\in(\frac{n-2}{2},n-2)$. Let $\alpha\in\left(0,\min\left\{1,\tau-1,n-2-\tau\right\}\right)$. Let $(N^n,g_b)$ be a stable Ricci-flat ALE metric such that Inequality \ref{basic-ass-2} holds on a neighborhood $B_{C^{2,\alpha}_{\tau}}(g_b,\varepsilon_{\L})$ with exponent $\theta\in(0,1)$. Let $(g(t))_{t\in[0,+\infty)}$ be a solution to the Ricci flow in $B_{C^{2,\alpha}_{\tau}}(g_b,\varepsilon)$, $\varepsilon< \varepsilon_{\L}$, which satisfies the conclusion of Theorem \ref {theo-dyn-stab-loc-max-lambda} and let $g_{\infty}\in B_{C^{2,\alpha}_{\tau}}(g_b,\varepsilon)$ be its limit Ricci-flat metric.

Then,
\begin{equation}
\|\R_{g(t)}\|_{C^0}\leq C\left(t^{-\frac{\theta}{2(1-\theta)}-\frac{n}{4}}+t^{-\frac{1}{1-\theta}}\right),\quad t\geq 1, \label{est-C-0-norm-scal}
\end{equation}
for some time-independent positive constant $C=C(n,g_b,\varepsilon,\theta)$.
Moreover, if $\R_{g(0)}\in L^1$ then $\R_{g(t)}\in L^1$ for every $t\geq 0$ and,
\begin{equation}
\|\R_{g(t)}\|_{L^1}\leq C\left(\|\R_{g(0)}\|_{L^1}+|\lambda_{\operatorname{ALE}}(g(0))|\right),\quad t\geq 0,\label{est-L-1-norm-scal}
\end{equation}
some time-independent positive constant $C=C(n,g_b,\varepsilon,\theta)$. \end{prop}
\begin{proof}
According to the evolution equation satisfied by the scalar curvature given in [(\ref{evo-eqn-scal-for}), Lemma \ref{lemma-evo-eqn-Ric-Rm}] together with Duhamel's formula:
\begin{equation}
\begin{split}
\R_{g(t)}(x)=\int_N&K(x,t,y,s)\R_{g(0)}(y)\,d\mu_{g(0)}(y)\\
&+\int_0^t\int_NK(x,t,y,s)|\Ric(g(s))|^2_{g(s)}(y)\,d\mu_{g(s)}(y)ds, \quad t\geq 0, \quad x\in N.\label{duha-scal-curv-form}
\end{split}
\end{equation}
In particular, if $2\leq s:=\frac{t}{2}<t$, Theorem \ref{Gaussian-est-scal-heat-ker} ensures that,
\begin{equation}
\begin{split}\label{prep-est-C-0-scal}
\|\R_{g(t)}\|_{C^0}\leq\,& \sup_{x\in N}\left\|K\left(x,t,\cdot,\frac{t}{2}\right)\right\|_{L^2}\|\R_{g\left(\frac{t}{2}\right)}\|_{L^2}+\int_{\frac{t}{2}}^t\|\Ric(g(t'))\|^2_{C^0}\,dt'\\
\leq\,&\frac{C}{t^{\frac{n}{4}}}\|\R_{g\left(\frac{t}{2}\right)}\|_{L^2}+\int_{\frac{t}{2}}^t\|\Ric(g(t'))\|^2_{C^0}\,dt'\\
\leq\,&\frac{C}{t^{\frac{n}{4}}}\|\R_{g\left(\frac{t}{2}\right)}\|_{L^2}+C\int_{\frac{t}{4}}^{t}\|\Ric(g(t'))\|^2_{L^2}\,dt'\\
\leq\,&\frac{C}{t^{\frac{n}{4}}}\|\R_{g\left(\frac{t}{2}\right)}\|_{L^2}+Ct^{-\frac{1}{1-\theta}}.
\end{split}
\end{equation}
The second line is justified by Proposition \ref{coro-A-priori-L^2-C^0-est-RF} applied to $0<r^2:=1<t$ which implies that $\frac{t}{2}-1\geq \frac{t}{4}$ since $t\geq 4$. The last line is obtained thanks to [(\ref{ric-hess-L2-control}), Lemma \ref{stabilityALE}].

Now, Proposition \ref{prop-a-priori-L2-est-rm-ric} together with [(\ref{est-l2-ric-hess-eta-a-priori}), Lemma \ref{stabilityALE}] lead to:
\begin{equation}
\begin{split}\label{prep-bis-est-C-0-scal}
\|\R_{g\left(\frac{t}{2}\right)}\|_{L^2}\leq \,&n \left\|\Ric\left(g\left(\frac{t}{2}\right)\right)\right\|_{L^2}\\
\leq\,& C\int_{\frac{t}{2}-1}^{\frac{t}{2}}\|\Ric(g(t'))\|_{L^2}\,dt'\\
\leq\,& Ct^{-\frac{\theta}{2(1-\theta)}},
\end{split}
\end{equation}
if $t\geq 2$ is large enough, for some time-independent positive constant $C=C(n,g_b,\varepsilon,\theta)$. Estimates (\ref{prep-est-C-0-scal}) and (\ref{prep-bis-est-C-0-scal}) end the proof of (\ref{est-C-0-norm-scal}).\\

In order to prove (\ref{est-L-1-norm-scal}), we observe that by considering (\ref{duha-scal-curv-form}), one obtains:
\begin{equation*}
\begin{split}
\|\R_{g(t)}\|_{L^1}\leq\,& \int_N|\R_{g(t)}(x)|\,d\mu_{g(t)}(x)\\
\leq\,&\sup_{y\in N}\|K(\cdot,t,y,0)\|_{L^1}\|\R_{g(0)}\|_{L^1}\\
&+\int_0^t\sup_{y\in N}\|K(\cdot,t,y,s)\|_{L^1}\int_N|\Ric(g(s))|_{g(s)}^2(y)\,d\mu_{g(s)}(y)ds\\
\leq\,&e^{\int_0^t\|\R_{g(s)}\|_{C^0}\,ds}\left(\|\R_{g(0)}\|_{L^1}+\int_0^t\|\Ric(g(s))\|_{L^2}^2\,ds\right)\\
\leq\,& Ce^{\int_0^{+\infty}\|\R_{g(s)}\|_{C^0}\,ds}\left(\|\R_{g(0)}\|_{L^1}+|\lambda_{\operatorname{ALE}}(g(0))|\,ds\right)\\
\leq\,&C\left(\|\R_{g(0)}\|_{L^1}+|\lambda_{\operatorname{ALE}}(g(0))|\,ds\right),
\end{split}
\end{equation*}
for some time-dependent positive constant $C=C(n,g_b,\varepsilon,\theta)$.

Here, we have used [(\ref{L^1-est-fct}), Proposition \ref{L^1-bound-heat-kernel-fct}] in the third inequality together with [(\ref{ric-hess-L2-control}), Lemma \ref{stabilityALE}] in the penultimate line. Finally, [(\ref{inequ-c-0-thm-int-ric}), Theorem \ref{theo-dyn-stab-loc-max-lambda}] is invoked in the last line since $\|\R_{g(s)}\|_{C^0}\leq n\|\Ric(g(s))\|_{C^0}$. We could have alternatively used (\ref{est-C-0-norm-scal}) to justify this step.

\end{proof}

The next proposition establishes a link between the integral of the scalar curvature and the ADM-mass along a suitable solution to the Ricci flow whenever the scalar curvature is assumed to decay fast enough initially. 
\begin{prop}\label{prop-id-scal-int-mass-init}
Let $n\geq 4$ and $\tau\in(\frac{n-2}{2},n-2)$. Let $\alpha\in\left(0,\min\left\{1,\tau-1,n-2-\tau\right\}\right)$. Let $(N^n,g_b)$ be a stable Ricci-flat ALE manifold. Let $(g(t))_{t\in[0,+\infty)}$ be a solution to the Ricci flow in $B_{C^{2,\alpha}_{\tau}}(g_b,\varepsilon)$ for some $\varepsilon>0$ satisfying the conclusion of Theorem \ref {theo-dyn-stab-loc-max-lambda} and let $g_{\infty}\in B_{C^{2,\alpha}_{\tau}}(g_b,\varepsilon)$ be its limit Ricci-flat metric.

If $\R_{g(0)}=O(\rho_{g_b}^{-q})$ for some $q>n$ then
\begin{equation}
m_{\operatorname{ADM}}(g(0))=\lim_{t\rightarrow+\infty}\int_N\R_{g(t)}\,d\mu_{g(t)}\label{mass-id-t-infty},
\end{equation}
where $m_{\operatorname{ADM}}(g(0))$ is the mass of the initial metric $g(0)$ defined in (\ref{def-mass}).

\end{prop}
\begin{rk}
Formula (\ref{mass-id-t-infty}) is inspired by the work \cite{Li-Yu-AF} where the same identity is proved in the context of asymptotically flat immortal Ricci flows with non-negative scalar curvature on $\RR^n$.
\end{rk}

\begin{proof}
Let us prove formula (\ref{mass-id-t-infty}). Recall by linearizing the scalar curvature around the limit metric $g_{\infty}$ based on [(\ref{lem-lin-equ-scal-first-var}), Lemma \ref{lem-lin-equ-Ric-first-var}] that:
\begin{equation*}
\begin{split}
\Bigg|\R_{g(t)}&-\div_{g(t)}\left(\div_{g(t)}(g(t)-g_{\infty})-\nabla^{g_{\infty}}\tr_{g_{\infty}}(g(t)-g_{\infty})\right)\Bigg|\\
&\leq C(n,g_{\infty})\left(|g(t)-g_{\infty}|\left|\nabla^{g_{\infty},2}(g(t)-g_{\infty})\right|+\left|\nabla^{g_{\infty}}(g(t)-g_{\infty})\right|^2\right).
\end{split}
\end{equation*}
In particular, by the definition of the $ADM$ mass introduced in (\ref{def-mass}), one gets by integration by parts that:
\begin{equation}
\begin{split}\label{comp-mass-int-R}
\Bigg|m_{\operatorname{ADM}}(g(t))&-\int_N\R_{g(t)}\,d_{\mu_{g(t)}}\Bigg|\\
 \leq\,&C\int_N\left(|g(t)-g_{\infty}|\left|\nabla^{g_{\infty},2}(g(t)-g_{\infty})\right|+\left|\nabla^{g_{\infty}}(g(t)-g_{\infty})\right|^2\right)\,d\mu_{g(t)}\\
\leq\,&C\|g(t)-g_{\infty}\|_{C^{2}_{\tau'}}^2,
\end{split}
\end{equation}
for $\tau'\in(0,\tau)$ such that $2\tau'+2>n$ which justifies the second inequality ensuring the integrability of $\rho_{g_b}^{-2\tau'-2}$.

As a first conclusion, Theorem \ref{theo-dyn-stab-loc-max-lambda} implies that the righthand side of (\ref{comp-mass-int-R}) converges to $0$ as $t$ tends to $+\infty$.

If $\R_{g(0)}=O(\rho_{g_b}^{-q})$ for some $q>n$ then the mass $m_{\operatorname{ADM}}(g(\cdot))$ is constant in time, according to \cite[Theorem $1$]{Dai-Ma-Mass}. Therefore, one gets the expected result:
\begin{equation*}
\begin{split}
0=&\lim_{t\rightarrow+\infty}\left(m_{\operatorname{ADM}}(g(t))-\int_N\R_{g(t)}\,d_{\mu_{g(t)}}\right)\\
=&m_{\operatorname{ADM}}(g(0))-\lim_{t\rightarrow+\infty}\int_N\R_{g(t)}\,d_{\mu_{g(t)}}.
\end{split}
\end{equation*}

\end{proof}

In \cite[Theorem $2$]{Dai-Ma-Mass}, it is shown that if $(g(t))_{t\geq 0}$ is a solution to the Ricci flow which is ALE for each time and such that the scalar curvature is integrable, then in case this solution converges to a limit ALE metric $g_{\infty}$ in the $C^{1,\alpha}_{\tau}$-topology, $\tau>\frac{n-2}{2}$, one has $$\lim_{t\rightarrow+\infty}m_{\operatorname{ADM}}(g(t))=m_{\operatorname{ADM}}(g_{\infty}),$$ \textit{provided} that,
\begin{equation}
\lim_{t\rightarrow+\infty}\|\R_{g(t)}-\R_{g_{\infty}}\|_{L^1}=0.\label{ass-dai-ma}
\end{equation}
In the setting of Proposition \ref{prop-id-scal-int-mass-init}, the convergence (\ref{ass-dai-ma}) is unlikely to hold true and it explains the restriction on $\tau$ to lie below $n-2$.

Indeed, on the contrary, one would get $m_{\operatorname{ADM}}(g(0))=m_{\operatorname{ADM}}(g_{\infty})=0$ since $g_{\infty}$ is an ALE Ricci flat metric. On the other hand, there are small perturbations in the $C^{2,\alpha}_{\tau}$-topology of any given ALE Ricci flat metric with non-zero ADM mass: see for instance the discussion at the beginning of \cite[Section $3$]{Der-Ozu-Lam}. This leads to a contradiction for stable Ricci-flat ALE spaces (such as all known examples).

	\subsection{A discussion on previous stability results}~~\\
	
	In this section, we compare our stability result given by Theorem \ref{theo-dyn-stab-loc-max-lambda} together with its proof and previous results on the same topic.
	
	In \cite[Theorem $1.4$]{Sch-Sch-Sim} and \cite[Theorem $1.2$]{app-scal}, stability of Euclidean space $(\RR^n,g_{e})$ along the DeTurck Ricci flow is investigated  for small perturbations in $L^{\infty}$ of the Euclidean metric $g_{e}$ lying in $L^p$, $p\geq 1$. In \cite{Sch-Sch-Sim}, their argument is based on the Lyapunov function $t\in\RR_+\rightarrow\int_{\RR^n}|\tilde{g}(t)-g_{e}|^2\,d\mu_{g_e}$ where $(\tilde{g}(t))_{t>0}$ denotes a solution to DeTurck Ricci flow whereas \cite{app-scal} uses estimates based on Duhamel's formula. Both methods lead to the same solution $(g(t))_{t>0}$ which is of Type III: in \cite[Theorem $1.2$]{app-scal}, the curvature and its covariant derivatives satisfy
$|\nabla^{\tilde{g}(t),k}\Rm(\tilde{g}(t))|_{\tilde{g}(t)}=O(t^{-\frac{n}{2p}-1-\frac{k}{2}})$ for all $k\geq 0$ uniformly in space.

Moreover, these two articles obtain a corresponding solution $(g(t))_{t>0}$ to the Ricci flow by pulling back the solution to the DeTurck Ricci flow by a one-parameter family of diffeomorphisms $(\psi_t)_{t>0}$, i.e. $g(t):=\psi_t^*\tilde{g}(t)$. This family of diffeomorphisms is generated by minus the evolving Bianchi one-form $B(g(t),\nabla^{g_{e}},g(t))$ introduced in (\ref{lovely-bianchi-defn}): 
\begin{equation*}
\partial_t\psi_t:=-B(g(t),\nabla^{g_{e}},g(t))\circ\psi_t,\quad \psi_t|_{t=0}=\operatorname{Id}_{\RR^n}.
\end{equation*}
The article \cite[Theorem $9.2$]{Sch-Sch-Sim} shows in detail that there exists a limit diffeomorphism $\psi_{\infty}$ which is the pointwise limit of $\psi_t$ as $t$ tends to $+\infty$ and such that $g(t)$ converges to $\psi_{\infty}^*g_e$ in $C^{k}_{loc}$. 

In contrast, Theorem \ref{theo-dyn-stab-loc-max-lambda} gives a Type IIb solution $(g(t))_{t>0}$ to the Ricci flow and our functional $\lambda_{\operatorname{ALE}}$ plays the role of a Lyapunov functional. The use of the DeTurck Ricci flow is delicate here in the sense that the limit metric of the flow is not known a priori.

 However, the work \cite{Der-Kro} proves that one can choose a time-dependent gauge to show that a suitable DeTurck Ricci flow converges to a nearby ALE Ricci flat metric in the setting of integrable and linearly stable ALE Ricci flat metrics. In \cite{Der-Kro}, the convergence is shown to hold in the $L^2\cap L^{\infty}$ topology. No convergence rate in time was obtained. This rate has been established in the work \cite{Kro-Pet-Lp-stab} in case the background Ricci flat metric carries a parallel spinor.
 
 Coming back to the setting and notations of Theorem \ref{theo-dyn-stab-loc-max-lambda}, the possibility of making sense of the DeTurck Ricci flow with respect to the background limit metric $g_{\infty}$ is not taken for granted in general as the decay in time [(\ref{Bianchi-C-0-est}), Proposition \ref{Bianchi-C-0-est-prop}] suggests. Indeed, in order to justify the existence of a diffeomorphism $\psi_{\infty}$ of $N$ as the limit of the one-parameter family of diffeomorphisms generated by $B(g(t),\nabla^{g_{\infty}},g(t))$, [(\ref{Bianchi-C-0-est}), Proposition \ref{Bianchi-C-0-est-prop}] asks the \L ojasiewicz exponent $\theta$ to be close to $1$ enough. More precisely, if $\theta>\frac{1}{2}$ then it is not too difficult to show the existence of $\psi_{\infty}$ and as a result, the corresponding DeTurck Ricci flow converges to $\psi_{\infty}^{\ast}g_{\infty}$ in $C^k_{loc}$, $k\geq 0$. Whether this convergence holds true in weighted H\"older norms is not straightforward and in light of the proof of Proposition \ref{Bianchi-C-0-est-prop}, it would require an even more stringent condition on the exponent $\theta$.\\
 
 Finally, let us notice that in case the background ALE Ricci flat metric $g_b$ is integrable and linearly stable (which includes the Euclidean metric on $\RR^n$) then \cite[Proposition $7.15$]{Der-Ozu-Lam} shows that $\theta=\frac{2\tau-n}{2\tau-(n-2)}$. In particular, the criterion $\theta>\frac{1}{2}$ is equivalent to $\tau>\frac{n+2}{2}$. As $\tau$ is constrained to lie in $(\frac{n}{2}-1,n-2)$, the condition $\tau>\frac{n+2}{2}$ is non empty for $n>6$.

\section{Instability of Ricci-flat ALE metrics}\label{section instability}

Similarly to \cite{Has-Sta}, we prove that there exists an ancient solution to the Ricci flow starting at any unstable ALE Ricci-flat metric $(N^n,g_b)$ that satisfies Inequality \ref{basic-ass-2} and which flows away in the $C^{2,\alpha}_{\tau}$ topology.
\subsection{A priori energy estimates in the unstable case}~~\\

Let us start by proving that a Ricci flow starting in a given $C^{2,\alpha}_{\tau}$-neighborhood of an unstable ALE Ricci flat metric, where a suitable \L{}ojasiewicz inequality is satisfied, escapes a possibly larger neighborhood in finite time.
\begin{lemma}\label{instabdefT}
Let $(N^n,g_b)$ be an unstable ALE Ricci flat metric such that Inequality \ref{basic-ass-2} holds on a neighborhood $B_{C^{2,\alpha}_{\tau}}(g_b,\varepsilon_{\L})$ with exponent $\theta\in(0,1)$. Let $g_0$ be a metric in $B_{C^{2,\alpha}_\tau}(g_b,\varepsilon)$, $\varepsilon<\varepsilon_{\L}$, and let $(g(t))_{t\in[0,T)}$ be Shi's solution to the Ricci flow starting at $g_0$. If $\lambda_{\operatorname{ALE}}(g_0)>0$ then there exists a finite positive time $T(g_0)$ defined by
    \begin{equation*}
    T(g_0) = \min\left\{t>0\,|\,\forall s\in[0,t),\,\|g(s)-g_b\|_{C^{2,\alpha}_\tau}<\varepsilon\,\text{and $\|g(t)-g_b\|_{C^{2,\alpha}_\tau}=\varepsilon$}\right\}.
    \end{equation*}
   Moreover, there exists a positive constant $C=C(n,g_b,\varepsilon,\theta)$ such that $T(g_0)<\frac{1}{C\lambda_{\operatorname{ALE}}^{1-\theta}(g_0)}$ and,
   \begin{equation}
        \lambda_{\operatorname{ALE}}(g(t))\geq \frac{\lambda_{\operatorname{ALE}}(g(s))}{(1 - C\lambda_{\operatorname{ALE}}^{1-\theta}(g(s))(t-s))^\frac{1}{1-\theta}},\quad 0\leq s<t\leq T(g_0).\label{ineg lambdaALE positif}
    \end{equation}
Finally, there exists $\delta(n,g_b,\varepsilon)>0$ such that if $\delta\in(0,\delta(n,g_b,\varepsilon))$, and $g_0\in B_{C^{2,\alpha}_{\tau}}(g_b,\delta)$ then $T(g_0)>T(n,g_b)>0$.

\end{lemma}

\begin{proof}
  The fact that $T(g_0)$ exists and is positive comes from Proposition \ref{claim-T-max-pos} together with Remark \ref{rk-claim-T-max-pos}. The existence of $\delta(n,g_b,\varepsilon)>0$ such that if $\delta\in(0,\delta(n,g_b,\varepsilon))$, and $g_0\in B_{C^{2,\alpha}_{\tau}}(g_b,\delta)$ then $T(g_0)>T(n,g_b)>0$ is also due to that same Proposition \ref{claim-T-max-pos}.

    Let us now prove that $T(g_0)$ is finite. As long as $g(t)$ stays in $B_{C^{2,\alpha}_\tau}(g_b,\varepsilon)$, we can integrate the \L{}ojasiewicz inequality for ALE metrics given by Inequality \ref{basic-ass-2} with exponent $\theta$ to get
    
    \begin{equation}
\begin{split}\label{int-L-2-inequ-lambda}
\frac{d}{dt}\lambda_{\operatorname{ALE}}(g(t))\geq C\lambda_{\operatorname{ALE}}(g(t))^{2-\theta},\quad t\in(0,T(g_0)),
\end{split}
\end{equation}
for some positive constant $C$ independent of time.

Integrating this differential inequality leads to:

\begin{equation}
\lambda_{\operatorname{ALE}}(g(t))\geq \frac{\lambda_{\operatorname{ALE}}(g_0)}{\left(1-C\lambda_{\operatorname{ALE}}(g_0)^{1-\theta}\cdot t\right)^{\frac{1}{1-\theta}}},\quad 0\leq t\leq T(g_0),\label{covid-20-lambda-unst}
\end{equation}
for some positive constant $C$ independent of time. This gives us the full estimate (\ref{ineg lambdaALE positif}).

    The righthand side of (\ref{covid-20-lambda-unst}) blows up as $t$ tends to $\frac{1}{C\lambda_{\operatorname{ALE}}^{1-\theta}(g_0)}$, and since $\lambda_{\operatorname{ALE}}(g(t))$ depends continuously on $\|g(t)-g_b\|_{C^{2,\alpha}_\tau}$ as recalled in Section \ref{sec-def-lambda}, we have $T(g_0)<\frac{1}{C\lambda_{\operatorname{ALE}}^{1-\theta}(g_0)}$.

\end{proof}

The following lemma is the analogous statement to Lemma \ref{stabilityALE} for unstable ALE Ricci flat metrics.
\begin{lemma}[A priori $L^2$ estimate for the Ricci flow : positive $\lambda_{\operatorname{ALE}}$]\label{instability-ALE}
Let $(N^n,g_b)$ be an unstable ALE Ricci flat metric such that Inequality \ref{basic-ass-2} holds on a neighborhood $B_{C^{2,\alpha}_{\tau}}(g_b,\varepsilon_{\L})$ with exponent $\theta\in(0,1)$. Let $g_0$ be a metric in $B_{C^{2,\alpha}_\tau}(g_b,\varepsilon)$, $\varepsilon<\varepsilon_{\L}$, such that $\lambda_{\operatorname{ALE}}(g_0)>0$ and let $T(g_0)>0$ be defined as in Lemma \ref{instabdefT}. Then one has the following estimate:
\begin{equation}\label{crucial-est-int-l2-norm-time-insta}
\begin{split}
\int_0^{T(g_0)}\|\Ric(g(t))\|_{L^2}^2+&\|\nabla^{g(t),2}f_{g(t)}\|_{L^2}^2\,dt\leq C\lambda_{\operatorname{ALE}}(g(T(g_0))),
\end{split}
\end{equation}
for some time-independent positive constant $C=C(n,g_b,\varepsilon,\theta)$.

Finally, if $\eta\in\left[0,\frac{\theta}{2-\theta}\right)$, one has
\begin{equation}
\begin{split}\label{est-l2-ric-hess-eta-a-priori-insta}
\int_0^{T(g_0)}\|\Ric(g(t))\|_{L^2}^{1-\eta}+\|\nabla^{g(t),2}f_{g(t)}\|_{L^2}^{1-\eta}\,dt\leq C\lambda_{\operatorname{ALE}}(g(T(g_0)))^{\frac{\theta}{2}(1+\eta)-\eta},
\end{split}
\end{equation}
for some positive constant $C=C(n,g_b,\varepsilon,\theta,\eta)$. 
\end{lemma}

\begin{proof}

  Observe first that by definition of the Ricci flow together with the first variation of $\lambda_{\operatorname{ALE}}$ computed in [(\ref{mono-lambda}), Proposition \ref{first-var-prop}],
\begin{equation}
\begin{split}\label{int-L-2-inequ-lambda}
2\int_0^{T(g_0)}\|\Ric(g(t))+\nabla^{g(t),2}f_{g(t)}\|_{L^2}^2\,dt&=\lambda_{\operatorname{ALE}}(g(T(g_0)))-\lambda_{\operatorname{ALE}}(g_0)\\
&\leq \lambda_{\operatorname{ALE}}(g(T(g_0))).
\end{split}
\end{equation}
This proves (\ref{crucial-est-int-l2-norm-time-insta}) once we invoke Proposition \ref{stabilityALE-RF-L2}. 

Now, on the one hand, consider the function $t\rightarrow\lambda_{\operatorname{ALE}}(g(t))^{\gamma}$ for some positive constant $\gamma$ to be constrained later. Observe as in \cite{Has-Mul} that if $\eta\in[0,\theta/(2-\theta))$ for some $\theta\in[0,1)$:
\begin{equation}
\begin{split}\label{diff-evo-equ-lambda-exp-insta}
\frac{d}{dt}\lambda_{\operatorname{ALE}}(g(t))^{\gamma}&=\gamma \lambda_{\operatorname{ALE}}(g(t))^{\gamma-1}\frac{d}{dt}\lambda_{\operatorname{ALE}}(g(t))\\
&=2\gamma \lambda_{\operatorname{ALE}}(g(t))^{\gamma-1}\|\Ric(g(t))+\nabla^{g(t),2}f_{g(t)}\|^2_{L^2}\\
&=2\gamma \lambda_{\operatorname{ALE}}(g(t))^{\gamma-1}\|\Ric(g(t))+\nabla^{g(t),2}f_{g(t)}\|_{L^2}^{(1+\eta)+(1-\eta)}\\
&\geq C\gamma \lambda_{\operatorname{ALE}}(g(t))^{\gamma-1+(1+\eta)\left(1-\frac{\theta}{2}\right)}\|\Ric(g(t))+\nabla^{g(t),2}f_{g(t)}\|_{L^2}^{1-\eta}\\
&=C\gamma \|\Ric(g(t))+\nabla^{g(t),2}f_{g(t)}\|_{L^2}^{1-\eta},
\end{split}
\end{equation}
if $$\gamma:=\frac{\theta}{2}(1+\eta)-\eta>0.$$ Here we have used the \L{}ojasiewicz inequality for ALE metrics given by Inequality \ref{basic-ass-2} in the fourth line.

Integrating (\ref{diff-evo-equ-lambda-exp-insta}) in time, one gets the expected result, i.e.
 \begin{equation}
\begin{split}
\int_0^{T(g_0)}\|\Ric(g(t))+\nabla^{g(t),2}f_{g(t)}\|_{L^2}^{1-\eta}\,dt
&\leq \frac{1}{C\gamma}\lambda_{\operatorname{ALE}}(g(T(g_0)))^{\gamma}.
\end{split}
\end{equation}
This concludes the proof of (\ref{est-l2-ric-hess-eta-a-priori-insta}) by invoking Proposition \ref{stabilityALE-RF-L2} once more.


\end{proof}

\subsection{A digression on $\lambda_{\operatorname{ALE}}$}~~\\

By Perelman's work \cite{Per-Ent}, recall that for a closed Riemannian manifold $(M^n,g)$, 
\begin{equation*}
\|\Ric(g) + \nabla^{g,2}{f}\|^2_{L^2(e^{-f_g}d\mu_g)} \geq \lambda(g)^{2},
\end{equation*}
where $\lambda(g)$ denotes Perelman's energy.

In the non-compact situation, a similar inequality holds generally if the metric stays at bounded $C^{2,\alpha}_{\tau}$-distance for $\tau>n-2$.

\begin{prop}\label{inégalité globale evolution lambda ALE}
    Let $\tau>n-2$ and $\alpha\in(0,1)$. Let $(N^n,g_b)$ be a Ricci flat ALE metric. Then we have the following inequality for the $\lambda_{\operatorname{ALE}}$-functional: for any metric $g\in B_{C^{2,\alpha}_{\tau}}(g_b,\varepsilon)$, there exist $ \delta= \delta(\tau)>2 $ and $C = C(n,\varepsilon)$ such that we have
    \begin{equation}
        \|\Ric(g) + \nabla^{g,2}{f_g}\|_{L^2(e^{-f_g}d\mu_g)}^2 \geq C|\lambda_{\operatorname{ALE}}(g)|^{\delta},\label{inequ-pere-lambda-ALE}
    \end{equation}
    for $\delta = \delta(\tau)>2$.
\end{prop}

\begin{rk}
We underline the fact that Proposition \ref{inégalité globale evolution lambda ALE} is not useful in our setting since the convergence rate $\tau$ is assumed to be larger than $n-2$. However, we decide to keep it here since it has its own interest.
\end{rk}

\begin{proof}
    Let us consider a Ricci-flat ALE manifold $(N^n,g_b)$, and a Riemannian metric $g$ satisfying $\|g-g_b\|_{C^{2,\alpha}_\tau}\leq \varepsilon $ for some $\tau>n-2$ and $\alpha\in(0,1)$.
    
    By assumption and the definition of $f_g\in C^{2,\alpha}_\tau$, see Section \ref{subsec-lambda},  we have the following decay rates, $ |\R_g|+|\Delta_gf_g| \leq C(n)C\big( (1+\rho_g)^{-\tau-2} \big)$. Let $\gamma\in\left(\frac{n}{\tau+2},1\right)$, so that $\int_N(1+\rho_g)^{-\gamma(\tau+2)}e^{-f_g}d\mu_g<+\infty$, and define $ \delta := \frac{2-\gamma}{1-\gamma} $. Then we have for $c=c(\gamma,\tau,\delta)>0$,
    \begin{align*}
        \int_N|\Ric(g)+\nabla^{g,2}{f_g}|^2_g & \,e^{-f_g}d\mu_g \\
        \geq& \frac{1}{n}\int_N |R_g + \Delta_g f_g|^2\,e^{-f_g}d\mu_g\\
        =&\frac{1}{n}\int_N |R_g + \Delta_g f_g|^2(1+\rho_g)^{\gamma(\tau+2)}\big((1+\rho_g)^{-\gamma(\tau+2)}\,e^{-f_g}d\mu_g\big)\\
        \geq & \frac{c}{n} \Big\{\int_N |R_g + \Delta_g f_g|^{\frac{2}{\delta}}(1+\rho_g)^{\frac{\gamma(\tau+2)}{\delta}}\big((1+\rho_g)^{-\gamma(\tau+2)}\,e^{-f_g}d\mu_g\big)\Big\}^{\delta}\\
        \geq & \frac{c}{n} \Big\{\int_N |R_g + \Delta_g f_g|_g^{\frac{2}{\delta}+\gamma(1-\frac{1}{\delta})}\Big[|R_g + \Delta_g f_g|(1+\rho_g)^{(\tau+2)}\Big]^{-\gamma(1-\frac{1}{\delta})}\,e^{-f_g}d\mu_g\Big\}^{\delta}\\
        \geq & \frac{c}{n}\big[C(n)C\big]^{\gamma(\frac{1}{\delta}-1)} \Big\{\int_N |R_g + \Delta_g f_g|_g\,e^{-f_g}d\mu_g\Big\}^{\delta}\\
        \geq & \frac{c}{n}\big[C(n)C\big]^{\gamma(\frac{1}{\delta}-1)} \Big|\int_N \big(R_g + \Delta_g f_g\big)\,e^{-f_g}d\mu_g\Big|^{\delta}\\
        =&\frac{c}{n}\big[C(n)C\big]^{\gamma(\frac{1}{\delta}-1)} \Big|\frac{1}{2}\int_N\big( R_g + |\nabla_g f_g|^2_g\big)\,e^{-f_g}d\mu_g\Big|^{\delta},
    \end{align*}
    by Jensen inequality in the fourth line and thanks to \eqref{eqdeffg} in the last line. Here, we have used the fact that $\frac{2}{\delta}-\gamma\left(\frac{1}{\delta}-1\right)=1$.
\end{proof}

    Notice that Inequality (\ref{inequ-pere-lambda-ALE}) with a constant $C$ independent of $\|g-g_b\|_{C^{2,\alpha}_\tau}$ would imply that for any Ricci flow starting with a positive $\lambda_{\operatorname{ALE}}$, $\lambda_{\operatorname{ALE}}$ blows up in finite time. This is not true in general. 
    
    Indeed, Feldman, Ilmanen and Knopf \cite{Fel-Ilm-Kno} have constructed complete expanding gradient K\"ahler-Ricci solitons on the total space of the tautological line bundles $L^{-k}$, $k>n$ over $\mathbb{CP}^{n-1}$. These solutions on $L^{-k}$ are $U(n)$-invariant and are asymptotic to the cone $C(\mathbb{S}^{2n-1}/\mathbb{Z}_k)$ endowed with the Euclidean metric $\frac{1}{2}i\partial\overline{\partial}\, |\cdot|^2$, where $\mathbb{Z}_k$ acts on $\mathbb{C}^n$ diagonally. The curvature tensor of these solitons decay exponentially fast to $0$ at infinity, in particular these metrics are ALE and their mass vanish. On the other hand, the scalar curvature of these metrics is positive everywhere. 
    
    If $g_{\operatorname{FIK}}$ denotes one of the Feldman-Ilmanen-Knopf metrics, then $\lambda_{\operatorname{ALE}}^0(g_{\operatorname{FIK}})>0$.
     Now, for any $t>0$, we have $\lambda_{\operatorname{ALE}}^0(g_{\operatorname{FIK}}(t)) = t^{\frac{n}{2}- 1}\lambda_{\operatorname{ALE}}^0(g_{\operatorname{FIK}})>0$ as we noticed in \cite[Remark 8.3]{Der-Ozu-Lam}, since $g_{\operatorname{FIK}}$ is an expanding soliton. Again, this is in contrast with the compact situation where a Ricci-flow starting at a metric with positive $\lambda$-functional necessarily develops a finite-time singularity.
     
\subsection{An unstability result}~~\\

In this section, the main result is the corresponding statement to Theorem \ref{theo-dyn-stab-loc-max-lambda} in the presence of an unstable ALE Ricci flat metric.
\begin{theo}[unstable case - ancient Ricci flows coming out of Ricci-flat ALE spaces]
  Let $n\geq 4$ and $\tau\in(\frac{n-2}{2},n-2)$. Let $\alpha\in\left(0,\min\left\{1,\tau-1,n-2-\tau\right\}\right)$. 
    Let $(N^n,g_b)$ be an unstable Ricci-flat ALE metric such that Inequality \ref{basic-ass-2} holds on a neighborhood $B_{C^{2,\alpha}_{\tau}}(g_b,\varepsilon_{\L})$ with exponent $\theta\in(0,1)$.
      
      Then there exists a non Ricci-flat ancient solution to the Ricci flow $(g_{\infty}(t))_{t\in (-\infty,0]}$ with positive scalar curvature that converges in $C^{2,\alpha}_\tau$ to $g_b$ as $t\to -\infty$. More precisely, if $\eta\in\left(0,\frac{\theta}{2(1-\theta)}\right)$ and $\varepsilon\in(0,\varepsilon_{\L})$ then there exists a positive constant $C=C(n,g_b,\varepsilon,\theta,\eta)$ such that for $t\leq 0$,
      \begin{equation}
      \|g_{\infty}(t)-g_b\|_{C^{2,\alpha}_{\tau}}\leq \frac{C}{(1+C|t|)^{\frac{\theta}{2(1-\theta)}-\eta}}.
      \end{equation}
      Finally, the solution $g_{\infty}(t)$ is ALE of order $\tau$ uniformly in time in the sense that,
      \begin{equation}
      \rho_{g_b}^{2+\tau+k}|\nabla^{g_{\infty},k}\Rm(g_{\infty}(t))|_{g_{\infty}(t)}\leq C_k,\quad k\geq 0, \quad t\leq 0.\label{ALE-unif-unst}
      \end{equation}
\end{theo}
\begin{proof}
	Let us consider $g_b$ a Ricci-flat ALE metric which is not a maximizer of $\lambda_{\operatorname{ALE}}$. This means that there exists a sequence $(g_i)_i$ of metrics converging to $g_b$ in $C^{2,\alpha}_\tau$ with $\lambda_{\operatorname{ALE}}(g_i)>0$. By Lemma \ref{instabdefT}, the Ricci flow $(g_i(t))_t$ starting at $g_i(0) = g_i$ leaves $B_{C^{2,\alpha}_\tau}(g_b,\varepsilon)$ in finite time and we can consider the first time $T_i:=T(g_i)$ for which $\|g_i(T_i)-g_b\|_{C^{2,\alpha}_\tau} = \varepsilon$. 
	
	We start by proving the following positive uniform lower bound $\lambda_{\operatorname{ALE}}(g_i(T_i))$: 	
	\begin{claim}\label{Claim-pos-inf-lam}
	\begin{equation}
	0<C^{-1}\leq\liminf_{i\rightarrow+\infty}\lambda_{\operatorname{ALE}}(g_i(T_i))\leq\limsup_{i\rightarrow+\infty}\lambda_{\operatorname{ALE}}(g_i(T_i)) \leq C,\label{Claim-pos-inf-lam-est}
	\end{equation}
	for some positive constant $C=C(n,g_b,\varepsilon,\theta)$.
	\end{claim}
	
	\begin{proof}[Proof of Claim \ref{Claim-pos-inf-lam}]
	The upper bound in (\ref{Claim-pos-inf-lam-est}) is due to the continuity of the functional $\lambda_{\operatorname{ALE}}$ on a $C^{2,\alpha}_{\tau}$-neighborhood of $g_b$ for $\tau\in\left(\frac{n}{2}-1,n-2\right)$ and $\alpha\in(0,1)$ together with the fact that $g_i(T_i)\in B_{C^{2,\alpha}_{\tau}}(g_b,\varepsilon)$.
	
	The lower bound is more subtle to prove and relies on a priori weighted estimates established in Section \ref{sec-wei-est} that we now explain. The proof is very similar to that of Claim \ref{claim-Ric-C-2-alpha} in the proof of Proposition \ref{claim-T-max-closed}. 	
	
	First, let us check that the assumptions of Lemma \ref{a-priori-wei-lemma-Holder-met} are satisfied. Since $T_i>T(n,g_b)$ by Lemma \ref{instabdefT}, 
the covariant derivatives of $\Rm(g(T(n,g_b)/2))$ are uniformly bounded in space by Proposition \ref{prop-a-priori-C0-est-rm-ric}. Since $g_i(t)\in B_{C^{2,\alpha}_{\tau}}(g_b,\varepsilon)$ for all $t\in[0,T_i]$ by assumption, $\sup_{t\in[0,T_i]}\|\Rm(g_i(t))\|_{C^0}\leq C(n,g_b,\varepsilon)$. Therefore, Lemma \ref{lemma-shi-global} applies and guarantees that uniform-in-time bounds on the covariant derivatives of the curvature tensor hold true for $t\geq \frac{T(n,g_b)}{2}$.

Let us check condition (\ref{hyp-ad-hoc-ric-2}). Since $g_i(t)\in B_{C^{2,\alpha}_{\tau}}(g_b,\varepsilon)$ for all $t\in[0,T_i]$ then if $r^2:=\frac{T(n,g_b)}{2}$, it is straightforward to check that $\int_{t-r^2}^t\|\R_{g_i(s)}\|_{C^0}\,ds\leq \frac{1}{2}$ by considering $\varepsilon$ small enough. This fact lets us to use Proposition \ref{coro-A-priori-L^2-C^0-est-RF} to get for $0<r^2:=\frac{T(n,g_b)}{2}<t$:
\begin{equation}
\begin{split}
\|\Ric(g_i(t))\|_{C^0}\leq\,& C(n,g_b,\varepsilon)\exp\left(c(n)\int_{t-r^2}^t\|\Rm(g_i(s))\|_{C^0}\,ds\right)\|\Ric(g_i(t-r^2)\|_{L^2}\\
\leq\,&C(n,g_b,\varepsilon)\exp\left(C(n,g_b,\varepsilon)r^2\right)\|\Ric(g_i(t-r^2)\|_{L^2}\\
\leq\,&C(n,g_b,\varepsilon)\|\Ric(g_i(t-r^2)\|_{L^2}.\label{intermed-c0-L2-claim-unst}
\end{split}
\end{equation}
Now, (\ref{intermed-c0-L2-claim-unst}) and [(\ref{est-l2-ric-hess-eta-a-priori-insta}), Lemma \ref{instability-ALE}] with $s:=0<\frac{T(n,g_b)}{2}\leq t\leq T_i$ and $\eta=0$ imply:
\begin{equation}
\begin{split}
\int_0^{T_i}\|\Ric(g_i(t'))\|_{C^0}\,dt'\leq\,&\int_0^{\frac{T(n,g_b)}{2}}\|\Ric(g_i(t'))\|_{C^0}\,dt'+C(n,g_b,\varepsilon,\theta)\lambda_{\operatorname{ALE}}(g_i(T_i))^{\frac{\theta}{2}}\\
\leq\,&C(n,g_b)\|\Ric(g_i(0))\|_{C^0}+C(n,g_b,\varepsilon,\theta)\lambda_{\operatorname{ALE}}(g_i(T_i))^{\frac{\theta}{2}}\\
\leq\,&C(n,g_b,\varepsilon,\theta)\left(\delta_i+\lambda_{\operatorname{ALE}}(g_i(T_i))^{\frac{\theta}{2}}\right),\label{intermed-c0-int-t-claim-unst}
\end{split}
\end{equation}
where we have used Proposition \ref{prop-a-priori-L2-est-rm-ric} in the second inequality and where $\delta_i:=\|g_i(0)-g_b\|_{C^{2,\alpha}_{\tau}}$. In particular, if $ \varepsilon$ is chosen small enough, (\ref{intermed-c0-int-t-claim-unst}) ensures that 
[(\ref{hyp-ad-hoc-ric-2}), Lemma \ref{a-priori-wei-lemma-Holder-met}] holds true.

	By Lemma \ref{a-priori-wei-lemma-Holder-met} and using [(\ref{est-l2-ric-hess-eta-a-priori-insta}), Lemma \ref{instability-ALE}] as in the proof of (\ref{intermed-c0-int-t-claim-unst}), we have for any $\eta\in\left(0,\frac{\theta}{2-\theta}\right)$,
	\begin{equation}
	\begin{split}\label{est-dist-g-i-T-i-g-i}
	\|g_i(T_i)-g_i\|_{C^{2,\alpha}_{\tau}} \leq\,& C\left(\|\Ric(g_i)\|_{C^{0,\alpha}_{\tau+2}}+ \int_0^{T_i}\|\Ric(g_i(t'))\|_{C^0}^{1-\eta}\,dt'\right)	\\
	\leq\,&C\left(\|g_i-g_b\|_{C^{2,\alpha}_{\tau}}+ \int_0^{T_i}\|\Ric(g_i(t'))\|_{C^0}^{1-\eta}\,dt'\right)\\
	\leq\,&C\left(\delta_i+ \delta_i^{1-\eta}+\lambda_{\operatorname{ALE}}(g_i(T_i))^{\frac{\theta}{2}(1+\eta)-\eta}\right)\\
	\leq \,&C\left(\delta_i^{1-\eta}+\lambda_{\operatorname{ALE}}(g_i(T_i))^{\frac{\theta}{2}(1+\eta)-\eta}\right),
	\end{split}
	\end{equation}
	for some time-independent positive constant $C=C(n,g_b,\varepsilon,\theta,\eta)$.
	By the triangular inequality together with (\ref{est-dist-g-i-T-i-g-i}) and the very definition of $T_i$,
	\begin{equation*}
	\begin{split}
\varepsilon=\|g_i(T_i)-g_b\|_{C^{2,\alpha}_{\tau}}\leq\,&\liminf_{i\rightarrow+\infty}(\|g_i(T_i)-g_i\|_{C^{2,\alpha}_{\tau}}+\delta_i)\\
\leq\,&\liminf_{i\rightarrow+\infty}C\left(\delta_i^{1-\eta}+\lambda_{\operatorname{ALE}}(g_i(T_i))^{\frac{\theta}{2}(1+\eta)-\eta}\right)\\
=\,&C\liminf_{i\rightarrow+\infty}\lambda_{\operatorname{ALE}}(g_i(T_i))^{\frac{\theta}{2}(1+\eta)-\eta}.
\end{split}
\end{equation*}
This ends the proof of Claim \ref{Claim-pos-inf-lam}.
		\end{proof}
	We continue by proving the following 
	\begin{claim}\label{claim-T-unbded}
	The sequence $(T_i)_i$ is unbounded.
	\end{claim}
	\begin{proof}[Proof of Claim \ref{claim-T-unbded}]
	Assume on the contrary that there exists a bounded subsequence, still denoted by $(T_i)_i$, which, without loss of generality, is monotone non-increasing and converges to a finite positive number $T_{\infty}$. In particular, this implies that the Ricci flows $(g_i(t))_{t}$ are well-defined on $[0,T_{\infty}]$. Since $g_i(0)=g_i$ is assumed to converge to $g_b$ in the $C^{2,\alpha}_{\tau}$-topology and since $g_i(t)\in B_{C^{2,\alpha}_{\tau}}(g_b,\varepsilon)$ for any $t\in[0,T_{\infty}]$, Hamilton's compactness theorem \cite{Ham-Com-Thm} ensures that $g_i(t)$ converges locally smoothly to $g_b$ as $i$ tends to $+\infty$ on $N\times [0,T_{\infty}]$. Moreover, the property that $g_i(t)\in B_{C^{2,\alpha}_{\tau}}(g_b,\varepsilon)$ for all indices $i\geq 0$ together with Lemma \ref{lemme-Chal-Cho-Bru} ensure that the convergence of $g_i(t)$, $t\in[0,T_{\infty}]$, to $g_b$ holds in the $C^{2,\alpha'}_{\tau'}$-topology for any $\tau'\in(0,\tau)$ and $\alpha'\in(0,\alpha)$.
	
	Now, for the same reasons, since $(g_i(T_i))_i\subset B_{C^{2,\alpha}_{\tau}}(g_b,\varepsilon)$, Lemma \ref{lemme-Chal-Cho-Bru} implies that there is a subsequence, still denoted by $(g_i(T_i))_i$, that converges in the $C^{2,\alpha'}_{\tau'}$-topology, for $\tau'<\tau$ and $\alpha'\in(0,\alpha)$ to a metric $g_{\infty}\in C^{2,\alpha}_{\tau}$. By continuity of the functional $\lambda_{\operatorname{ALE}}$ on a $C^{2,\alpha'}_{\tau'}$-neighborhood of $g_b$ for $\tau'\in\left(\frac{n}{2}-1,\tau\right)$ and $\alpha'\in(0,\alpha)$, Claim \ref{Claim-pos-inf-lam} leads in particular to 
	\begin{equation}
	\lambda_{\operatorname{ALE}}(g_{\infty})>0.\label{lam-pos-contradiction}
	\end{equation}
	On the other hand, by Proposition \ref{prop-a-priori-L2-est-rm-ric} applied between times $\frac{T_i}{2}$ and $T_i$,
	\begin{equation}
	\begin{split}
\left\|g_i(T_i)-g_i\left(\frac{T_i}{2}\right)\right\|_{C^0}\leq\,& 2\int_{\frac{T_i}{2}}^{T_i}\|\Ric(g_i(t))\|_{C^0}\,dt\\
\leq\,&Ce^{CT_i}\left\|\Ric\left(g_i\left(\frac{T_i}{2}\right)\right)\right\|_{C^0}\to 0,\quad \text{as $i\to+\infty$}.\label{est-T-i-T-i-2}
\end{split}
\end{equation}
The limit in the last line is justified by $\frac{T_i}{2}<T_{\infty}$ and the convergence of $g_i(t)$, $t\in[0,T_{\infty}]$, to $g_b$ in the $C^0$ topology. In particular, by the triangular inequality and the previous estimate (\ref{est-T-i-T-i-2}), $\lim_{i\rightarrow+\infty}\|g_i(T_i)-g_b\|_{C^0}=0$ which identifies $g_{\infty}$ with $g_b$ by uniqueness of the limit: this leads to a contradiction with (\ref{lam-pos-contradiction}).  
	
	\end{proof}
	
	Consider now the Ricci flow $\tilde{g}_i(t):= g_i(t+T_i)$ obtained by translating time. It is defined on $[-T_i,0]$ and satisfies
	\begin{itemize}
		\item for $-t_i\leq t\leq 0$,
		\begin{equation}
		\|\tilde{g}_i(t)-g_b\|_{C^{2,\alpha}_\tau}\leq \varepsilon,\label{closeness-tilde-g-g-b}
		\end{equation}
		\item there exists a positive constant $C=C(n,g_b,\varepsilon,\theta)$ such that for all indices $i\geq 0$,
		\begin{equation}
		0<C^{-1}\leq\lambda_{\operatorname{ALE}}(\tilde{g}_i(0))\leq C,\label{low-bd-unif-lam-unstable}
		\end{equation}
		\item  and
		\begin{equation}
		\lim_{i\rightarrow +\infty}\|\tilde{g}_i(-T_i)-g_b\|_{C^{2,\alpha}_\tau}=0.\label{lim-tilde-g-init-time}
		\end{equation}
	\end{itemize}
	After passing to a subsequence if necessary, Hamilton's compactness theorem \cite{Ham-Com-Thm} ensures that 
	$(\tilde{g}_i(t))_{t\in[-T_i,0]}$ converges smoothly on compact time intervals to an ancient Ricci flow $\tilde{g}_\infty$ defined on $(-\infty,0]$. Moreover, the property (\ref{closeness-tilde-g-g-b}) together with Lemma \ref{lemme-Chal-Cho-Bru} imply that the convergence holds in the $C^{2,\alpha'}_{\tau'}$-topology for any $\tau'\in(0,\tau)$ and $\alpha'\in(0,\alpha)$. The continuity of the functional $\lambda_{\operatorname{ALE}}$ on a $C^{2,\alpha'}_{\tau'}$-neighborhood of $g_b$ for $\tau'\in\left(\frac{n}{2}-1,\tau\right)$ together with the lower bound (\ref{low-bd-unif-lam-unstable}) imply $\lambda_{\operatorname{ALE}}(\tilde{g}_{\infty}(0))>C>0$. This fact alone shows that the solution $(g_{\infty}(t))_{t\leq 0}$ is therefore nontrivial, i.e. non Ricci-flat. 
	
	There remains to prove that $\tilde{g}_\infty(t)$ converges to $g_b$ as $t$ tends to $-\infty$ in the $C^{2,\alpha}_{\tau}$-topology.
	
	According to the triangular inequality, 
	\begin{equation}
	\begin{split}\label{tri-ang-tilde-g-i-g-b}
	    \|\tilde{g}_i(t)-g_b\|_{C^{2,\alpha}_{\tau}}\leq\,& \|\tilde{g}_i(t)-\tilde{g}_i(-T_i)\|_{C^{2,\alpha}_{\tau}}  +\|\tilde{g}_i(-T_i)-g_b\|_{C^{2,\alpha}_{\tau}}\\
	    \leq\,&  C\left(\|\tilde{g}_i(-T_i)-g_b\|^{1-\eta}_{C^{2,\alpha}_{\tau}}+ \lambda_{\operatorname{ALE}}(\tilde{g}_i(t))^{\frac{\theta}{2}(1+\eta)-\eta}\right) +\|\tilde{g}_i(-T_i)-g_b\|_{C^{2,\alpha}_\tau}\\
	    \leq\,&C\left(\|\tilde{g}_i(-T_i)-g_b\|^{1-\eta}_{C^{2,\alpha}_{\tau}}+ \lambda_{\operatorname{ALE}}(\tilde{g}_i(t))^{\frac{\theta}{2}(1+\eta)-\eta}\right),
	    \end{split}
	\end{equation}
	if $\eta\in\left(0,\frac{\theta}{2-\theta}\right)$, for some time-independent positive constant $C=C(n,g_b,\varepsilon,\theta,\eta)$. Here we have essentially used Lemma \ref{a-priori-wei-lemma-Holder-met} together with [(\ref{est-l2-ric-hess-eta-a-priori-insta}), Lemma \ref{instability-ALE}] in the second line.

	Now, according to [(\ref{ineg lambdaALE positif}), Lemma \ref{instabdefT}] applied between times $T_i+t$ and $T_i$ with $-T_i\leq t\leq 0$, 
	\begin{equation*}
  \frac{\lambda_{\operatorname{ALE}}(\tilde{g}_i(t))}{(1 - C\lambda_{\operatorname{ALE}}^{1-\theta}(\tilde{g}_i(t))|t|)^\frac{1}{1-\theta}}\leq  \lambda_{\operatorname{ALE}}(\tilde{g}_i(0)),
\end{equation*}
which leads to:
\begin{equation}
\begin{split}\label{dec-time-lam-ALE-unst}
\lambda_{\operatorname{ALE}}(\tilde{g}_i(t))\leq\,& \frac{\lambda_{\operatorname{ALE}}(\tilde{g}_i(0))}{(1+C\lambda_{\operatorname{ALE}}^{1-\theta}(\tilde{g}_i(0))|t|)^{\frac{1}{1-\theta}}}\\
\leq\,& \frac{C}{(1+C|t|)^{\frac{1}{1-\theta}}},
\end{split}
\end{equation}
if $-T_i\leq t\leq 0$, for some time-independent positive constant $C=C(n,g_b,\varepsilon,\theta)$ which may vary from line to line. Here, we have used the uniform upper bound in (\ref{low-bd-unif-lam-unstable}) in the second line.

Therefore, if $-T_i\leq t\leq0$, estimates (\ref{tri-ang-tilde-g-i-g-b}) together with (\ref{dec-time-lam-ALE-unst}) lead to:
\begin{equation}
\limsup_{i\rightarrow+\infty}\|\tilde{g}_i(t)-g_b\|_{C^{2,\alpha}_{\tau}}\leq \frac{C}{(1+C|t|)^{\frac{\theta}{2(1-\theta)}-\eta}},\label{first-step-conv-anc-sol}
\end{equation}
for any positive $\eta$ sufficient small and where $C=C(n,g_b,\varepsilon,\theta,\eta)$>0. Since $(\tilde{g}_i(t))_{t\in[-T_i,0]}$ converges to $(g_{\infty}(t))_{t\leq 0}$ smoothly locally on compact time intervals, (\ref{first-step-conv-anc-sol}) implies the expected result:
\begin{equation*}
\|g_{\infty}(t)-g_b\|_{C^{2,\alpha}_{\tau}}\leq \frac{C}{(1+C|t|)^{\frac{\theta}{2(1-\theta)}-\eta}},\quad t\leq 0.
\end{equation*}

The fact that $(g_{\infty}(t))_{t\leq 0}$ satisfies (\ref{ALE-unif-unst}) is a direct application of Shi's estimates for ancient solutions to the Ricci flow: see \cite[Chapter $6$]{Cho-Boo} for instance.


\end{proof}

\newpage

\appendix

\section{First and second variations of geometric quantities}\label{section appendix}
In this appendix, we collect first and second derivatives of various geometric quantities.
We start by recalling the first variations of the Ricci and scalar curvature:
	\begin{lemma}\label{lem-lin-equ-Ric-first-var}
		Let $(N^n,g)$ be a Riemannian manifold. Let $h\in S^2T^*N$ be a smooth symmetric $2$-tensor on $N$. Then,
		\begin{equation}
		\begin{split}
		\delta_{g}(-2\Ric)(h)&=\Delta_{g}h+2\Rm(g)(h)-\Ric(g)\circ h-h\circ \Ric(g)-\Li_{B_{g}(h)}\\
		&=\Delta_{L,g}h-\Li_{B_{g}(h)},
		\end{split}
		\end{equation}
		where $\Delta_{L,g}$ denotes the Lichnerowicz operator as defined in introduced in Definition \ref{defn-Lic-Op} and where $B_{g}(h)$ denotes the linearized Bianchi gauge defined by:\begin{equation}
		B_{g}(h):=\div_{g}h-\frac{1}{2}g\left(\nabla^{g}\tr_{g}h,\cdot\right).\label{defn-bianchi-op}
		\end{equation}
		In particular, the first variation of the scalar curvature along a variation $h\in S^2T^*N$ is:
		\begin{equation}
		\begin{split}
		\delta_{g}\R(h)&=\div_{g}\div_{g}h-\Delta_{g}\tr_gh-\left<h,\Ric(g)\right>_g\label{lem-lin-equ-scal-first-var}
		\end{split}
		\end{equation}
\end{lemma}
A proof of this lemma can be found for instance in \cite[Chapter $2$]{Cho-Boo}.

We state without proof distortion bounds on distances along the Ricci flow which are a straightforward consequence of the Ricci flow equation:
\begin{prop}\label{prop-A-priori-dist-distance-RF}
Let $(N^n,g(t))_{t\in[0,T)}$ be a solution to the Ricci flow. Then, for any $0\leq s\leq t<T$, and $x,y\in N$,
\begin{equation}
e^{-\int_s^{t}\|\Ric(g(t'))\|_{C^0}\,dt'}d_{g(t)}(x,y)\leq d_{g(s)}(x,y)\leq e^{\int_s^{t}\|\Ric(g(t'))\|_{C^0}\,dt'}d_{g(t)}(x,y).\label{dist-bds-gal}
\end{equation}
\end{prop}

We continue by recalling the evolution equation satisfied by the Ricci tensor along the Ricci flow:

\begin{lemma}\label{lemma-evo-eqn-Ric-Rm}
Let $(N^n,g(t))_{t\in[0,T)}$ be a smooth Ricci flow. Then, on $N\times (0,T)$,
\begin{eqnarray}
\partial_t\Rm(g(t))&=&\Delta_{g(t)}\Rm(g(t))+\Rm(g(t))\ast\Rm(g(t)),\label{evo-eqn-Rm-for}\\
\partial_t\Ric(g(t))&=&\Delta_{L,g(t)}\Ric(g(t)),\label{evo-eqn-Ric-for}\\
\partial_t\R_{g(t)}&=&\Delta_{g(t)}\R_{g(t)}+2|\Ric(g(t))|^2_{g(t)}.\label{evo-eqn-scal-for}
\end{eqnarray}
\end{lemma}

We next give a global version of Shi's estimates when the initial metric is smooth:
\begin{lemma}\label{lemma-shi-global}
Let $(N^n,g(t))_{t\in[0,T)}$ be a complete smooth Ricci flow such that 
\begin{equation}
\sup_N|\nabla^{g(0),k}\Rm(g(0))|_{g(0)}\leq C_k,
\end{equation}
for any $k\geq 0$, where $C_k$ is a positive constant. If $\tilde{C}_0:=\sup_{t\in[0,T)}|\Rm(g(t))|_{g(t)}$ is finite then 
\begin{equation}
\sup_N|\nabla^{g(t),k}\Rm(g(t))|_{g(t)}\leq \tilde{C}_k,\quad t\in[0,T),
\end{equation}
 for some time-independent positive constant $\tilde{C}_k=\tilde{C}_k(n,\tilde{C}_0,(C_i)_{1\leq i\leq k})$.
\end{lemma}

\begin{proof}
This is essentially due to the maximum principle applied to the evolution equation satisfied by the $k$th-covariant derivatives of the curvature tensor which is derived in \cite[Theorem $13.2$]{ham-3d-pos}. Indeed, let us recall the proof if $k=1$, the cases $k\geq 2$ can be handled similarly. The covariant derivative $\nabla^{g(t)}\Rm(g(t))$ satisfies schematically once we differentiate [(\ref{evo-eqn-Rm-for}), Lemma \ref{lemma-evo-eqn-Ric-Rm}] :
\begin{equation*}
\partial_t\nabla^{g(t)}\Rm(g(t))=\Delta_{g(t)}\nabla^{g(t)}\Rm(g(t))+\nabla^{g(t)}\Rm(g(t))\ast \Rm(g(t)).
\end{equation*}
Notice that:
\begin{equation*}
\begin{split}
\partial_t|\Rm(g(t))|_{g(t)}^2\leq\,& \Delta_{g(t)}|\Rm(g(t))|_{g(t)}^2-2|\nabla^{g(t)}\Rm(g(t))|_{g(t)}^2\\
&+c(n)|\Rm(g(t))|_{g(t)}|\Rm(g(t))|_{g(t)}^2,\\
\partial_t|\nabla^{g(t)}\Rm(g(t))|_{g(t)}^2\leq\,& \Delta_{g(t)}|\nabla^{g(t)}\Rm(g(t))|_{g(t)}^2-2|\nabla^{g(t),2}\Rm(g(t))|_{g(t)}^2\\
&+c(n)|\Rm(g(t))|_{g(t)}|\nabla^{g(t)}\Rm(g(t))|_{g(t)}^2.
\end{split}
\end{equation*}

Then, based on Shi's approach to Bernstein's estimates \cite{Shi-Def}, one considers the function $F:=a|\Rm(g(t))|_{g(t)}^2+|\nabla^{g(t)}\Rm(g(t))|_{g(t)}^2$ for some positive constant $a$ to be specified later which satisfies:
\begin{equation}
\begin{split}\label{evo-eqn-F}
\partial_tF\leq\,& \Delta_{g(t)}F+(c(n)\tilde{C}_0-2a)|\nabla^{g(t)}\Rm(g(t))|_{g(t)}^2+ac(n)\tilde{C}_0^3\\
\leq\,&\Delta_{g(t)}F-c(n)\tilde{C}_0|\nabla^{g(t)}\Rm(g(t))|_{g(t)}^2+ac(n)\tilde{C}_0^3\\
\leq\,&\Delta_{g(t)}F-c(n)\tilde{C}_0F+2ac(n)\tilde{C}_0^3,
\end{split}
\end{equation}
if $a:=c(n)\tilde{C}_0$. Notice that in case $N$ is closed, the maximum principle applied to (\ref{evo-eqn-F}) gives the expected result. If $(N^n,g(t))_{t\in[0,T]}$ is non-compact there exists a smooth positive exhaustion function $\psi$ satisfying $\Delta_{g(t)}\psi\leq C(n,\tilde{C}_0,T)$ and $c_1(1+d_{g(0)}(p,x))\leq \psi(x)\leq c_1(1+d_{g(0)}(p,x))$ for all $x\in N$, for some point $p\in N$ and some positive constants $c_1$, $c_2$.
The existence of such a function is due to \cite[Lemmata $4.2$, $4.3$]{Shi-Def-Pos}.

Now, for $\gamma>0$, consider the function $F-\gamma \psi$ and observe that:
\begin{equation}
\left(\partial_t-\Delta_{g(t)}\right)(F-\gamma\psi)\leq -c(n)\tilde{C}_0(F-\gamma\psi)+c(n)\tilde{C}_0^4+\gamma C(n,\tilde{C}_0,T).\label{evo-eqn-F-gamma}
\end{equation}
Since for each $t\in[0,T)$ and each $\gamma>0$, $F(\cdot,t)-\gamma\psi(\cdot,t)$ is unbounded from below so that it attains its maximum on $N$. Moreover, notice that for every $\gamma>0$, 
\begin{equation}
\sup_N(F(\cdot,0)-\gamma\psi(\cdot,0))\leq \sup_NF(\cdot,0)\leq C(n,\tilde{C}_0,C_1).\label{control-t-0-F-gamma}
\end{equation}
If $F-\gamma\psi$ attains its maximum at an interior point $(t_0,x_0)\in (0,T)\in N$ then the maximum principle applied to (\ref{evo-eqn-F-gamma}) implies that 
\begin{equation}
F(x_0,t_0)-\gamma\psi(x_0,t_0)\leq c(n)\tilde{C}_0^3+\gamma C(n,\tilde{C}_0,T),
\end{equation}
where $c(n)$ and $C(n,\tilde{C}_0,T)$ are constants that may vary from line to line. This fact together with (\ref{control-t-0-F-gamma}) implies for every $(x,t)\in N\times[0,T)$:
\begin{equation*}
\begin{split}
F(x,t)\leq\, &\gamma\psi(x,t)+\max\left\{F(x_0,t_0)-\gamma\psi(x_0,t_0),\sup_NF(\cdot,0)-\gamma\psi(\cdot,0)\right\}\\
\leq\,& \gamma\psi(x,t)+\max\left\{c(n)\tilde{C}_0^3+\gamma C(n,\tilde{C}_0,T),C(n,\tilde{C}_0,C_1)\right\}.
\end{split}
\end{equation*}
By sending $\gamma$ to $0$, one gets the expected result, i.e. $\sup_N|\nabla^{g(t)}\Rm(g(t))|_{g(t)}\leq C(n,\tilde{C}_0,C_1).$ for all $t\in[0,T)$.

\end{proof}

We recall classical Shi's estimates on the curvature tensor and state the corresponding local Shi's estimate on the Ricci tensor.
\begin{prop}\label{prop-a-priori-C0-est-rm-ric}
Let $(N^n,g(0))$ be a complete Riemannian metric with bounded curvature, i.e. $|\Rm(g(0))|\leq C_0$ on $N$. Let $(g(t))_{t\in[0,T_{\operatorname{Shi}}]}$ be Shi's solution to the Ricci flow starting from $g(0)$ where $T_{\operatorname{Shi}}=T(n,C_0)$.
Then, for $t\in[0,T_{\operatorname{Shi}}]$,
\begin{equation}
\|\Rm(g(t))\|_{C^0}\leq2\|\Rm(g(0))\|_{C^0},\quad t^{\frac{k}{2}}\|\nabla^{g(t),k}\Rm(g(t))\|_{C^0}\leq C_k,\quad k\geq 0,\label{shi-est-curv-tensor}\\
\end{equation}
where $C_k=C(n,k,C_0)$.

Moreover, the following local Shi type estimates hold true: for each $k\geq 0$, there exist positive constants $C_k=C(n,k,C_0)$ and $r_0=r(n,\inj_{g(0)}(N))$ such that for $r<r_0$ and $x\in N$,
\begin{equation}
\sup_{B_{g(0)}(x,r)}t^{\frac{k}{2}}|\nabla^{g(t),k}\Ric(g(t))|_{g(t)}\leq C_k \sup_{B_{g(0)}(x,2r)}|\Ric(g(0))|_{g(0)},\quad t\in[0,T_{\operatorname{Shi}}].\label{shi-est-ric-cov-tensor}
\end{equation}
\end{prop}

\begin{proof}
Estimates (\ref{shi-est-curv-tensor}) are due to Shi \cite[Lemma $7.1$]{Shi-Def}.

Estimates (\ref{shi-est-ric-cov-tensor}) can be proved along the same lines by using Shi's estimates (\ref{shi-est-curv-tensor}) on the curvature tensor inductively.
\end{proof}
The following proposition establishes a priori rough $C^0$ and $L^2$ estimates on the Ricci curvature along a solution to the Ricci flow. 
\begin{prop}\label{prop-a-priori-L2-est-rm-ric}
Let $(N^n,g(t))_{t\in[0,T)}$ be a complete Ricci flow with bounded curvature, i.e. $\sup_N|\Rm(g(t))|<+\infty$ for each $t\in[0,T)$. 

Then for $0\leq s\leq t\leq T$,
\begin{equation}
\|\Ric(g(t))\|_{C^0}\leq e^{c(n)\int_s^t\|\Rm(g(t'))\|_{C^0}\,dt'}\|\Ric(g(s))\|_{C^0}.\label{shi-est-ric-tensor}
\end{equation}

Furthermore, if $\Ric(g(t))\in L^2$ for each $t\in [0,T)$ then, for $0\leq s\leq t<T$,
\begin{equation}
\|\Ric(g(t))\|_{L^2}\leq e^{c(n)\int_s^t\|\Rm(g(t'))\|_{C^0}\,dt'}\|\Ric(g(s))\|_{L^2}.\label{shi-est-ric-tensor-l2}
\end{equation}
\end{prop}
\begin{proof}
Estimate (\ref{shi-est-ric-tensor}) is a straightforward application of the maximum principle.

Indeed, using the evolution equation [(\ref{evo-eqn-Ric-for}), Lemma \ref{lemma-evo-eqn-Ric-Rm}] satisfied by the Ricci tensor together with the previous bound on the curvature tensor, one gets: 
\begin{equation}
\begin{split}
\partial_t|\Ric(g(t))|_{g(t)}^2\leq& \Delta_{g(t)}|\Ric(g(t))|^2_{g(t)}+c(n)|\Rm(g(t))|_{g(t)}|\Ric(g(t))|^2_{g(t)}\\
\leq&\Delta_{g(t)}|\Ric(g(t))|^2_{g(t)}+c(n)\sup_{t'\in[s,t]}\|\Rm(g(t'))\|_{C^0}|\Ric(g(t))|^2_{g(t)}.\label{subsol-ric-eqn}
\end{split}
\end{equation}
In particular, the function $e^{-c(n)\int_s^t\|\Rm(g(t'))\|_{C^0}\,dt'}|\Ric(g(t))|^2_{g(t)}$ is a subsolution of the heat equation along the Ricci flow on $N\times[0,T)$. The maximum principle implies that $$\|\Ric(g(t))\|_{C^0}\leq e^{c(n)\int_s^t\|\Rm(g(t'))\|_{C^0}\,dt'} \|\Ric(g(s))\|_{C^0},$$ for any $0\leq s\leq t\leq T$. 

If $\Ric(g(t))\in L^2$, $t\in[0,T]$, the proof is the $L^2$ counterpart of the proof of the $C^0$ estimate (\ref{shi-est-ric-tensor}). By differentiating in time and using the evolution equation [(\ref{evo-eqn-Ric-for}), Lemma \ref{lemma-evo-eqn-Ric-Rm}] satisfied by the Ricci tensor:
\begin{equation*}
\begin{split}
\frac{d}{dt}\|\Ric(g(t))\|^2_{L^2}=\,&\int_N2\left<\Ric(g(t)),\Delta_{g(t)}\Ric(g(t))+\Rm(g(t))\ast\Ric(g(t))\right>_{g(t)}\,d\mu_{g(t)}\\
&-\int_N\R_{g(t)}|\Ric(g(t))|^2_{g(t)}\\
\leq\,&-2\|\nabla^{g(t)}\Ric(g(t))\|^2_{L^2}+c(n)\|\Rm(g(t))\|_{C^0}\|\Ric(g(t))\|_{L^2}^2\\
\leq\,&c(n)\|\Rm(g(t))\|_{C^0}\|\Ric(g(t))\|_{L^2}^2.
\end{split}
\end{equation*}
Here we have used integration by parts in the second line which can be justified by using suitable cut-off functions. The expected estimate follows by Gr\"onwall's inequality applied to the previous differential inequality. 
\end{proof}

We continue by establishing a general pointwise formula linking the laplacian along the Ricci flow of the difference of the solution at two different time:
\begin{lemma}\label{lemma-beauty-int-lap-met-ric}
Let $(N^n,g(t))_{t\in[0,T)}$ be a smooth Ricci flow. Then, for $0\leq s\leq t<T$, the following formula holds pointwise if $h(t'):=g(t')-g(s)$, $t'\in[s,t]$:
\begin{equation}
\begin{split}\label{so-beautiful-int-der-lap-lemma}
\Delta_{g(t)}(g(t)-g(s))=\,&\left[2\Ric(g(t'))-\Ric(g(t'))\circ h(t')-h(t')\circ\Ric(g(t'))\right]_{t'=s}^{t}\\
&+\int_s^t2\Ric(g(t'))\circ\Ric(g(t'))+\Rm(g(t'))\ast\Ric(g(t'))\,dt'\\
&+\int_s^t\nabla^{g(t')}\Ric(g(t'))\ast \nabla^{g(t')}(g(t')-g(s))\,dt'\\
&+\int_s^t\Ric(g(t'))\ast \nabla^{g(t'),2}(g(t')-g(s))\,dt'.
\end{split}
\end{equation}

\end{lemma}

\begin{proof}
Observe first that:
\begin{equation}
\begin{split}\label{beautiful-int-der-lap}
\Delta_{g(t)}(g(t)-g(s))=\int_s^t\frac{d}{dt'}\Delta_{g(t')}(g(t')-g(s))\,ds,
\end{split}
\end{equation}
Now, for $t'\in[0,T)$,
\begin{equation*}
\begin{split}
\frac{d}{dt'}\Delta_{g(t')}(g(t')-g(s))=&\left(\frac{d}{dt'}\Delta_{g(t')}\right)(g(t')-g(s))+\Delta_{g(t')}(\partial_{t'}g(t'))\\
=&-2\Delta_{g(t')}\Ric(g(t'))+\left(\frac{d}{dt'}\Delta_{g(t')}\right)(g(t')-g(s))\\
=&-2\partial_{t'}\Ric(g(t'))+\Rm(g(t'))\ast\Ric(g(t'))\\
&+\sum_{k=0}^2\nabla^{g(t'),2-k}\Ric(g(t'))\ast \nabla^{g(t'),k}(g(t')-g_b).
\end{split}
\end{equation*}
Here, we have used Lemma \ref{lemma-evo-eqn-Ric-Rm} in the third line. As a first conclusion:
\begin{equation}
\begin{split}\label{beautiful-int-der-lap-bis}
\Delta_{g(t)}(g(t)-g(s))=&2\Ric(g(s))-2\Ric(g(t))+\int_s^t\Rm(g(t'))\ast\Ric(g(t'))\,dt'\\
&+\int_s^t\sum_{k=0}^2\nabla^{g(t'),2-k}\Ric(g(t'))\ast \nabla^{g(t'),k}(g(t')-g(s))\,dt'.
\end{split}
\end{equation}
We are half way from the proof of (\ref{so-beautiful-int-der-lap-lemma}): we refine the computation of $\left(\frac{\partial}{\partial t'}\Delta_{g(t')}\right)(g(t')-g(s))$ as follows. Notice that we only need to identify the zeroth order term $(g(t')-g(s))$. 


For doing so, recall the following evolution equation satisfied by the Christoffel symbols along the Ricci flow as given for instance in \cite[Chapter $2$]{Cho-Boo}:
\begin{equation}
\frac{\partial}{\partial t}\Gamma(g(t))_{ij}^k=-g(t)^{kl}\left(\nabla^{g(t)}_{i}\Ric(g(t))_{jl}+\nabla^{g(t)}_{j}\Ric(g(t))_{il}-\nabla^{g(t)}_{l}\Ric(g(t))_{ij}\right).\label{chris-symb-evo}
\end{equation}
In particular, if $h(t):=g(t)-g(s)$, the zeroth order term of $\left(\frac{\partial}{\partial t}\Delta_{g(t)}\right)(g(t)-g(s))$ is:
\begin{equation*}
\begin{split}
g(t)^{ij}\nabla^{g(t)}_{i}\left(\frac{\partial}{\partial t}\nabla^{g(t)}_{j}\right)h(t)_{pq}&=-g(t)^{ij}\nabla^{g(t)}_{j}\left(\partial_t\Gamma(g(t))_{ip}^k\right)h(t)_{kq}-g(t)^{ij}\nabla^{g(t)}_{j}\left(\partial_t\Gamma(g(t))_{iq}^k\right)h(t)_{pk}\\
&=\Delta_{g(t)}\Ric(g(t))_{pk}h(t)_{kq}+h(t)_{pk}\Delta_{g(t)}\Ric(g(t))_{kq}\\
&\quad+\left[g(t)^{ij}\nabla^{g(t)}_{j}\nabla^{g(t)}_{p}\Ric(g(t))_{ik}-g(t)^{ij}\nabla^{g(t)}_{j}\nabla^{g(t)}_{k}\Ric(g(t))_{ip}\right]h(t)_{kq}\\
&\quad+\left[g(t)^{ij}\nabla^{g(t)}_{j}\nabla^{g(t)}_{q}\Ric(g(t))_{ik}-g(t)^{ij}\nabla^{g(t)}_{j}\nabla^{g(t)}_{k}\Ric(g(t))_{iq}\right]h(t)_{pk}\\
&=\left(\Delta_{g(t)}\Ric(g(t))\circ h(t)+h(t)\circ\Delta_{g(t)}\Ric(g(t))\right)_{pq}\\
&\quad+\left[\nabla^{g(t)}_{p}\div_{g(t)}\Ric(g(t))_k-\nabla^{g(t)}_{k}\div_{g(t)}\Ric(g(t))_p\right]h(t)_{kq}\\
&\quad+\left[\nabla^{g(t)}_{q}\div_{g(t)}\Ric(g(t))_k-\nabla^{g(t)}_{k}\div_{g(t)}\Ric(g(t))_q\right]h(t)_{pk}\\
&\quad+\Rm(g(t))\ast \Ric(g(t))\ast h(t)_{pq}\\
&=\left(\partial_t\Ric(g(t))\circ h(t)+h(t)\circ\partial_t\Ric(g(t))\right)_{pq}\\
&\quad+\left(\Rm(g(t))\ast \Ric(g(t))\ast h(t)\right)_{pq}.
\end{split}
\end{equation*}
Here, we have used the traced Bianchi identity in the last line to cancel the terms involving the divergence $\div_{g(t)}\Ric(g(t))$ together with the evolution equation satisfied by $\Ric(g(t))$ given by Lemma \ref{lemma-evo-eqn-Ric-Rm}.
To conclude, we have obtained schematically:
\begin{equation}
\begin{split}\label{new-ref-est-der-laplacian}
\left(\frac{\partial}{\partial t'}\Delta_{g(t')}\right)(g(t')-g(s))=&\partial_{t'}\Ric(g(t'))\circ h(t')+h(t')\circ\partial_{t'}\Ric(g(t'))+\Rm(g(t'))\ast\Ric(g(t'))\\
&+\nabla^{g(t')}\Ric(g(t'))\ast \nabla^{g(t')}(g(t')-g(s))\\
&+\Ric(g(t'))\ast \nabla^{g(t'),2}(g(t')-g(s)).
\end{split}
\end{equation}
With (\ref{new-ref-est-der-laplacian}) in hand, we are in a position to conclude. Indeed, by integrating by parts with respect to time,
\begin{equation}
\begin{split}\label{so-beautiful-int-der-lap-sras}
\Delta_{g(t)}(g(t)-g(s))=&2\Ric(g(s))-\Ric(g(s))\circ h(s)-h(s)\circ\Ric(g(s))\\
&-2\Ric(g(t))+\Ric(g(t))\circ h(t)+h(t)\circ\Ric(g(t))\\
&+\int_s^t2\Ric(g(t'))\circ\Ric(g(t'))+\Rm(g(t'))\ast\Ric(g(t'))\,dt'\\
&+\int_s^t\nabla^{g(t')}\Ric(g(t'))\ast \nabla^{g(t')}(g(t')-g(s))\,dt'\\
&+\int_s^t\Ric(g(t'))\ast \nabla^{g(t'),2}(g(t')-g(s))\,dt',
\end{split}
\end{equation}
as expected.

\end{proof}

We finally recall a special case of the Gagliardo-Nirenberg interpolation inequalities essentially proved by Aubin \cite{Aub-Ine-Int} in this setting that we will make use of:
\begin{lemma}[Interpolation inequalities for asymptotically conical metrics]\label{inter-inequ-gag-nir}
Let $(M^n,g)$ be a complete Riemannian metric. Then, the following interpolation inequalities hold true for any integer $m\geq 0$:
\begin{equation}
\sup_M|\nabla^{g,j}T|_g\leq C(n,\inj_g(x))\sup_M|T|_g^{1-\frac{j}{m}}\cdot \sup_M|\nabla^{g,m}T|_g^{\frac{j}{m}},\label{int-inequ-easy}
\end{equation}
where $T$ is any tensor on $M$ with compact support in $B(x,\inj_g(x)/2)$, $x \in M$ and $0\leq j\leq m$.\\

In particular, if the curvature tensor decays quadratically with derivatives, i.e. if 
\begin{equation*}
A^k(g):=\sup_Mr_p^{2+k}|\nabla^{g,k}\Rm(g)|_g<+\infty,
\end{equation*}
 for all $k\geq 0$ and if $\inj_g(x)\geq \iota r_p(x)$ from some positive constant $\iota$ and $p\in M$ uniform in $x\in M$, then
\begin{equation}
\sup_{B_g(x,\iota r_p(x)/4)}r_p(x)^j|\nabla^{g,j}T|_g\leq C\cdot\sup_{B_g(x,\iota r_p(x)/2)}|T|_g^{1-\frac{j}{m}}\cdot \sup_{B_g(x,\iota r_p(x)/2)}\bigg(\sum_{k=0}^mr_p(x)^k|\nabla^{g,k}T|_g\bigg)^{\frac{j}{m}},
\end{equation}
where $C=C(n,\iota, (A^k(g))_{0\leq k\leq m})$.
\end{lemma}

\newpage

\bibliographystyle{alpha.bst}
\bibliography{bib-4d-RF-stab}

\end{document}